\documentclass[12pt,oneside, reqno]{amsart}
 \usepackage{graphicx}
\usepackage{placeins}
 \usepackage{hhline}
\usepackage{amsmath,amsthm,amscd,amssymb,mathrsfs}
\usepackage{xspace}
\usepackage[all]{xypic}
\usepackage{booktabs} 
\usepackage{physics}
\usepackage{array} 
\newcolumntype{C}{>{$}c<{$}} 
\usepackage{hyperref}
\usepackage[lite,initials]{amsrefs}
\usepackage{verbatim}
\usepackage{amscd} 
\usepackage[all]{xy} 
\usepackage{youngtab} 
\usepackage{ytableau} 
\usepackage{nicefrac}
\usepackage{xfrac}
\usepackage{longtable} 

\usepackage{mathdots}
\usepackage{tikz}
\newcommand{\ccircle}[1]{* + <1ex>[o][F-]{#1}}
\newcommand{\ccirc}[1]{\xymatrix@1{* + <1ex>[o][F-]{#1}}}
\usepackage[T1]{fontenc} 
\usepackage{cleveref} 
\numberwithin{equation}{section} 

\usepackage{color}
\makeatletter
\newtheorem{rep@theorem}{\rep@title}
\newcommand{\newreptheorem}[2]{%
\newenvironment{rep#1}[1]{%
 \def\rep@title{#2 \ref{##1}}%
 \begin{rep@theorem}}%
 {\end{rep@theorem}}}
\makeatother

\newtheorem{theorem}{Theorem}[section]
\newreptheorem{theorem}{Theorem}
\newreptheorem{lemma}{Lemma}

\newtheorem{lemma}[theorem]{Lemma}

\newtheorem{prop}[theorem]{Proposition}

\theoremstyle{definition}
\newtheorem{definition}[theorem]{Definition}
\newtheorem{examplex}{Example}
\newenvironment{example}
 {\pushQED{\qed}\examplex}
 {\popQED\endexamplex}
\theoremstyle{remark}
\newtheorem{remark}[theorem]{Remark}

\newcommand{\defi}[1]{\textsf{#1}} 
\newcommand{\isom}{\cong}

\newcommand{\End}{\operatorname{End}}
\newcommand{\Gr}{\operatorname{Gr}}
\newcommand{\SL}{\operatorname{SL}}
\newcommand{\GL}{\operatorname{GL}}

\newcommand{\F}{\mathcal{F}}

\newcommand{\PP}{\mathbb{P}}
\newcommand{\FF}{\mathbb{F}}

\newcommand{\RR}{\mathbb{R}}
\newcommand{\NN}{\mathbb{N}}
\newcommand{\CC}{\mathbb{C}}
\newcommand{\ZZ}{\mathbb{Z}}
\newcommand{\QQ}{\mathbb{Q}}

\def\bw#1{{\textstyle\bigwedge^{\hspace{-.2em}#1}}}

\def\phi{ \varphi }

\def \a{\alpha}
\def \b{\beta}

\def \g{\mathfrak{g}}
\def \fa{\mathfrak{a}}
\def \e{\mathfrak{e}}

\def \sl{\mathfrak{sl}}

\newcommand{\ad}{\operatorname{ad }}
\newcommand{\Mat}{\operatorname{Mat}}

\newcounter{nameOfYourChoice}

\begin{document}
\date{\today}

\author{Fr\'ed\'eric Holweck}\email{frederic.holweck@utbm.fr}
\address{Laboratoire Interdisciplinaire Carnot de Bourgogne, ICB/UTBM, UMR 6303 CNRS,
Universit\'e Bourgogne Franche-Comt\'e, 90010 Belfort Cedex, France
}
\address{Department of Mathematics and Statistics,
Auburn University,
Auburn, AL, USA
}

\author{Luke Oeding}\email{oeding@auburn.edu}
\address{Department of Mathematics and Statistics,
Auburn University,
Auburn, AL, USA
}

\title[Toward Jordan Decompositions of Tensors]{Toward Jordan Decompositions for Tensors}
\begin{abstract} 
We expand on an idea of Vinberg to take a tensor space and the natural Lie algebra that acts on it and embed their direct sum into an auxiliary algebra. Viewed as endomorphisms of this algebra, we associate adjoint operators to tensors. We show that the group actions on the tensor space and on the adjoint operators are consistent, which means that the invariants of the adjoint operator of a tensor, such as the Jordan decomposition, are invariants of the tensor.  We show that there is an essentially unique algebra structure that preserves the tensor structure and has a meaningful Jordan decomposition.
 We utilize aspects of these adjoint operators to study orbit separation and classification in examples relevant to tensor decomposition and quantum information. 
\end{abstract}
\maketitle

\section{Introduction}
The classical Jordan decomposition of square matrices is the following:
\begin{theorem}[Jordan Canonical Form]
Let $\FF$ be an algebraically closed field. Every $A\in \Mat_{n\times n }(\FF)$ is similar to its Jordan canonical form, which is a decomposition:
\[
A \sim J_{k_1}(\lambda_1) \oplus \cdots \oplus J_{k_d}(\lambda_d) \oplus {\bf 0 }
,\]
where the $k\times k $ Jordan blocks are 
$J_k(\lambda) = 
\left(\begin{smallmatrix} 
\lambda & 1 \\[-1ex]
& \lambda & \small{\ddots} \\[-1ex] 
& & \ddots & 1 \\
& & & \lambda
\end{smallmatrix}\right)
$. 
The algebraic multiplicity of the eigenvalue $\lambda$ is the sum $\sum_{\lambda_j = \lambda} k_j$, and the geometric multiplicity of $\lambda$ is the number $s$ of blocks with eigenvalue $\lambda$, $s = \sum_{\lambda_j = \lambda} 1$. 
\end{theorem}
One may view the JCF as an expression $A = S+N$ with $S$ diagonal (semisimple) and $N$ upper triangular (nilpotent), and the two commute. The JCF answers an orbit-classification problem: The group $\SL(V)$ acts on the vector space $V\otimes V^{*}$, and JCF gives canonical representatives for the orbits. 

Let $V \cong \FF^n$ be a vector space, which naturally carries the action of the special linear group, denoted $\SL(V)$, of invertible linear transformations with determinant 1. 
Let $A \in \Mat_{n\times n} (\FF)$. Then $A$ represents a linear mapping $T_A \colon V \to V$ with respect to the standard basis of $V$, and $T_A$ is viewed as an element of the algebra of endomorphisms $\End(V)$. As a representation of $\SL(V)$, we have $\End(V) = V \otimes_\FF V^*$. Writing this as a tensor product of $\SL(V)$-modules and keeping track of the dual indicates the induced action, conjugation. 

The analogous statements hold for operators $T_A \in \End(V)$, and one has the \defi{Jordan-Chevalley decomposition}, which is the expression $T_A = T_S + T_N$ where $T_S$ is semisimple (diagonalizable), $T_ N$ is nilpotent, and $T_S$ and $T_N$ commute. Recall that a linear operator $T$ is \defi{nilpotent} if any of the following equivalent conditions hold: $T^r = 0$ for some finite positive integer $r$, if $T$ has no non-zero eigenvalues, or if the characteristic polynomial of $T$ is $\chi_T(t) = t^n$. An operator $T_S$ is semisimple if it is diagonalizable, which one can check by seeing if, for each eigenvalue the geometric and algebraic multiplicities agree.

Like many notions for matrices, there are many possible generalizations to higher-order tensors. Our starting point is to note that the vector space $V\otimes V^{*}$ is also a Lie algebra of endomorphisms and is the Lie algebra of the group $\GL(V)$ acting on it. So, in this case, the orbit problem concerns a Lie algebra acting on itself. Since conjugating by scalar multiples of the identity is trivial, we work with $\SL(V)$ instead since it has the same action as $\GL(V)$ in this case.

For tensors, we consider a tensor space $W$, like $U_{1}\otimes U_{2} \otimes U_{3}$ or $\bw{k} V$, the natural group $G$ acting on it, like $\SL(U_{1})\times \SL(U_{2})\times \SL(U_{3})$ or $\SL(V)$ respectively, and the Lie algebra $\mathfrak g$, and attempt to build a graded algebra starting from their direct sum $\fa = \g \oplus W$ with enough properties so that elements of $W$ can be viewed as adjoint operators acting on $\fa$ whose Jordan decomposition is both non-trivial and invariant under the $G$-action. This method, which is inspired by the work of Vinberg, Kac, and others \cites{Kac80, Vinberg-Elashvili, GattiViniberghi} invites comparison to the work of Landsberg and Manivel \cite{LM02}, which constructed all complex simple Lie algebras. Our goal, rather than to classify algebras that arise from this construction, is to explore the construction of this adjoint operator and its Jordan decomposition and the consequences for tensors. In particular, the Jacobi identity is not necessary to compute the adjoint operators or their Jordan forms. We also note the work of Levy \cite{levy2014rationality}, which also studied the generalizations of Vinberg's method and associated Jordan decompositions \cite{GattiViniberghi}. 

We consider the adjoint operator of a tensor $T \in W$ in a particular embedding into an algebra $\fa$. One would like to take the Jordan decomposition of the adjoint operator and pull it back to a decomposition of the tensor, but this process might not be possible, especially since the algebra $\fa$ typically has much larger dimension than $W$. Hence when $\ad(T) = S + N$ as a sum of a semi-simple $S$ and nilpotent $N$,  one can't necessarily expect to be able to find preimages $s, n\in W$ such that $\ad(s) = S$ and $\ad(n) = N$ with $[s,n]=0$. We study this issue in the case of $\sl_6 \oplus \bw{3} \CC^6$ in Example~\ref{ex:g36}, and show that pure semi-simple elements do not exist in $\bw{3} \CC^6$, but we can construct pure semisimple elements that are not concentrated in a single grade of the algebra. To understand this fully would more work beyond the scope of the present article, but it could be a nice future direction to pursue.

Our aim is to set a possible stage where this operation could occur. We note that in the Vinberg cases where the algebra is a Lie algebra, this process actually does pull back to provide Jordan depositions of tensors for these special formats. We will see what happens when we relax the Lie requirement on the auxillary algebra, which has an advantage of being able to be defined for any tensor format, and the disadvantage of not being well-understood in the non-Lie regime. This article provides several initial steps in this direction.

\subsection{Linear algebra review}\label{sec:LA}
Let $V^*$ denote the vector space dual to $V$, that is the set of linear functionals $\{ V\to \FF\}$, which is also the dual $\SL(V)$-module to $V$ with the right action:
\[\begin{matrix}
\SL(V) \times V^* &\to& V^* \\
(g,\alpha) &\mapsto & \alpha g^{-1} 
.\end{matrix}
\]
Then the natural (linear) action of $\SL(V)$ on $V\otimes V^*$ is obtained by defining the action on simple elements and extending by linearity:
Simple elements $V\otimes V^*$ are of the form $v\otimes \a$, so the action is induced from
\[\begin{matrix}
\SL(V) \times V\otimes V^* &\to& V\otimes V^* \\
(g, v\otimes \alpha) &\mapsto & (gv)\otimes( \alpha g^{-1} )
.\end{matrix}
\]
Hence, the natural action of $\SL(V)$ on $\End(V)$ is by conjugation. Since the matrix $A = (a_{ij})$ is obtained from $T_A$ by expanding $T_A$ in the standard basis $\{e_1,\ldots, e_n\}$ of $V$ and dual basis $\{f_1,\ldots, f_n\}$ of $V^*$ via extracting the coefficients in the expression 
\[
T_A = \sum_{ij} a_{ij} e_i \otimes f_j
.\]
So for a vector $x = \sum_i x_i e_i\in V$, $T_A(x) = \sum_{ij} a_{ij} e_i \otimes f_j \sum_k x_ke_k = \sum_{ij,k} a_{ij} x_k e_i \otimes f_j ( e_k) = \sum_{ij} a_{ij} x_j e_i = \sum_i (\sum_{j} a_{ij} x_j) e_i $, i.e. the usual matrix-vector product $Ax$. 

The natural $\SL_n(\FF)$-action on matrices is also conjugation (since scalar multiplication commutes with matrix product and linear operators):
\[
g.T_A = \sum_{ij} a_{ij} g.(e_i \otimes f_j) = \sum_{ij} a_{ij} (ge_i) \otimes (f_j g^{-1}) 
.\]
Evaluating $g.T_A$ on a vector $x$ we see that $(g.T_A)(x) = g(T_A(g^{-1}x))$. Replacing $T_A$ and $g$ respectively with the matrices that represent them with respect to the standard basis, $A$ and $M$ respectively, and using associativity, we obtain $g.T_A(x) = MAM^{-1} x$. So, the standard matrix representative of the coordinate-changed transformation $g.T_A$ is $MAM^{-1}$.

If the field $\FF$ is algebraically closed, then the operator $T_A$ has generalizd eigen-pairs $(\lambda, v) \in \FF\times \FF^n$ such that for some $m\in \NN$ 
\[
(Av - \lambda v)^k = 0, \quad\text{but } (Av - \lambda v)^{k-1} \neq 0.
\]
The subspace of generalized eigenvectors associated to a fixed $\lambda$ is $G$-invariant, and a Jordan chain of linearly independent generalized eigenvectors provides a basis such that the operator is in almost-diagonal form, the Jordan canonical form referenced above.

\begin{remark}
Morozov's theorem \cite{MR0007750} is key to studying nilpotent orbits such as in \cite{Vinberg-Elashvili, Antonyan}. It says that every nilpotent element is part of an $\sl_2$-triple. In the matrix case, and in adapted coordinates these triples consist of a matrix with a single 1 above the diagonal (which may be taken as part of a Jordan block), its transpose, and their commutator, which is on the diagonal. This triple forms a 3-dimensional Lie algebra isomorphic to $\sl_2$. 
\end{remark}

\section{Historical Remarks on Jordan Decomposition for Tensors}\label{sec:history}
Concepts from linear algebra often have several distinct generalizations to the multi-linear or tensor setting. It seems natural that any generalization of Jordan decomposition to tensors should involve the concepts of eigenvalues and the conjugation of operators (simultaneous change of basis of source and target) by a group that respects the tensor structure. 
 In their seminal work that reintroduced hyperdeterminants (one generalization of the determinant to tensors), Gelfand, Kapranov, and Zelevinski \cite{GKZ} wondered what a spectral theory of tensors might look like. This was fleshed out for the first time simultaneously in the works \cite{Lim05_evectors, Qi05_eigen} (see also \cite{qi2018tensor}). One definition of an eigen-pair of a tensor $T \in (\CC^{n})^{\otimes d}$ is a vector $v \in \CC^n$ and a number $\lambda$ such that $T(v^{\otimes d-1}) = \lambda v$. Cartwright and Sturmfels computed the number of eigenvectors of a symmetric tensor \cite{CartwrightSturmfels2011} and Ottaviani gave further computations using Chern classes of vector bundles, see for example \cite{OedOtt13_Waring}. Additional considerations in this vein appeared in \cite{Gnang2011}, which gave a different product, and hence algebra structure, on tensor spaces.
While this concept of eigenvector has been fundamental, it doesn't seem immediately amenable to a Jordan decomposition because of the following features:
\begin{enumerate}
\item The operator $T \colon S^{d-1} V \to V$ is not usually square. 
\item Contracting $T \in V^{\otimes d}$ with the same vector $v\in V$ $d-1$ times has a symmetrization effect, which collapses the group that acts to $\SL(V)$ versus $\SL(V)^{\times d}$.
\item It's not immediately apparent that a conjugation group action on $T$ makes sense because of the (typically) non-square matrices coming from contractions or flattenings.
\end{enumerate}

See \cite{MR3841899} for a multilinear generalization of a Jordan decomposition, which considered 3rd-order tensors and focused on approximating a tensor with one that is block upper triangular. 

The generalization we will offer follows the following line of research. 
Any semisimple Lie algebra has a Jordan decomposition \cite{kostant1973convexity}, and Jordan decomposition is preserved under any representation \cite[Th.~9.20]{FultonHarris}. See also \cite{HuangKim} for recent work in this direction. The existence of such a Jordan decomposition has been the key to answering classical questions like orbit classification. 
Specifically, given a vector space $V$ with a connected linear algebraic group $G \subset \GL(V)$ acting on it, what are all the orbits? Which pairs $(G, V)$ have finitely many orbits \cite{Kac80, Kac85}, or if not, which are of \emph{tame} or \emph{wild} representation type? These questions have intrigued algebraists for a long time. Dynkin \cite{dynkin1960semisimple, dynkin2000selected} noted to call something a \emph{classification} of representations one must have a set of ``characteristics'' satisfying the following:
\begin{itemize}
\item The characteristics must be invariant under inner automorphisms so that the characteristics of equivalent representations must coincide.
\item They should be complete: If two representations have the same characteristics, they must be equivalent.
\item They should be compact and easy to compute. 
\end{itemize}

Indeed, a requirement for finitely many orbits is that the dimension of the group $G$ must be at least that of the vector space $V$. However, for tensors, this is rare \cite{venturelli2019prehomogeneous}. Yet, in some cases, Vinberg's method \cite{Vinberg75} can classify orbits even in the tame setting and, in fact, does so by embedding the question of classifying orbits for the pair $(G, V)$ to classifying nilpotent orbits in the adjoint representation $V'$ of an auxiliary group $G'$. Part of the classification comes down to classifying the subalgebras of the Lie algebra $\g'$, which was done by Dynkin \cite{dynkin1960semisimple, dynkin2000selected}. Another crucial part of this classification is the characteristics, which are computed utilizing Morozov $\sl_2$-triples in the auxiliary algebra $\fa$.
 One can follow these details in \cite{Vinberg-Elashvili}, for instance, in the case of $(\SL(9), \bw{3}\CC^9)$ whose orbit classification relied upon the connection to the adjoint representation of $\e_8$, or for the case of $(\SL(8), \bw{4}\CC^8)$ in \cite{Antonyan}. However, Vinberg and {\`E}la{\v{s}}vili, comment that while they did succeed in classifying the orbits for $(\SL(9), \bw{3}\CC^9)$, which involved 7 families of semisimple elements depending on 4 continuous parameters, and up to as many as 102 nilpotent parts for each, such a classification of orbits for $(\SL(n), \bw{3}\CC^n)$, ``if at all possible, is significantly more complicated.'' One reason for this is that all the orbits for case $n$ are nilpotent for case $n+1$ (being non-concise), hence, for $n\geq 10$ even the nilpotent orbits will depend on parameters. In addition, it is not clear what should come next after the sequence $E_6, E_7, E_8$. We offer an option for the next step by being agnostic to the algebra classification problem, and we just focus on putting tensor spaces into naturally occurring graded algebras, whose product (bracket) is compatible with the group action on the tensor space. 

Though we do not compute a Jordan decomposition many times in this article, we emphasize that all the features we study (like matrix block ranks and eigenvalues) are consequences of the existence of a Jordan decomposition compatible with the tensor structure. Unless otherwise noted, all the computations we report take under a few seconds on a 2020 desktop computer. We implemented these computations in a package in Macaylay2 \cite{M2} called \texttt{ExteriorExtensions} \cite{oeding2023exteriorextensions}. We include this package and example computations in the ancillary files of the arXiv version of this article.

\section{Constructing an algebra extending tensor space}
Now we will generalize a part of the Vinberg construction (discussed briefly in Section~\ref{sec:history}) embedding tensor spaces into a graded algebra, obtaining adjoint operators for tensors such that the group action on the tensor space is consistent with the Jordan decomposition of the operator.
This involves structure tensors of algebras (see \cite{bari2022structure, ye2018fast} for recent studies).

\subsection{A graded algebra extending a $G$-module}\label{sec:requirements}
Let $M$ denote a finite-dimensional $G$-module, with $\g$ the Lie algebra of $G$, considered an algebra over $\FF$. We wish to give the vector space $\fa = \g \oplus M$ the structure of an algebra that is compatible with the $G$-action. In order to have closure, we may need to extend $M$ to a larger $G$-module, i.e., we may also consider the vector space $\fa' = \g \oplus M \oplus M^*$ in Section~\ref{sec:Z3}, or more in Section~\ref{sec:Zm}.

We will attempt to define a bracket on $\fa$ with the following properties:
\begin{equation}
[\;,\;] \colon \fa \times \fa \to \fa
\end{equation}
\begin{enumerate}
\item The bracket is bi-linear and hence equivalent to a structure tensor $B \in \fa^*\otimes \fa^* \otimes \fa$. 
\item The bracket is \emph{interesting}, i.e., the structure tensor is non-zero.
\item The bracket respects the grading, i.e., $[\;,\;] \colon \fa_i \times \fa_j \to \fa_{i+j}$. 
\item The bracket agrees with the $\g$-action on $\g$, and on $M$. This ensures that the Jordan decomposition respects the $G$-action on $M$.
\item\label{prop:equi} The structure tensor $B$ is $G$-invariant, so that the $G$-action on elements of $\fa$ is conjugation for adjoint operators, and hence Jordan decomposition makes sense. 
\setcounter{nameOfYourChoice}{\value{enumi}}
\end{enumerate} 
Additionally, we could ask for the following properties that would make $\fa$ into a Lie algebra. 
\begin{enumerate}
\setcounter{enumi}{\value{nameOfYourChoice}}
\item The bracket is globally skew-commuting, i.e., $[T,S] = -[S,T]$ for all $S,T \in \fa$. Note it must be skew-commuting for the products $\g \times \g \to \g$ and $\g \times M \to M$ if it is to respect the grading and the $G$-action on $\g$ and on $M$. 
\item The bracket satisfies the Jacobi criterion, making $\fa$ a Lie algebra.
\end{enumerate}
We will see that these last 2 criteria may not always be possible to impose. We may study potential connections to Lie superalgebras in future work. However, the conditions (1)-(5) are enough to define the following.

\begin{definition}
Given an element $T$ in an algebra $\fa$, we associate its \defi{adjoint form} \[\ad_T :=[T,\;] \colon \fa \to \fa.\]
\end{definition}

\begin{prop}
Suppose $\fa$ is a $G$-module. Then the structure tensor $B$ of an algebra $\fa$ is $G$-invariant if and only if the operation $T \mapsto \ad_T$ is $G$-equivariant in the sense that 
\begin{equation}
\ad_{gT} = g(\ad_T)g^{-1},
\end{equation}
with $gT$ denoting the $G$-action on $\fa$ on the LHS and juxtaposition standing for the matrix product on the RHS.
In particular, the Jordan form of $\ad_T$ is a $G$-invariant for $T\in \fa$. 
\end{prop}
\begin{proof}
Let $B = \sum_{i,j,k} B_{i,j,k} \a_i \otimes \a_j \otimes a_k \in \fa^* \otimes \fa^* \otimes \fa$ represent a potential bracket $[\;,\;]$.
For $T\in \fa$ we have that $\ad_T = B(T) \in \FF\otimes \fa^* \otimes \fa$ is the contraction in the first factor. 
Since the $G$-action is a linear action, it suffices to work with on rank-one tensors such as $\a \otimes \b \otimes a$, for which the contraction with $T$ is $\a(T)\cdot \b\otimes a$, and $\cdot$ denotes the scalar product. For $g\in G$ the $G$-action on the contraction is 
\[
g.(\a(T)\b\otimes a )= \a(T)\cdot(g.\b)\otimes (g.a),
\]
because $G$ acts as the identity on $\FF$.
Extending this by linearity and noting that $g.\b(v) = \b(g^{-1}v)$ (the dual action) we have that 
\begin{equation}\label{g.bt1}
g.(B(T)) = g(B(T))g^{-1},
\end{equation}
 where no dot (juxtaposition) means matrix product.
Then we compute:
\[\begin{matrix}
g.(\a\otimes \b \otimes a) = (g.\a)\otimes (g.\b) \otimes (g.a), \text{ and} \\[1ex]
(g.(\a\otimes \b \otimes a))(T) = (g.\a)(T) \cdot (g.\b)\otimes (g.a) = \a(g^{-1} T) \cdot (g.\b)\otimes (g.a).
\end{matrix}\]
This implies that 
\begin{equation}\label{g.bt2}
(g.B)(T) = g.(B(g^{-1}T)) 
= g(B(g^{-1}T))g^{-1},
\end{equation}
where the second equality is by \eqref{g.bt1}.

If we assume that $g.B = B$ we can conclude from \eqref{g.bt2} that
\[B(T) = g(B(g^{-1} T))g^{-1},\]
or replacing $T$ with $gT$ 
\[B(gT) = g(B( T))g^{-1},\]
Hence, the construction of the adjoint operators is $G$-equivariant. 
This argument is also reversible since \eqref{g.bt2} holds for any tensor $B$, and if
\[B(T) = g(B(g^{-1}T))g^{-1}\]
holds for all $T$, then $B(T) = (g.B)(T)$ for all $T$, which implies that $B = g.B$. 
\end{proof}

\begin{definition}\label{def:GJD}
We will say that a graded algebra $\mathfrak{a} = \g \oplus W$ has a Jordan decomposition consistent with the $G = \text{Lie}(\g)$-action (or say $\mathfrak{a}$ has GJD for short) if its structure tensor is $G$-invariant (and non-trivial). An element $T\in \fa$ is called \defi{ad-nilpotent} or simply \defi{nilpotent}, respecively \defi{ad-semisimple} or \defi{semisimple}, if $\ad_T$ is nilpotent (resp. semisimple).
\end{definition}

After we posted a draft of this article to the arXiv, Mamuka Jibladze asked us about the commuting condition for the Jordan-Chevalley decomposition, which led us to the following.
\begin{remark} If $\fa$ has GJD, and $T\in W$, considering $\ad_T \in \End(\fa)$, and its Jordan decomposition $\ad_T = (\ad_T)_S + (\ad_T)_N$ for semisimple $(\ad_T)_S $ and nilpotent $ (\ad_T)_N$ with $[ (\ad_T)_S , (\ad_T)_N] =0$. 
We ask if we can find corresponding $s,n \in \fa$ so that $\ad_s = (\ad_T)_S $ and $\ad_n = (\ad_T)_N $, and we don't require that $s,n$ commute in $\fa$, but rather that their corresponding adjoint operators do. Notice that if we did have $[\ad_s,\ad_n] = 0 $ in $\End(\fa)$ and $[s,n]= 0$ in $\fa$, then this would mean that the elements $s,n$ would have to satisfy the Jacobi identity $\ad_{[s,n]} = [\ad_s, \ad_n]$, which may not hold for all elements of $\fa$. 

The question of the existence of such $s,n$ regards the adjoint map itself $\ad\colon \fa \to \End (\fa)$, and we can try to solve for them by solving a system of equations on the image of the adjoint map. We report on this and other related computations in example~\ref{ex:g36}.
\end{remark}

\subsection{Invariants from adjoint operators}
Since conjugation preserves eigenvalues, we have that elementary symmetric functions of the eigenvalues, including the trace and determinant, and hence also the characteristic polynomial of the adjoint form of a tensor are all invariant under the action of $G$. Hence, $x$ and $y$ are not in the same orbit if they do not have the same values for each invariant. We list three such invariants for later reference.

\begin{definition}For generic $T \in \g$ the adjoint operator $\ad_T$ induces \defi{trace-power} invariants:
\[
f_k(T) := \tr((\ad_T)^k).
\]
\end{definition}
The ring $\mathcal F := \CC[f_1,\ldots,f_n]$ is a subring of the invariant ring $\CC[V]^G$, though $\F$ is not likely to be freely generated by the $f_k$, and could be a proper subring, see also \cite{wallach2005hilbert}. This opens a slew of interesting commutative algebra problems such as computing the syzygies of the $f_k$'s, or even the entire minimal free resolution (over $\CC[V]$) of such; obtaining a minimal set of basic invariants; finding expressions of other known invariants in terms of the $f_{k}$, for example, one could ask for an expression of the hyperdeterminant, as was done in \cite{BremnerHuOeding, HolweckOedingE8}. 

Invariance also holds for the set of algebraic multiplicities of the roots of the adjoint form of the tensor, as well as the ranks of the adjoint form and its blocks induced from the natural grading. This leads us to the following invariants that are easy to compute and often sufficient to distinguish orbits.
\begin{definition}\label{def:profiles}
Suppose $T\in M$, $\fa$ is an algebra containing $\g \oplus M$, and $\ad_T$ is the adjoint operator of $T$. 
We call the \defi{adjoint root profile}, which lists the roots (with multiplicities) of the characteristic polynomial of $\ad_{T}$. 
\end{definition}
\begin{definition}
We list the ranks of the blocks of $\ad_{T}$ and its powers and call this the \defi{adjoint rank profile}. We depict the rank profile by a table whose rows correspond to the power $k$ on $(\ad_T)^k$ and whose columns are labeled by the blocks of $(\ad_T)^k$, with the last column corresponding to the total rank.
\end{definition}
The ranks of powers of $\ad_T$ indicate the dimensions of the generalized eigenspaces for the 0-eigenvalue, and this can be computed without knowing the rest of the eigenvalues of $\ad_T$. In particular, if these ranks are not constant, $\ad_T$ is not semi-simple, and if the rank sequence does not go to 0 then $\ad_T$ is not nilpotent.

Recall that the \defi{null cone} $\mathcal{N}_G$ is the common zero locus of the invariants, i.e., the generators of $\CC[V]^G$. We have the following straightforward conclusion:
\begin{prop}
Suppose $\fa$ has a Jordan decomposition consistent with the $G$-action on $M$. If a tensor $T$ is in the null cone $\mathcal{N}_G$, then $\ad_T$ is a nilpotent operator. Moreover, if the trace powers $f_k$ generate the invariant ring $\CC[V]^G$, then null-cone membership and nilpotency are equivalent. 
\end{prop}
\begin{proof}
If $T\in \mathcal{N}_G$, then every invariant vanishes, in particular all elementary symmetric polynomials in the eigenvalues of $\ad_T$, hence all the eigenvalues of $\ad_T$ are zero, so null-cone membership always implies nilpotency. 
Conversely, if $\ad_T$ is nilpotent, then all the eigenvalues of $\ad_T$ are zero, and since the trace powers are symmetric functions in the eigenvalues of $\ad_T$, they all vanish. If these trace powers generate the invariant ring, then nilpotent tensors are in the null cone.
\end{proof}

\subsection{Algebra structures on $\End(V)_0$}
The vector space of traceless endomorphisms, denoted $\g = \End(V)_0$ (or $\sl_n$ or $\sl(V)$ when we imply the Lie algebra structure), can be given more than one $G$-invariant algebra structure, as we will now show.
We attempt to define a $G$-equivariant bracket 
\[
\g \times \g \to \g.
\]
The following result implies that up to re-scaling, there are two canonical $G$-equivariant product structures on $\g$, one commuting and one skew-commuting.

\begin{prop} Let $\g = \End(V)_0$ denote the vector space of traceless endomorphisms of a finite-dimensional vector space $V$. Then $\g^* \otimes \g^* \otimes \g$ contains 2 copies of the trivial representation, and more specifically, each of $\bw{2} \g^* \otimes \g$ and $S^2 \g^* \otimes \g$ contains a 1-dimensional space of invariants. 
\end{prop}
\begin{proof}
Since $\g$ is an irreducible $G$-module, we only need to show that there is an isomorphic copy of $\g^*$ in each of $\bw{2} \g^*$ and $S^2\g^*$, then each of $\bw{2} \g^*\otimes \g$ and $S^2\g^*\otimes \g$ will have a non-trivial space of invariants. This is just a character computation, but we can also see it as an application of the Pieri rule and the algebra of Schur functors. We do the case of $S^2 \g^*$ since the other case is quite similar, and we already know that $\End(V)_0$ has the structure of the Lie algebra $\sl (V)$ with a skew-commuting product.

Recall that $V \otimes V^* = \End(V)_0 \oplus \CC$, where the trivial factor is the trace, and as a Schur module $\End(V)_0 = S_{2,1^n-2}V$. 
Also $S^2(A\oplus B) = S^2A \oplus (A\otimes B) \oplus S^2 B $ so we can compute $S^2 \g$ by computing $S^2 (V\otimes V^*)$ and taking a quotient.
We have
\[
S^2(V\otimes V^*) = (S^2 V \otimes S^2 V^*) \oplus (\bw2 V \otimes \bw 2 V^*) 
.\]
Now apply the Pieri rule (let exponents denote repetition in the Schur functors)
\[
= (S_{2^n}V \oplus \underline{S_{3,2^{n-2},1}V} \oplus S_{4,2^{n-2}} V)
\oplus (\bw n V \oplus \underline{S_{2,1^{n-2}} V }\oplus S_{2,2, 1^{n-3}}V )
,\]
where we have underlined the copies of $\g$. Since $S^2(V\otimes V^*)$ contains 2 copies of $\g $, and only one copy can occur in the complement of $S^2 \sl(V)$ (which is $\g \otimes \CC \oplus S^2 \CC$), we conclude that there must be precisely one copy of $\g$ in $S^2 \g$. 
\end{proof}
\begin{remark}
Note the traceless commuting product is defined via:
\[\begin{matrix}
\g \times \g &\to& \g
\\
(A,B ) & \mapsto & (AB+ BA) - I\cdot \tr(AB+ BA).
\end{matrix}
\]
Then we know that the trace is $G$-invariant, the $G$ action is linear and moreover 
\[g(AB+BA)g^{-1} = (gAg^{-1})(gBg^{-1})+(gBg^{-1})(gAg^{-1}),\]
 so $g.[A, B] = [g.A, g.B]$, i.e., the product is $G$-equivariant.
\end{remark}

\begin{remark}
Since both $\bw{2} \sl_n \otimes \sl_n$ and $S^2 \sl_n \otimes \sl_n$ have a space of non-trivial invariants, we could put a commuting or skew-commuting product on $\sl_n$ and yield different algebra structures on $\fa$. However, if we want the product to agree with the action of $\sl_n$ on itself and with the action of $\sl_n$ on $M$ (and hence obtain a Jordan decomposition), then we should insist that we choose the bracket that is skew-commuting on $\sl_n$.
This structure is inherited from viewing $\g$ as the adjoint representation of $G$.
\end{remark}

\subsection{A $\ZZ_2$ graded algebra from a $G$-module.}\label{sec:Z2}
For a $G$-module $M$ we define $\fa = \g \oplus M$ and attempt to construct a bracket 
\[
[\;,\;] \colon \fa \times \fa \to \fa,
\]
 viewed as an element of a tensor product $B\in \fa^* \otimes \fa^* \otimes \fa$, with the requirements in Section~\ref{sec:requirements}.

For the bracket on $\fa$ to respect the $\ZZ_2$ grading $\fa_0 = \g$, $\fa_1 = M$, it must impose conditions that respect the following decomposition.
\[\fa^* \otimes \fa^* \otimes \fa = (\fa_0^* \oplus \fa_1^*) \otimes (\fa_0^* \oplus \fa_1^*) \otimes (\fa_0 \oplus \fa_1) \]
\[\begin{matrix}
 &=
 & \fa_0^* \otimes \fa_0^* \otimes \fa_0 &\oplus & \fa_0^* \otimes \fa_0^* \otimes \fa_1 
& \oplus& \fa_0^* \otimes \fa_1^* \otimes \fa_0 &\oplus & \fa_0^* \otimes \fa_1^* \otimes \fa_1 \\
&& \oplus\; \fa_1^* \otimes \fa_0^* \otimes \fa_0 &\oplus & \fa_1^* \otimes \fa_0^* \otimes \fa_1 
&\oplus& \fa_1^* \otimes \fa_1^* \otimes \fa_0 &\oplus & \fa_1^* \otimes \fa_1^* \otimes \fa_1 
\end{matrix}
\]
Correspondingly denote by $B_{ijk}$ the graded pieces of $B$, i.e., $B_{ijk}$ is the restriction of $B$ to $\fa_i^* \otimes \fa_j^* \otimes \fa_k$. 
Respecting the grading requires that the maps $B_{001} =0$, $B_{010} =0$, $B_{100} =0$, and $B_{111} =0$. 
So $B$ must have the following structure:
\[
B \in 
\begin{matrix}
\fa_0^* \otimes \fa_0^* \otimes \fa_0 
 &\oplus & \fa_0^*\otimes \fa_1^* \otimes \fa_1 &\oplus & \fa_1^*\otimes \fa_0^* \otimes \fa_1 
& \oplus& \fa_1^* \otimes \fa_1^* \otimes \fa_0 
,\end{matrix}
\]
 For $X\in \g$ write $B(X) = \ad_X$, likewise, for $T\in M$ write $B(T) = \ad_T$, and correspondingly with the graded pieces of each.
So, the adjoint operators have formats:
\begin{equation}\label{eq:block2}
B(X) = \begin{pmatrix}
B_{000}(X) & 0 \\ 
0 & B_{011}(X)
\end{pmatrix}, 
\quad \quad \text {and}\quad\quad B(T) = \begin{pmatrix}
0 & B_{110}(T) \\ 
 B_{101}(T) & 0
\end{pmatrix}
,\end{equation}
where we note that each of the blocks is a map:
\[\begin{matrix}
B_{000}(X) \colon \fa_0 \to \fa_0, & \quad &
B_{011}(X) \colon \fa_1 \to \fa_1, \\
B_{110}(T) \colon \fa_1 \to \fa_0, & \quad &
B_{101}(T)\colon \fa_0 \to \fa_1, 
\end{matrix}
\]
that depends linearly on its argument, $X$ or $T$. The linearity of the construction is apparent so that if $X\in \g, T\in M$, then 
\[
B(X+T) = B(X) + B(T),
\]
respecting the matrix decompositions at \eqref{eq:block2}.

Agreement with the $\g$-action would require that $B_{000}$ be the usual commutator on $\g$ and that $B_{011}$ should be the standard $\g$-action on $M$, which is not an obstruction.
The remaining requirement is a $G$-invariant in $\fa_1^*\otimes \fa_1^* \otimes \fa_0$ (respectively in $\fa_0^*\otimes \fa_1^* \otimes \fa_1$), which will allow for an invariant $B_{110}$ (respectively $B_{101}$). Formally:
\begin{prop}
The vector space $\fa = \g \oplus M = \fa_0 \oplus \fa_1$ has a $G$-invariant structure tensor, and hence elements of the corresponding graded algebra have a non-trivial Jordan decomposition that is consistent with the $G$-action on $T\in M$ if and only if the spaces of $G$-invariants in $\fa_1^*\otimes \fa_1^* \otimes \fa_0$ and in $\fa_0^*\otimes \fa_1^* \otimes \fa_1$ are non-trivial.
\end{prop}

Skew-symmetry would force the elements $B_{000}$ and $B_{110}$ to be skew-symmetric in their first two arguments, and $B_{101} = -B_{011}^\top$. On the level of modules, this is
\[
B \in 
\begin{matrix}
 \bw{2}\fa_0^* \otimes \fa_0 
 &\oplus & \fa_0^*\wedge \fa_1^* \otimes \fa_1 
& \oplus& \bw{2} \fa_1^* \otimes \fa_0 
,\end{matrix}
\]
where we have encoded the condition that $B_{101} = -B_{011}^\top$ by replacing $ \left(\fa_0^* \otimes \fa_1^* \otimes \fa_1 \right) 
 \oplus \left( \fa_1^* \otimes \fa_0^* \otimes \fa_1 \right) $ with $ \fa_0^* \wedge \fa_1^* \otimes \fa_1 $.
We record this condition as follows:
\begin{prop}
The algebra $\fa = \g \oplus M = \fa_0 \oplus \fa_1$ has a skew-commuting product with a non-trivial Jordan decomposition that is consistent with the $G$-action on $T\in M$ if and only if the spaces of $G$-invariants in $\bw{2}\fa_1^* \otimes \fa_0$ and in $\fa_0^* \wedge \fa_1^* \otimes \fa_1$ are non-trivial.
\end{prop}

\begin{example}[Trivectors on a 6-dimensional space]
\label{ex:g36}
Consider $M = \bw 3\CC^6$, and $\g = \sl_6$. Note that $M\cong M^*$ as $G$-modules, and likewise $\g^* = \g$. We ask if there is a non-trivial invariant 
\[
B \in 
\begin{matrix}
\fa_0^* \otimes \fa_0^* \otimes \fa_0 
 &\oplus & \fa_0^*\otimes \fa_1^* \otimes \fa_1 &\oplus & \fa_1^*\otimes \fa_0^* \otimes \fa_1 
& \oplus& \fa_1^* \otimes \fa_1^* \otimes \fa_0 
.\end{matrix}
\]
Noting the self-dualities and permuting tensor factors, we can check for invariants in 
\[
\begin{matrix}
\fa_0 \otimes \fa_0 \otimes \fa_0 
 &\oplus & \fa_0\otimes \fa_1 \otimes \fa_1 
.\end{matrix}
\]
By the Pieri rule we have $M\otimes M = \bw6 \CC^6 \oplus S_{2,1,1,1,1}\CC^6 \oplus S_{2,2,1,1}\CC^6 \oplus S_{2,2,2}\CC^6$. 

Since $\sl_6$ is irreducible, existence of a non-trivial space of invariants in $M^*\otimes M^* \otimes \sl_6$ requires a summand in $M\otimes M$ be isomorphic to $\sl_6$, which is the case since $\sl_6 \cong S_{2,1,1,1,1}\CC^6$ as a $G$-module. 
Note also (by \cite[Ex.~15.32]{FultonHarris})
$\bw{2} M = \bw6 \CC^6 \oplus S_{2,2,1,1}\CC^6$. 
So, it is impossible to have a $G$-invariant structure tensor for a globally skew-commuting product in this example.
But (by the same exercise) since $\sl_6 \subset S^2 \bw{3}\CC^6$, we see that $\fa$ does have a non-trivial $G$-invariant structure tensor that is commuting on $M$. We give up skew-symmetry and the Jacobi identity but retain Jordan decompositions of adjoint operators. 

The orbits of $\SL_6(\CC)$ in $\PP \bw3 \CC^6$ were classified in the 1930s by Schouten \cites{Schouten31,GurevichBook}. Their closures are linearly ordered. 
Table~\ref{tab:w3c6} shows that the adjoint rank profiles separate orbits. In the last case, we stop the table since the form is not nilpotent.
\begin{table}
\scalebox{.9}{
\begin{tabular}{l||l||l||l}
\begin{tabular}{l}Grassmannian:\\ $e_0 e_1 e_2$ \end{tabular}
&
\begin{tabular}{l} Restricted Chordal:\\ $e_0 e_1 e_2 + e_0 e_3 e_4$\end{tabular}
&
\begin{tabular}{l} Tangential: \\$e_0 e_1 e_2 + e_0 e_3 e_4 + e_1e_3e_5$\end{tabular}
& 
\begin{tabular}{l} Secant (general): \\$e_0 e_1 e_2 + e_3 e_4e_5$\end{tabular}
\\
$ \left|\begin{smallmatrix}
 B_{00} & B_{01} & B_{10}& B_{11}& B \\[.5ex] 
\hline\\[.5ex]
 0 & 10 & 10 & 0 & 20 \\
 0 & 0 & 0 & 1 & 1 \\
 0 & 0 & 0 & 0 & 0 \\
\end{smallmatrix} \right|$
&$ \left|\begin{smallmatrix}
 B_{00} & B_{01} & B_{10}& B_{11}& B \\[.5ex] 
\hline\\[.5ex]
0 &15 &15 &0 &30 \\
10 &0 &0 &6 &16 \\
0 &1 &1 &0 &2 \\
1 &0 &0 &0 &1 \\
0 &0 &0 &0 &0 \\
\end{smallmatrix}\right|$
&
$ \left|\begin{smallmatrix}
 B_{00} & B_{01} & B_{10}& B_{11}& B \\[.5ex] 
\hline\\[.5ex]
0 &19 &19 &0 &38 \\
18 &0 &0 &11 &29 \\
0 &10 &10 &0 &20 \\
9 &0 &0 &2 &11 \\
0 &1 &1 &0 &2 \\
0 &0 &0 &1 &1 \\
0 &0 &0 &0 &0 \\
\end{smallmatrix}\right|$
&
$ \left|\begin{smallmatrix}
 B_{00} & B_{01} & B_{10}& B_{11}& B \\[.5ex] 
\hline\\[.5ex]
0 &19 &19 &0 &38 \\
18 &0 &0 &19 &37 \\
0 &18 &18 &0 &36 \\
\end{smallmatrix}\right|$\\
\end{tabular}
}
\caption{Normal forms and adjoint rank profiles of orbits in $\PP \bw3 \CC^6$.}\label{tab:w3c6}
\end{table}
The characteristic polynomials for the nilpotent elements are $t^{55}$. For the (non-nilpotent) element $T = e_0 e_1 e_2 + e_3 e_4e_5$ we have $\chi_T(t) = \left(t\right)^{19}\left(3\,t^{2}-4\right)^{9}\left(3\,t^{2}+4\right)^{9}$, with root profile is $(19_0, (9_\CC)^2,(9_\RR)^2)$, i.e., there are roots of multiplicity 19 at 0, and 2 complex roots and 2 real roots each with multiplicity 9. 

For the nilpotent normal forms, the trace-power invariants are zero. For the general point, the trace powers of order $4k$ are non-zero. 
Since the ring of invariants is generated in degree 4, the invariant $\Tr(\ad_{T}^4)$ must be equal to a scalar multiple of this invariant, known as a hyperpfaffian. It has value $36$ on the form $e_0 e_1 e_2 + e_3 e_4e_5$. 

Now, we can ask for the JCF for the adjoint operator of these normal forms. Let us show the example with $S = e_0 e_1 e_2 + e_3 e_4e_5$. 
The kernel of the adjoint operator $\ad_S$ is spanned by the following elements from $\sl_6$: $h_1 = E_{0,0}- E_{1,1}$, $h_2 = E_{1,1}- E_{2,2}$, $h_4 = E_{3,3}- E_{4,4}$, $h_5 = E_{4,4}- E_{5,5}$ together with the 12 elements of the form $E_{i,j}$ where both $i$ and $j$ come from the same block from the partition $\{0,1,2\}\cup \{3,4,5\}$, and the element $S_- = -e_0 e_1 e_2 + e_3 e_4e_5$.

The kernel of $(\ad_S)^2$ increases by 1 dimension and includes the new vector $h_3$. 
The kernel of $(\ad_S)^3$ increases by 1 dimension and instead of the vector $S_-$, it is spanned by the two elements $e_0 e_1 e_2, e_3 e_4e_5$. 
So we can start a Jordan chain as:
\[v_1 = h_{3}+e_{3}e_{4}e_{5},\]
\[v_2 = \ad_S v_1 = \frac{1}{2}h_{1}+h_{2}+\frac{3}{2}h_{3}+h_{4}+\frac{1}{2}h_{5}-e_{0}e_{1}e_{2}+e_{3}e_{4}e_{5},\]
\[v_3 = \ad_S v_2 = -\frac{3}{2}e_{0}e_{1}e_{2}+\frac{3}{2}e_{3}e_{4}e_{5}.\]
Then complete the chain by adding elements from the kernel of $\ad_S$:
\[\begin{matrix}
h_{1},& h_{2},& h_{4},& h_{5},
& E_{0,\:1},& E_{0,\:2},& E_{1,\:2},& E_{3,\:4},\\
& E_{3,\:5},& E_{4,\:5},& E_{1,\:0},& E_{2,\:0},& E_{2,\:1},& E_{4,\:3},& E_{5,\:3},& E_{5,\:4}.
\end{matrix}
\]
The other eigenspaces have dimensions equal to their algebraic multiplicities, so choosing the remaining basis vectors of $\fa$ to be full sets of eigenvectors corresponding to the eigenvalues $\pm 1, \pm i$ for $\ad_S$ one obtains a matrix $Q$ whose columns correspond to these basis vectors and the final matrix $Q^{-1}\ad_S Q$ is in JCF, with only one non -diagonal block which is a $3\times 3$ Jordan block $J_3(0)$. 

As a final comment in this example, we mention that while none of the orbits of $\SL_6$ in $\bw3 \CC^6$ appear to be semisimple, we checked that the mixed vector $v_1$ above is, in fact, semisimple. It seems that there are many more things to discover about this algebra.
\end{example}

\begin{example}[4-vectors on an 8-dimensional space]
Now consider $M = \bw 4\CC^8$ and $\g = \sl_8$. Note that $M\cong M^*$ as $G$-modules. 
By the Pieri rule $M\otimes M = \bw8 \CC^8 \oplus S_{2,1,1,1,1,1,1}\CC^8 \oplus S_{2,2,1,1,1,1}\CC^8 \oplus S_{2,2,2,1,1}\CC^8 \oplus S_{2,2,2,2}\CC^8$. 
Since $\sl_8$ is irreducible, $M^*\otimes M^* \otimes \sl_8$ has a non-trivial space of invariants if and only if a summand in $M\otimes M$ is isomorphic to $\sl_8$, which is the case since $\sl_8 \cong S_{2,1,1,1,1,1,1}\CC^6$ as a $G$-module. 

Note also (by \cite[Ex.~15.32]{FultonHarris})
$\bw{2} M = \bw8 \CC^8 \oplus S_{2,1,1,1,1,1,1}\CC^8\oplus S_{2,2,2,1,1}\CC^8$, which contains a copy of $\sl_8$. 
So $\fa$ has a non-trivial $G$-invariant structure tensor for a skew-commuting product in this case. One checks that this product (which is unique up to scalar) also satisfies the Jacobi identity.
Antonyan \cite{Antonyan} noticed that this algebra is a copy of the Lie algebra $\mathfrak{e}_7$ and carried out Vinberg's method \cite{Vinberg75}, which says, essentially, that since $\mathfrak{e}_7$ is a semisimple Lie-algebra the nilpotent orbits can be classified by utilizing Dynkin classification of subalgebras of semisimple Lie algebras \cite{dynkin1960semisimple}. Antonyan uses a modification of Dynkin's \emph{Characteristics} to separate nilpotent orbits. 

The appendix in \cite{oeding2022} provides normal forms for each nilpotent orbit. The adjoint rank profiles can distinguish orbits. The adjoint rank profile has the advantage that it does not require one to be able to use the group action to put a given tensor into its normal form, and in that sense, it is an automatic computation. It is interesting to consider normal forms of nilpotent orbits whose stabilizers have type associated with the full Lie algebra $\mathfrak{e}_7$, and respectively $\mathfrak{e}_7(a_1)$ and $\mathfrak{e}_7(a_2)$. 
The respective normal forms, orbit numbers (from Antonyan), and adjoint rank profiles are listed in Table \ref{tab:e7s}.
\begin{table}
\[
\begin{matrix}
\text{\textnumero } 83: & e_{1345}+e_{1246}+e_{0356}+e_{1237}+e_{0247}+e_{0257}+e_{0167} \\
\text{\textnumero } 86: & e_{1245}+e_{1346}+e_{0256}+e_{1237}+e_{0347}+e_{0157}+e_{0167} \\
\text{\textnumero } 88: & e_{2345}+e_{1346}+e_{1256}+e_{0356}+e_{1237}+e_{0247}+e_{0157}
\end{matrix}
\]
%
\[\begin{matrix} \text{\textnumero } 83: \hfill \\
\left|\begin{smallmatrix}
B_{00} & B_{01} & B_{10}& B_{11}& B \\[.5ex]
\hline\\[.5ex]
 0&62&62&0&124\\
 54&0&0&61&115\\
 0&53&53&0&106\\
 46&0&0&52&98\\
 0&45&45&0&90\\
 38&0&0&44&82\\
 0&37&37&0&74\\
 31&0&0&36&67\\
 0&30&30&0&60\\
 24&0&0&29&53\\
 0&23&23&0&46\\
 19&0&0&22&41\\
 0&18&18&0&36\\
 14&0&0&17&31\\
 0&13&13&0&26\\
 10&0&0&12&22\\
 0&9&9&0&18\\
 6&0&0&9&15\\
 0&6&6&0&12\\
 4&0&0&6&10\\
 0&4&4&0&8\\
 2&0&0&4&6\\
 0&2&2&0&4\\
 1&0&0&2&3\\
 0&1&1&0&2\\
 0&0&0&1&1\\
 0&0&0&0&0\\
 \end{smallmatrix}\right|
 \end{matrix}\quad 
\begin{matrix} \text{\textnumero } 86: \hfill\\ 
 \left|\begin{smallmatrix}
 B_{00} & B_{01} & B_{10}& B_{11}& B \\[.5ex]
\hline\\[.5ex]
 0&61&61&0&122\\
 52&0&0&59&111\\
 0&50&50&0&100\\
 43&0&0&48&91\\
 0&41&41&0&82\\
 34&0&0&39&73\\
 0&32&32&0&64\\
 26&0&0&30&56\\
 0&24&24&0&48\\
 18&0&0&23&41\\
 0&17&17&0&34\\
 13&0&0&16&29\\
 0&12&12&0&24\\
 8&0&0&11&19\\
 0&7&7&0&14\\
 5&0&0&6&11\\
 0&4&4&0&8\\
 2&0&0&4&6\\
 0&2&2&0&4\\
 1&0&0&2&3\\
 0&1&1&0&2\\
 0&0&0&1&1\\
 0&0&0&0&0\\
 \end{smallmatrix}\right|
 \end{matrix} \quad
\begin{matrix} \text{\textnumero } 88: \hfill\\ 
 \left|\begin{smallmatrix}
 B_{00} & B_{01} & B_{10}& B_{11}& B \\[.5ex]
\hline\\[.5ex]
 0&63&63&0&126\\
 56&0&0&63&119\\
 0&56&56&0&112\\
 50&0&0&56&106\\
 0&50&50&0&100\\
 44&0&0&50&94\\
 0&44&44&0&88\\
 38&0&0&44&82\\
 0&38&38&0&76\\
 32&0&0&38&70\\
 0&32&32&0&64\\
 27&0&0&32&59\\
 0&27&27&0&54\\
 22&0&0&27&49\\
 0&22&22&0&44\\
 18&0&0&22&40\\
 0&18&18&0&36\\
 14&0&0&18&32\\
 0&14&14&0&28\\
 11&0&0&14&25\\
 0&11&11&0&22\\
 8&0&0&11&19\\
 0&8&8&0&16\\
 6&0&0&8&14\\
 0&6&6&0&12\\
 4&0&0&6&10\\
 0&4&4&0&8\\
 3&0&0&4&7\\
 0&3&3&0&6\\
 2&0&0&3&5\\
 0&2&2&0&4\\
 1&0&0&2&3\\
 0&1&1&0&2\\
 0&0&0&1&1\\
 0&0&0&0&0\\
 \end{smallmatrix}\right|\end{matrix}\]
\caption{Some normal forms of orbits in $\bw4 \CC^8$ and their adjoint rank profiles.}\label{tab:e7s}
\end{table}
These orbits are also distinguishable by their dimensions (seen in the first row of the adjoint rank profiles by Remark~\ref{rem:conical}). We also highlight orbits \textnumero 65, \textnumero 67, and \textnumero 69, which all have the same dimension (60). Their normal forms and adjoint rank profiles are listed in Table \ref{tab:60s}. Here, two of them even appear to have the same tensor rank (though the actual rank could be smaller). 
\begin{table}
\[
\begin{matrix}
\text{\textnumero } 65: & e_{2345}+e_{0246}+e_{1356}+e_{0237}+e_{1237}+e_{0147}+e_{0157}\\
\text{\textnumero } 67: &e_{1345}+e_{1246}+e_{0346}+e_{0256}+e_{1237}+e_{0247}+e_{0167}\\
\text{\textnumero } 69: &e_{1345}+e_{1246}+e_{0356}+e_{1237}+e_{0247}+e_{0157}
\end{matrix} 
\]
\[\begin{matrix} \text{\textnumero } 65: \hfill\\ 
 \left|\begin{smallmatrix}
 B_{00} & B_{01} & B_{10}& B_{11}& B \\[.5ex]
\hline\\[.5ex]
0&60&60&0&120\\
50&0&0&57&107\\
0&47&47&0&94\\
39&0&0&44&83\\
0&36&36&0&72\\
28&0&0&34&62\\
0&26&26&0&52\\
20&0&0&24&44\\
0&18&18&0&36\\
12&0&0&17&29\\
0&11&11&0&22\\
8&0&0&10&18\\
0&7&7&0&14\\
4&0&0&6&10\\
0&3&3&0&6\\
2&0&0&2&4\\
0&1&1&0&2\\
0&0&0&1&1\\
0&0&0&0&0\\
\end{smallmatrix}\right|\end{matrix}\quad
\begin{matrix} \text{\textnumero } 67: \hfill\\ 
 \left|\begin{smallmatrix}
 B_{00} & B_{01} & B_{10}& B_{11}& B \\[.5ex]
\hline\\[.5ex]
0&60&60&0&120\\
50&0&0&57&107\\
0&47&47&0&94\\
39&0&0&44&83\\
0&36&36&0&72\\
29&0&0&33&62\\
0&26&26&0&52\\
20&0&0&24&44\\
0&18&18&0&36\\
13&0&0&16&29\\
0&11&11&0&22\\
8&0&0&10&18\\
0&7&7&0&14\\
4&0&0&6&10\\
0&3&3&0&6\\
1&0&0&3&4\\
0&1&1&0&2\\
0&0&0&1&1\\
0&0&0&0&0\\
\end{smallmatrix}\right|\end{matrix}\quad
\begin{matrix} \text{\textnumero } 69: \hfill\\ 
 \left|\begin{smallmatrix}
 B_{00} & B_{01} & B_{10}& B_{11}& B \\[.5ex]
\hline\\[.5ex]
0&60&60&0&120\\
52&0&0&58&110\\
0&50&50&0&100\\
43&0&0&48&91\\
0&41&41&0&82\\
34&0&0&39&73\\
0&32&32&0&64\\
25&0&0&30&55\\
0&23&23&0&46\\
18&0&0&22&40\\
0&17&17&0&34\\
13&0&0&16&29\\
0&12&12&0&24\\
8&0&0&11&19\\
0&7&7&0&14\\
4&0&0&6&10\\
0&3&3&0&6\\
2&0&0&3&5\\
0&2&2&0&4\\
1&0&0&2&3\\
0&1&1&0&2\\
0&0&0&1&1\\
0&0&0&0&0\\
\end{smallmatrix}\right|\end{matrix}\]
\caption{More normal forms of orbits in $\bw4 \CC^8$ and their adjoint rank profiles.}\label{tab:60s}
\end{table}

Notice that even the ranks of the powers may not distinguish orbits, but the blocks for \textnumero 65 and \textnumero 67 do have some different ranks starting at the 6-th power.
\end{example}

The previous two examples are special cases of the following straightforward generalization:
\begin{theorem}
The vector space $\fa = \sl_{2m} \oplus \bw{m} \CC^{2m}$ has a $\ZZ_2$-graded algebra structure with a Jordan decomposition consistent with the $G$-action. There is a unique (up to scale) equivariant bracket product that agrees with the $\g$-action on $M= \bw{m}\CC^{2m}$. Moreover, it must satisfy the property that the restriction to $M \times M \to \g$ must be commuting when $m$ is odd and skew-commuting when $m$ is even. 
\end{theorem}
\begin{proof}
Note first that $\sl_{2m}$ is an irreducible $\g = \sl_{2m}$-module (the adjoint representation), and hence a non-zero invariant structure tensor exists if and only if there is a copy of $\g$ in $M\otimes M$. Moreover, the number of such is determined by the multiplicity of $\g$ in $M\otimes M$. 
Indeed, by \cite[Ex.~15.32]{FultonHarris}, we have for $M =\bw{m} \CC^{2m}$ that precisely one copy of $\g = S_{2,1^{2m-2}} \CC^{2m}$ is contained only in $S^2 \bw{m} \CC^{2m}$ when $m$ is odd, and only in $\bw2 \bw{m} \CC^{2m}$ when $m$ is even.
\end{proof}

\subsection{A $\ZZ_3$ graded algebra from a $\g$-module}\label{sec:Z3}
At the risk of confusion of notation, for this subsection, let
$\fa = \fa_0 \oplus \fa_1 \oplus \fa_{-1}$ with $\fa_0 = \g$ and $\fa_1 = M$ as before, but also $\fa_{-1} = M^*$, the dual $\g$-module.
For the bracket on $\fa$ to respect the $\ZZ_3$ grading, it must impose conditions that respect the following decomposition.
\[\begin{array}{rcl}
\fa^* \otimes \fa^* \otimes \fa &=& (\fa_0^* \oplus \fa_1^*\oplus \fa_{-1}^*) \otimes (\fa_0^* \oplus \fa_1^*\oplus \fa_{-1}^*) \otimes (\fa_0 \oplus \fa_1\oplus \fa_{-1}) \\
 &=& 
\bigoplus_{i,j,k \in \{0,1,-1\}} \fa_i^*\otimes \fa_j^* \otimes \fa_k
\end{array}
\]
Correspondingly denote by $B_{ijk}$ the graded pieces of $B$, i.e., $B_{ijk}$ is the restriction of $B$ to $\fa_i^* \otimes \fa_j^* \otimes \fa_k$, and we equate $1$ with $+$ and $-1$ with $-$ for notational ease.
Respecting the $\ZZ_3$ grading now requires the following vanishing: $B_{ijk} = 0 $ if $k \neq i+j \mod 3$. Thus, the only non-zero blocks of $B$ must be:
\[
\begin{matrix}
B_{000} & B_{0++} & B_{0--} \\
B_{+0+} & B_{+-0} & B_{++-} \\
B_{-0-} & B_{--+} & B_{-+0} 
\end{matrix}
\]

\noindent
So $B$ must have the following structure:
\[
B \in 
\begin{matrix}
 && \fa_0^* \otimes \fa_0^* \otimes \fa_0 &\oplus & \fa_0^*\otimes \fa_1^* \otimes \fa_1 &\oplus & \fa_0^*\otimes \fa_{-1}^* \otimes \fa_{-1} \\ 
&\oplus & \fa_1^*\otimes \fa_0^* \otimes \fa_1 &\oplus & \fa_1^*\otimes \fa_{-1}^* \otimes \fa_0 &\oplus & \fa_1^*\otimes \fa_1^* \otimes \fa_{-1} \\
& \oplus& \fa_{-1}^* \otimes \fa_0^* \otimes \fa_{-1} & \oplus& \fa_{-1}^* \otimes \fa_{-1}^* \otimes \fa_1 & \oplus& \fa_{-1}^* \otimes \fa_1^* \otimes \fa_0 
\end{matrix}
\]

Correspondingly, there are three types of adjoint operators: for $X\in \g$ write $B(X) = \ad_X$, likewise, for $T\in M$ write $B(T) = \ad_T$, and for $\tau \in M^*$ write $B(\tau) = \ad_\tau$, and correspondingly with the graded pieces of each.
So, the adjoint operators have formats
\begin{equation}\label{eq:block3}
\begin{matrix}
B(X) = \left(\begin{smallmatrix}
B_{000}(X) & 0 & 0 \\ 
0 & B_{0++}(X) &0 \\
0& 0 & B_{0--}(X)\\
\end{smallmatrix}\right), &
B(T) = \left(\begin{smallmatrix}
0 & 0 & B_{+-0}(T) \\ 
B_{+0+}(T) &0 & 0 \\
0& B_{++-}(T) &0\\
\end{smallmatrix}\right),\\\\
B(\tau) = \left(\begin{smallmatrix}
0 & B_{-+0}(\tau) & 0 \\ 
0 &0 & B_{--+}(\tau) \\
B_{-0-}(\tau)& 0 &0\\
\end{smallmatrix}\right).
\end{matrix}
\end{equation}
The linearity of the construction and the grading of the bracket is apparent.
Note that each block is a map that depends linearly on its argument ($X, T$, or $\tau$).

\begin{theorem}
The vector space $\fa = \sl_{n} \oplus \bw{k} \CC^{n} \oplus \bw{n-k} \CC^{n}$ has an essentially unique non-trivial $\ZZ_3$-graded algebra structure with a Jordan decomposition consistent with the $G$-action precisely when $n = 3k$. 
Any non-trivial equivariant bracket product must satisfy the property that the restriction to $M \times M \to \g$ must be skew-commuting when $k$ is odd and commuting when $k$ is even. 
\end{theorem}
\begin{proof}
Agreement with the $\g$-action requires that $B_{000}$ be the usual commutator on $\g$ and that $B_{0++}$ should be the usual $\g$-action on $M$, while $B_{0--}$ should be the usual $\g$-action on $M^*$. 
More care must be given to ensure that the other blocks come from an invariant tensor.

For $B(T)$ we seek a non-zero invariant tensor in each block of $ \fa_1^*\otimes \fa_0^* \otimes \fa_1 \oplus \fa_1^*\otimes \fa_{-1}^* \otimes \fa_0 \oplus \fa_1^*\otimes \fa_1^* \otimes \fa_{-1} $, or noting 
that $\fa_{-1}^* = \fa_1$ we seek a non-zero tensor in each block of
$ \fa_1^*\otimes \fa_0^* \otimes \fa_1 \oplus \fa_1^*\otimes \fa_{1} \otimes \fa_0 \oplus \fa_1^*\otimes \fa_1^* \otimes \fa_{1}^* $. In the last block, note that $(\bw{k} \CC^n)^{\otimes 3}$ decomposes by an iterated application of Pieri's rule. In order to have $3k$ boxes fitting in a rectangle of height $n$, and it can have either 0, 1, 2 or 3 columns, since $k\leq n$ and $n-k \leq n$, the corresponding possibilities for $k$ are (in order) $k=0$, $3k =n$, or $3k=2n$ or $3k = 3n$. The middle two are the only non-trivial ones, and they correspond to the modules $M = \bw{k} \CC^{3k}$ and $M^* =\bw{2k} \CC^{3k}$. Hereafter $n=3k$. 

Now we look for invariants in $ \fa_1^*\otimes \fa_0^* \otimes \fa_1 \oplus \fa_1^*\otimes \fa_{1} \otimes \fa_0 $. 
Since $\sl_{3k}$ is an irreducible $\g = \sl_{3k}$-module (the adjoint representation), an interesting non-zero invariant structure tensor exists if and only if there is a copy of $\g^*$ in $M^*\otimes M$. One sees a copy of $\g^*$ in $M^* \otimes M$ by taking the transpose and noting that $\g \cong \g^*$ as $\g$-modules. 

\begin{lemma}\label{lem:dualBrackets}
Consider $M = \bw{k} \CC^n$ and $M^* = \bw{n-k} \CC^n$. There is precisely one copy of $\sl_n = \g = S_{2,1^{n-2}}$ in $M \otimes M^*$. Moreover, if $n=2k$, then the copy of $\g$ lives in $S^2 M$ if $k$ is odd and in $\bw 2 M$ if $k$ is even. 
\end{lemma}
\begin{proof}[Proof of Lemma~\ref{lem:dualBrackets}]
By the Pieri rule for $M =\bw{k} \CC^{n}$ and $M^* = \bw{k} \CC^{n}$ there is always a copy of $\g = S_{2,1^{n-2}} \CC^{n}$ in $M
\otimes M^*$ obtained by adding to a column of height $k$ one box to the second column and the rest to the first column. The total number of boxes is $n$.
This decomposition is multiplicity-free, so there is only one copy of $\g $ in $M\otimes M^*$. 
The ``moreover'' statement follows from \cite[Ex.~15.32]{FultonHarris}. 
\end{proof}
Similarly, if we wish to define the bracket for elements $\tau$ of $M^*$ we must find an interesting non-zero invariant tensor in $\fa_{-1}^* \otimes \fa_0^* \otimes \fa_{-1} \oplus \fa_{-1}^* \otimes \fa_{-1}^* \otimes \fa_1 \oplus \fa_{-1}^* \otimes \fa_1^* \otimes \fa_0 $. This is the same computation as in the case of $T$ in $M$. So, we obtain no additional obstructions. 

The question of whether this bracket can be symmetric or skew-symmetric for $T \in M$ (respectively for $\tau \in M^*$) comes down to the existence of a non-zero invariant in $\bw{2} (\fa_1^*) \otimes \fa_{-1}$ or in $S^{2} (\fa_1^*) \otimes \fa_{-1}$. Again, this amounts to finding a copy of $\fa_{-1}$ in $\bw 2\fa_1$ or in $S^2 \fa_1$. 

Again, by \cite[Ex.~15.32]{FultonHarris}, when $k$ is even, there is only a copy of $\fa_{-1} \isom \fa_1$ in $S^2\fa_1$, hence the bracket must be commuting on this part. When $k$ is odd, there is only a copy of $\fa_{-1} \isom \fa_1$ in $\bw2\fa_1$; hence the bracket must be skew-commuting on this summand. Moreover, since these decompositions are multiplicity-free, the structure tensors are essentially unique.
\end{proof}
\begin{remark}
Up to a scalar multiple the map $\fa_{-1}\times \fa_{1} \to \fa_0$ must be contraction, and the map $\fa_1 \times \fa_1 \to \fa_{-1}$ must be multiplication. Note that the product of $k$ forms is skew-symmetric when $k$ is odd and symmetric when $k$ is even. Similarly, the map $\fa_{-1} \times \fa_{-1} \to \fa_{1}$ must be contraction with the volume form followed by product and then contraction again.
\end{remark}

Skew-symmetry would require that the maps $B_{000}$, $B_{++-}$, $B_{--+}$ must be skew-symmetric, and that 
 $B_{0++} = -B_{+0+}^\top$, 
 $B_{0--} = -B_{-0-}^\top$, 
 $B_{+-0} = -B_{-+0}^\top$, 
hence, this would force $B$ in
\[
\begin{matrix}
 \bw{2}\fa_0^* \otimes \fa_0 
& \oplus \bw{2} \fa_1^* \otimes \fa_{-1}
& \oplus \bw{2} \fa_{-1}^* \otimes \fa_{1} 
 &\oplus \fa_0^*\wedge \fa_1^* \otimes \fa_1 
 &\oplus \fa_0^*\wedge \fa_{-1}^* \otimes \fa_{-1}
 &\oplus \fa_1^*\wedge \fa_{-1}^* \otimes \fa_0 
,\end{matrix}
\]
where, for example, we have again encoded the condition that $B_{101} = -B_{011}^\top$ by replacing $ \left(\fa_0^* \otimes \fa_1^* \otimes \fa_1 \right) 
 \oplus \left( \fa_1^* \otimes \fa_0^* \otimes \fa_1 \right) $ with $ \fa_0^* \wedge \fa_1^* \otimes \fa_1 $.

\begin{example}[Trivectors on a 9-dimensional space]
In Table~\ref{tab:w3c9}, we report the adjoint rank profiles of points in $\bw{3} \CC^9$ of rank $k$ (as sums of points of the Grassmannian $\Gr(3,9)$). We note that the rank profiles distinguish ranks up to 5, but not 6.
\begin{table}
\begin{tabular}{l||l}
Rank 1: & Rank 2:\\
$\left|\begin{smallmatrix}
 B_{00} & B_{01} & B_{02} & B_{10} & B_{11} & B_{12} & B_{10} & B_{21} & B_{22} & B \\[.5ex]
\hline\\[.5ex]
0 &0 &19 &19 &0 &0 &0 &20 &0 &58 \\
0 &0 &0 &0 &0 &1 &0 &0 &0 &1 \\
0 &0 &0 &0 &0 &0 &0 &0 &0 &0 
\end{smallmatrix} \right|$
&
$\left|\begin{smallmatrix}
 B_{00} & B_{01} & B_{02} & B_{10} & B_{11} & B_{12} & B_{10} & B_{21} & B_{22} & B \\[.5ex]
\hline\\[.5ex]
0 &0 &38 &38 &0 &0 &0 &38 &0 &114 \\
0 &19 &0 &0 &0 &20 &19 &0 &0 &58 \\
0 &0 &0 &0 &1 &0 &0 &0 &1 &2 \\
0 &0 &0 &0 &0 &0 &0 &1 &0 &1 \\
0 &0 &0 &0 &0 &0 &0 &0 &0 &0 \\
\end{smallmatrix}\right|
$
\\[4ex]
\hline
Rank 3: &Rank 4:\\
$\left|\begin{smallmatrix}
B_{00} & B_{01} & B_{02} & B_{10} & B_{11} & B_{12} & B_{10} & B_{21} & B_{22} & B \\[.5ex]
\hline\\[.5ex]
0 &0 &56 &56 &0 &0 &0 &56 &0 &168 \\
0 &56 &0 &0 &0 &56 &56 &0 &0 &168 \\
\end{smallmatrix}\right|$
&
$\left|\begin{smallmatrix}
 B_{00} & B_{01} & B_{02} & B_{10} & B_{11} & B_{12} & B_{10} & B_{21} & B_{22} & B \\[.5ex]
\hline\\[.5ex]
0 &0 &72 &72 &0 &0 &0 &72 &0 &216 \\
0 &72 &0 &0 &0 &72 &72 &0 &0 &216 \\
\end{smallmatrix}\right|$
\\[2ex]
\hline
Rank 5:& Rank 6:\\
$\left|\begin{smallmatrix}
 B_{00} & B_{01} & B_{02} & B_{10} & B_{11} & B_{12} & B_{10} & B_{21} & B_{22} & B \\[.5ex]
\hline\\[.5ex]
0 &0 &80 &80 &0 &0 &0 &80 &0 &240 \\
0 &80 &0 &0 &0 &80 &80 &0 &0 &240 \\
\end{smallmatrix}\right|$
&
$\left|\begin{smallmatrix}
B_{00} & B_{01} & B_{02} & B_{10} & B_{11} & B_{12} & B_{10} & B_{21} & B_{22} & B \\[.5ex]
\hline\\[.5ex]
0 &0 &80 &80 &0 &0 &0 &80 &0 &240 \\
0 &80 &0 &0 &0 &80 &80 &0 &0 &240 \\
\end{smallmatrix}\right|$\\[2ex]
\end{tabular}
\caption{Adjoint rank profiles for each tensor rank in $\bw{3} \CC^9$.}\label{tab:w3c9}
\end{table}

We are curious about nilpotent orbits, \textnumero 79, \textnumero 82, and \textnumero 87 (in the notation of \cite{Vinberg-Elashvili}), with respective types $D_4$, $A_3+A_1$, and $A_3+A_1$.
Orbit \textnumero 79 has normal form with indices $129\, 138\, 237\, 456$ and the following adjoint rank profile:
\[ \left|\begin{smallmatrix}
 B_{00} & B_{01} & B_{02} & B_{10} & B_{11} & B_{12} & B_{10} & B_{21} & B_{22} & B \\[.5ex]
\hline\\[.5ex]
0&56&0&0&0&56&56&0&0&168\\
0&0&46&46&0&0&0&48&0&140\\
36&0&0&0&38&0&0&0&38&112\\
0&28&0&0&0&29&28&0&0&85\\
0&0&19&19&0&0&0&20&0&58\\
9&0&0&0&11&0&0&0&11&31\\
0&1&0&0&0&2&1&0&0&4\\
0&0&1&1&0&0&0&1&0&3\\
0&0&0&0&1&0&0&0&1&2\\
0&0&0&0&0&1&0&0&0&1\\
0&0&0&0&0&0&0&0&0&0\\
 \end{smallmatrix}\right|\]
\noindent 
Change to $129\, 138\, 237\, 458$ (should be of type $A_3+A_1$) produces the following rank profile:
\[ \left|\begin{smallmatrix}
 B_{00} & B_{01} & B_{02} & B_{10} & B_{11} & B_{12} & B_{10} & B_{21} & B_{22} & B \\[.5ex]
\hline\\[.5ex]
0&55&0&0&0&54&55&0&0&164\\
0&0&33&33&0&0&0&38&0&104\\
16&0&0&0&21&0&0&0&21&58\\
0&6&0&0&0&10&6&0&0&22\\
0&0&2&2&0&0&0&0&0&4\\
1&0&0&0&0&0&0&0&0&1\\
0&0&0&0&0&0&0&0&0&0
 \end{smallmatrix}\right|\]
This agrees with the rank profile of orbit \textnumero 82 and not that of \textnumero 87, which has rank profile:
\[ \left|\begin{smallmatrix}
 B_{00} & B_{01} & B_{02} & B_{10} & B_{11} & B_{12} & B_{10} & B_{21} & B_{22} & B \\[.5ex]
\hline\\[.5ex]
0&52&0&0&0&60&52&0&0&164\\
0&0&33&33&0&0&0&38&0&104\\
16&0&0&0&21&0&0&0&21&58\\
0&8&0&0&0&6&8&0&0&22\\
0&0&1&1&0&0&0&2&0&4\\
1&0&0&0&0&0&0&0&0&1\\
0&0&0&0&0&0&0&0&0&0
 \end{smallmatrix}\right|\]
 The differences are subtle and are not recognized by the total rank, only by the block ranks. 
 
The rank profiles and trace powers of all $\SL_9$ orbits in $\bw3\CC^9$ take too much space to include here, so we included them 
 as an ancillary file accompanying the arXiv version of this article. 
 
Two other interesting cases are two different representatives for orbit \textnumero 9:
\[E_1 = e_{3}e_{4}e_{5}+e_{0}e_{3}e_{6}+e_{1}e_{4}e_{6}+e_{2}e_{5}e_{6}+e_{1}e_{3}e_{7}+e_{2}e_{4}e_{7}+e_{0}e_{5}e_{7}+e_{1}e_{2}e_{8},\]
\[
E_2 = e_{3}e_{4}e_{5}+e_{0}e_{3}e_{6}+e_{1}e_{4}e_{6}+e_{2}e_{4}e_{6}+e_{2}e_{3}e_{7}+e_{0}e_{4}e_{7}+e_{1}e_{5}e_{7}+e_{0}e_{2}e_{8}.\]
Note that \cite{Vinberg-Elashvili} point out that these representatives have different minimal ambient regular semisimple subalgebras, but since they have the same characteristic, they lie on the same $\SL_9$-orbit. They both have the same rank profile (see below); hence, the rank profiles seem to give the same information as the characteristics.
\[ \left|\begin{smallmatrix}
 B_{00} & B_{01} & B_{02} & B_{10} & B_{11} & B_{12} & B_{10} & B_{21} & B_{22} & B \\[.5ex]
\hline\\[.5ex]
0&76&0&0&0&74&76&0&0&226\\
0&0&66&66&0&0&0&72&0&204\\
58&0&0&0&62&0&0&0&62&182\\
0&54&0&0&0&56&54&0&0&164\\
0&0&48&48&0&0&0&50&0&146\\
41&0&0&0&44&0&0&0&44&129\\
0&37&0&0&0&38&37&0&0&112\\
0&0&31&31&0&0&0&35&0&97\\
24&0&0&0&29&0&0&0&29&82\\
0&22&0&0&0&26&22&0&0&70\\
0&0&19&19&0&0&0&20&0&58\\
15&0&0&0&17&0&0&0&17&49\\
0&13&0&0&0&14&13&0&0&40\\
0&0&10&10&0&0&0&11&0&31\\
6&0&0&0&8&0&0&0&8&22\\
0&4&0&0&0&8&4&0&0&16\\
0&0&4&4&0&0&0&2&0&10\\
3&0&0&0&2&0&0&0&2&7\\
0&1&0&0&0&2&1&0&0&4\\
0&0&1&1&0&0&0&1&0&3\\
0&0&0&0&1&0&0&0&1&2\\
0&0&0&0&0&1&0&0&0&1\\
0&0&0&0&0&0&0&0&0&0
\end{smallmatrix}\right|\]
In the Appendix we provide the trace-power invariants and rank profiles for all the orbits in \cite{Vinberg-Elashvili}.
\end{example}

\subsection{Extending the exterior algebra}\label{sec:Zm}
\begin{theorem}
The vector space $\fa = \sl_{n} \oplus \bigoplus_{k=1 \ldots n-1}\bw{k} \CC^{n}$ has a $\ZZ_n$-graded algebra structure with a Jordan decomposition consistent with the $G = \SL(V)$-action. There is a unique (up to scale) equivariant bracket product that agrees with the $\g$-action on each $\g_i$. If $n = 2k$, the equivariant bracket must satisfy the property that the restriction to $\bw{k} \CC^{n} \times \bw{k} \CC^{n} \to \g$ must be commuting when $k$ is odd and skew-commuting when $k$ is even. For any $k$ such that $2k \neq n$ the bracket $\bw{k} \CC^{n} \times \bw{k} \CC^{n} \to \bw{2k \mod n}\CC^n$ must be skew-commuting when $k$ is odd and commuting when $k$ is even. 
\end{theorem}
\begin{proof}
To construct this algebra, consider the exterior algebra $\bw{\bullet} \CC^n$ as a graded algebra, and we attach $\sl_n$ to it to try to make a Jordan decomposition as well. We do this by replacing $\bw0\CC^n = \bw n\CC^n$ with $\sl_n = \fa_0$. We define the brackets $\sl_n \times \bw{k} \CC^n \to \bw{k}\CC^n$ via the usual Lie algebra action on $\bw{k} \CC^n$. 
We use the usual exterior algebra products or their contracted forms to define the brackets $[\,,\,] \colon \bw{i} \CC^n \times \bw{j} \CC^n \to \bw{i+j \mod n} \CC^n$ for $i+j \not \equiv 0 \mod n$. When $i+j = n$ we utilize Lemma~\ref{lem:dualBrackets} to define the bracket $[\,,\,] \colon \bw{k} \CC^n \times \bw{n-k} \CC^n \to \sl_n$, and in the case $n=2k$ we have a ``middle map'', which must be commuting / skew-commuting respectively when $k$ is odd /even, again by Lemma~\ref{lem:dualBrackets}. 

There is an equivariant bracket $\bw{k} \CC^{n} \times \bw{k} \CC^{n} \to \bw{2k \mod n}\CC^n$ if and only if there is an invariant in $(\bw{k} \CC^{n})^* \otimes (\bw{k} \CC^{n})^* \otimes \bw{2k \mod n}\CC^n$, or equivalently if and only if there is a copy of $ \bw{2k \mod n}\CC^n$ in $\bw{k} \CC^{n} \otimes \bw{k} \CC^{n}$. This is true again by the Pieri rule. 

Uniqueness comes from the fact that the Pieri rule not only gives the decomposition of $\bw{k}\CC^n \otimes \bw{\ell} \CC^n$ but also has the consequence that the decomposition into irreducible representations is multiplicity-free. In particular, the spaces of invariants $\bw{n}\CC^n$, or $\sl_n(\CC^)$ can only occur once, so the dimension of the space of choices of brackets is at most one, hence there is a unique bracket up to scale (in each graded piece).
\end{proof}
\begin{remark}
We see from this result that $\fa = \sl_{n} \oplus \bigoplus_{k=1 \ldots n-1}\bw{k} \CC^{n}$ cannot be a Lie algebra whenever the conditions of the theorem force it to have commuting products unless one takes those graded products to be the trivial product (choosing the scalars to be zero). 
\end{remark}

\subsection{Rank subadditivity and semi-continutity}
Because of the linearity of the construction, the rank of the adjoint operator $B(T)$ gives a bound for the additive rank and border rank of the element $T$. 
In the $\ZZ_{2}$-graded and $\ZZ_{3}$-graded cases, we have respective block structures seen at \eqref{eq:block2} and \eqref{eq:block3}, and these blocks also provide rank and border rank bounds. In general, the following statement mimics prior work of Landsberg and Ottaviani (see \cite{LanOtt11_Equations, galkazka2017vector}).
\begin{prop}[Landsberg-Ottaviani \cite{LanOtt11_Equations}]\label{prop:LanOtt}
Suppose the operation $T \mapsto F_T$ produces a linear mapping depending linearly on $T$. Then if the maximum over $T \in X$ has $\rank(F_T)= k$, then
\[
X\text{-}\rank(T) \leq r \Rightarrow \rank(F_T)\leq kr,
\]
and the contrapositive is the tensor rank bound:
\[
 \rank(F_T)> kr \Rightarrow X\text{-}\rank(T) > r 
.\]
Moreover, semi-continuity also allows us to conclude that
\[
 \rank(F_T)> kr \Rightarrow X\text{-border}\rank(T) > r 
.\]
\end{prop}
In our case, this result translates to the following.
\begin{prop}
Suppose $M$ is a graded $\g$-module such that $\g \oplus M$ has a meaningful Jordan decomposition, that $X$ is a $G$-variety for $G$ the connected component of the identity of the Lie group $\exp \g$. The $X$-rank of a tensor $T$ is bounded by
\[
X\text{-rank}(T) \leq \rank(B(T)) / k, \quad
and \quad
X\text{-rank}(T) \leq \rank(B_{I}(T)) / k_{I} 
.\]
where $k$ is the maximal possible rank of $B(S)$ for $S\in X$, $k_{I}$ is the maximal possible rank of $B_{I}(S)$ for $S\in X$ for each relevant index $I$.
\end{prop}

We also note that the semi-continuity of matrix rank implies that the rank profiles of adjoint forms can be used as a negative test for orbit closure containment.

\subsection{Dimensions of conical orbits from adjoint rank profiles} Suppose $\fa = \fa_0\oplus \fa_1 \oplus \cdots$ with $\g = \fa_0$ and a tensor $T\in M = \fa_1$, as constructed above so that $T$ has a meaningful Jordan decomposition. The rank of the block matrices $B_{101}(T)$ are connected to the dimension of the cone over the orbit closure $\PP(G.T)\subset\PP(M)$. More precisely, we have the following:
\begin{prop} Notation as above.
 If $T$ is such that $G.T\subset M$ is conical, i.e., invariant by scalar multiplication, then $\rank(B_{101}(T))=\dim\left(\PP(\overline{G.T})\subset\PP M \right)+1$. If $G.T$ is not conical, then
 $\rank(B_{101}(T))=\dim\left( \PP(\overline{G.T})\subset\PP M \right)$.
\end{prop}

\begin{proof} Consider $\pi:M\to \PP M$ the projection and let $p$ be a point of the orbit $G.T\subset M$. Our proof comes down to an interpretation of the tangent space $\widehat{T}_p G.T$. Without loss of generality, we may assume $p=T$ since the dimension of the tangent space of an orbit is constant along the orbit.
The tangent space $\widehat{T}_T G.T$ projects to $T_{\pi(T)}\pi (G.T)$ and we have:
\begin{equation}
 \widehat{T}_T G.T=T+[\mathfrak{g},T].
\end{equation}
If $G.T$ is conical then $T\in [\mathfrak{g},T]$ while if $G.T$ is not conical $T\notin [\mathfrak{g},T]$. Combining with the fact that the image of $B_{101}(T)$ is
 the space $[\mathfrak{g},T]$ the result follows.
\end{proof}

\begin{remark}\label{rem:conical}
Since nilpotent elements are conical, we can read the dimension of the orbit from the block of the first row in the Table of Example \ref{ex:g36}. For the Grassmannian, the restricted chordal and the tangential variety, which are all nilpotent and thus conical, the first entry of the table provides the dimension of the cone of the orbit $(10, 15, 19)$, i.e., the dimension of the projective variety plus one. However, the secant variety is not conical, and therefore the value provided by the table, i.e., $19$, is the dimension of the secant variety $\sigma_2(\Gr(3,6))\subset \PP \bw3 \CC^6$. 
\end{remark}

\subsection{Some approachable examples}
As the exterior algebra $\bw{\bullet}\CC^n$ has dimension $2^n$, this can get unwieldy quickly. Here, we study several small cases that are still manageable.
\begin{example}[$4\times 4\times 4$ tensors and matrix multiplication]\label{ex:444}
Notice that we have a containment $\bw 3 \CC^{12} \supset \CC^4 \otimes \CC^4\otimes \CC^4$, so we can work with the algebra $\fa = \sl_{12}\oplus \bw 3 \CC^{12} \oplus \bw 6\CC^{12} \oplus \bw 9 \CC^{12} $, and consider tensors in $\CC^4 \otimes \CC^4\otimes \CC^4$ as elements in $\fa$, which has dimension 1507, and yields adjoint operators represented by $1507 \times 1507$ matrices.

We might also use $\fa^{(4,4,4)} = \sl_{4}^{\times 3}\oplus (\CC^4 \otimes \CC^4\otimes \CC^4) \oplus (\bw 2 \CC^4 \otimes \bw 2 \CC^4\otimes \bw 2\CC^4) \oplus (\bw 3 \CC^4 \otimes \bw 3 \CC^4\otimes \bw 3\CC^4) $, which produces adjoint operators with matrix size $389 \times 389$.
The matrix $2\times 2$ multiplication tensor viewed in $\bw 4 \CC^{12}$ is 
computed via $\trace(ABC^\top)$ for generic $2\times 2$ matrices $A,B,C$, so that $AB = C$, and the term $\lambda a_ib_j c_k$ is present in $\trace(ABC^\top)$ if $\lambda a_i b_j$ appears in the entry corresponding to $c_k$ and the indices $i,j,k$ are double indices.

In Table~\ref{tab:mmult} we list the rank profiles for adjoint operators in $\fa$ for several tensors, starting with the matrix multiplication tensor, then tensors of increasing tensor ranks. The adjoint form for random of Rank $7$ has the same rank profile as the Rank $6$ case, so the rank profiles over $\fa$ do not distinguish Rank $6$ from Rank $7$. 
We also tried restricting to just the sub-algebra $\fa^{(4,4,4)}$ and found that we get less information: and $5$ and $6$ have the same profile in $\fa^{(4,4,4)}$, but their rank profiles in $\fa$ are distinct.

\begin{table}
\[\text{mmult}:\;  e_{0}e_{4}e_{8}+e_{2}e_{5}e_{8}+e_{1}e_{4}e_{9}+e_{3}e_{5}e_{9}+e_{0}e_{6}e_{10}+e_{2}e_{7}e_{10}+e_{1}e_{6}e_{11}+e_{3}e_{7}e_{11}\]
\[\begin{array}{rl}
\ad_T^\fa: & \left|\begin{smallmatrix}
B_{00}&B_{10}&B_{20}&B_{30}&B_{01}&B_{11}&B_{21}&B_{31}&B_{02}&B_{12}&B_{22}&B_{32}&B_{03}&B_{13}&B_{23}&B_{33}&B \\
\hline\\[.5ex]
 0&132&0&0&0&0&219&0&0&0&0&219&132&0&0&0&702\\
 0&0&132&0&0&0&0&0&132&0&0&0&0&132&0&0&396\\
 0&0&0&0&0&0&0&0&0&132&0&0&0&0&132&0&264\\
 0&0&0&0&0&0&0&0&0&0&132&0&0&0&0&0&132\\
 0&0&0&0&0&0&0&0&0&0&0&0&0&0&0&0&0\\
\end{smallmatrix}\right|
\end{array}
\]

\[\text{Rank 1: } e_{0}e_{4}e_{8}\]
\[\begin{array}{rl}
\ad_T^\fa: &  \left|\begin{smallmatrix}
B_{00}&B_{10}&B_{20}&B_{30}&B_{01}&B_{11}&B_{21}&B_{31}&B_{02}&B_{12}&B_{22}&B_{32}&B_{03}&B_{13}&B_{23}&B_{33}&B \\
\hline\\[.5ex]
0&28&0&0&0&0&84&0&0&0&0&84&28&0&0&0&224\\
0&0&0&0&0&0&0&0&0&0&0&0&0&1&0&0&1\\
0&0&0&0&0&0&0&0&0&0&0&0&0&0&0&0&0\\
\end{smallmatrix}\right|
\end{array}
\]

\[\text{Rank 2: } e_{0}e_{4}e_{8}+e_{1}e_{5}e_{9}\]
 \[\begin{array}{rl}
\ad_T^\fa: &  \left|\begin{smallmatrix}
B_{00}&B_{10}&B_{20}&B_{30}&B_{01}&B_{11}&B_{21}&B_{31}&B_{02}&B_{12}&B_{22}&B_{32}&B_{03}&B_{13}&B_{23}&B_{33}&B \\
\hline\\[.5ex]
0&56&0&0&0&0&147&0&0&0&0&147&56&0&0&0&406\\
0&0&37&0&0&0&0&0&37&0&0&0&0&20&0&0&94\\
0&0&0&0&0&0&0&0&0&1&0&0&0&0&1&0&2\\
0&0&0&0&0&0&0&0&0&0&1&0&0&0&0&0&1\\
0&0&0&0&0&0&0&0&0&0&0&0&0&0&0&0&0\\
\end{smallmatrix}\right|
\end{array}
\]

\[\text{Rank 3: } e_{0}e_{4}e_{8}+e_{1}e_{5}e_{9}+e_{2}e_{6}e_{10}\]
 \[\begin{array}{rl}
\ad_T^\fa: &  \left|\begin{smallmatrix}
B_{00}&B_{10}&B_{20}&B_{30}&B_{01}&B_{11}&B_{21}&B_{31}&B_{02}&B_{12}&B_{22}&B_{32}&B_{03}&B_{13}&B_{23}&B_{33}&B \\
\hline\\[.5ex]
0&84&0&0&0&0&192&0&0&0&0&192&84&0&0&0&552\\
0&0&83&0&0&0&0&0&83&0&0&0&0&57&0&0&223\\
0&0&0&0&0&0&0&0&0&56&0&0&0&0&56&0&112\\
0&0&0&0&0&0&0&0&0&0&56&0&0&0&0&0&56\\
0&0&0&0&0&0&0&0&0&0&0&0&0&0&0&0&0\\
\end{smallmatrix}\right|
\end{array}
\]

\[\text{Rank 4: } e_{0}e_{4}e_{8}+e_{1}e_{5}e_{9}+e_{2}e_{6}e_{10}+e_{3}e_{7}e_{11}\]
 \[\begin{array}{rl}
\ad_T^\fa: &  \left|\begin{smallmatrix}
B_{00}&B_{10}&B_{20}&B_{30}&B_{01}&B_{11}&B_{21}&B_{31}&B_{02}&B_{12}&B_{22}&B_{32}&B_{03}&B_{13}&B_{23}&B_{33}&B \\
\hline\\[.5ex]
0&111&0&0&0&0&219&0&0&0&0&219&111&0&0&0&660\\
0&0&111&0&0&0&0&0&111&0&0&0&0&111&0&0&333\\
0&0&0&0&0&0&0&0&0&111&0&0&0&0&111&0&222\\
0&0&0&0&0&0&0&0&0&0&111&0&0&0&0&0&111\\
0&0&0&0&0&0&0&0&0&0&0&0&0&0&0&0&0\\
\end{smallmatrix}\right|
\end{array}
\]

\[\text{Rank 5: }e_{0}e_{4}e_{8}+e_{1}e_{5}e_{9}+e_{2}e_{6}e_{10}+e_{3}e_{7}e_{11} + abc\] 
 \[\begin{array}{rl}
\ad_T^\fa: &  \left|\begin{smallmatrix}
B_{00}&B_{10}&B_{20}&B_{30}&B_{01}&B_{11}&B_{21}&B_{31}&B_{02}&B_{12}&B_{22}&B_{32}&B_{03}&B_{13}&B_{23}&B_{33}&B \\
\hline\\[.5ex]
 0&135&0&0&0&0&219&0&0&0&0&219&135&0&0&0&708\\
0&0&135&0&0&0&0&0&135&0&0&0&0&135&0&0&405\\
0&0&0&0&0&0&0&0&0&135&0&0&0&0&135&0&270\\
0&0&0&0&0&0&0&0&0&0&135&0&0&0&0&0&135\\
0&0&0&0&0&0&0&0&0&0&0&0&0&0&0&0&0\\
\end{smallmatrix}\right|
\end{array}
\]
 
\[\text{Rank 6: } e_{0}e_{4}e_{8}+e_{1}e_{5}e_{9}+e_{2}e_{6}e_{10}+e_{3}e_{7}e_{11} + a_1b_1c_1 + a_2b_2c_2\] 
 \[\begin{array}{rl}
\ad_T^\fa: &  \left|\begin{smallmatrix}
B_{00}&B_{10}&B_{20}&B_{30}&B_{01}&B_{11}&B_{21}&B_{31}&B_{02}&B_{12}&B_{22}&B_{32}&B_{03}&B_{13}&B_{23}&B_{33}&B \\
\hline\\[.5ex]
 0&141&0&0&0&0&219&0&0&0&0&219&141&0&0&0&720\\
0&0&141&0&0&0&0&0&141&0&0&0&0&141&0&0&423\\
0&0&0&0&0&0&0&0&0&141&0&0&0&0&141&0&282\\
0&0&0&0&0&0&0&0&0&0&141&0&0&0&0&0&141\\
0&0&0&0&0&0&0&0&0&0&0&0&0&0&0&0&0\\
 \end{smallmatrix}\right|\\
\end{array}
\]
 
\caption{Some adjoint rank profiles of tensors in $\CC^4\otimes \CC^4\otimes \CC^4$ viewed as elements of $\fa = \sl_{12}\oplus  \bw{3}\CC^{12}$.}\label{tab:mmult}
\end{table}
We were curious to see if there was a difference in the rank-detection power between $\fa = \sl_{12} \bigoplus_{k=1}^3 \bw{3k}\CC^{12}$ and its parent $\tilde \fa = \sl_{12} \bigoplus_{k=1}^{11} \bw{k}\CC^{12}$. Since the latter algebra is much larger (dimension $4327$), the computations take longer, but we found they're still in reach. However, this matrix is still unable to detect the difference between general among rank 6 and general among rank 7. Since there are 144 blocks in each adjoint operator corresponding to this $\ZZ_{12}$ graded algebra, we only report the ranks of the powers of these operators for each rank of tensor in Table~\ref{tab:fullC12}:
\begin{table}
\begin{tabular}{r| |l|l|l|l|l|l|l|l|}
ad powers & Rank 1: & Rank 2:&Rank 3: & Rank 4:&Rank 5: & Rank 6: &Rank 7: & mmult \\
\hline
$\!\begin{array}{c}
 A\\
 A^2\\
 A^3\\
 A^4\\
 A^5
 \end{array}\!$
 &
 
 $\!\begin{array}{c}
 572\\
 1\\
 0\\
 0\\
 0
 \end{array}\!$
&
$\!\begin{array}{c}
 1\,018\\
 118\\
 14\\
 1\\
 0
 \end{array}\!$
& 
$\!\begin{array}{c}
 1\,368\\
 259\\
 130\\
 56\\
 0
 \end{array}\!$
& 
$\!\begin{array}{c}
 1\,644\\
 381\\
 246\\
 111\\
 0
 \end{array}\!$
& 
$\!\begin{array}{c}
1\,824 \\
453\\
294\\
135\\
 0
 \end{array}\!$
&
$ \!\begin{array}{c}
1\,860\\
471\\
306\\
141\\
 0
 \end{array}\! $
&
$\!\begin{array}{c}
1\,860\\
471\\
306\\
141\\
 0
 \end{array}\! $
&
$\!\begin{array}{c}
1\,812\\
444\\
288\\
132\\
0
\end{array}\! $
\end{tabular}
\caption{Ranks of powers of adjoint operators $A^k$ of various ranks of tensors and the matrix multiplication tensor all living in $(\CC^4)^{\otimes 4} \subset \bw{3}\CC^{12}$ considered as elements of the algebra $\sl_{12} \oplus \bigoplus_{d=1}^{11} \bw{d}\CC^{12}$.}\label{tab:fullC12}
\end{table}


These computations seem to suggest that the matrix multiplication tensor should be somehow below rank 5, which seems contradictory to the well-known result that the border rank of the 2 by 2 matrix multiplication operator is 7. However, the rank semi-continuity arguments actually go in the other direction. If the rank of the linear operator of a given tensor is greater than that for a known tensor rank, then that provides a lower bound for the border rank. But when the rank is higher we get no conclusion unless we could say that the rank conditions on the linear operator were necessary and sufficient for that particular border rank. This appears not to be the case. So all we can conclude is that this tool shows that matrix multiplication has border rank at least 5 (not new) and suggests that it is not generic among those tensors of rank 5. We are curious if the Jordan form for these adjoint operators on this algebra sheds any additional light on the orbit closure issue that arises here. At present, we do not have any additional information to report.
\end{example}

\begin{remark}Note on a 2020 era desktop Macintosh computer running v.1.21 of Macaulay2 we have the following approximate run times: For the matrix multiplication tensor, the adjoint matrix took approximately 4.3s to construct, and the rank profile was computed in 3s. Up until tensor rank 5 building the adjoint matrix takes about 2-3s, and computing the rank profile takes approximately 4-5s. For tensor rank 6 and 7, building the matrix took 6.5s-7, but finding the rank profile took 108s and 360s respectively over $\QQ$, but reducing to $\ZZ_{10000000019}$ gives the same answers in under 3s. The reason for the increase in complexity in computing the rank over $\QQ$ is that for the low-rank tensors, we selected normal forms, which are sparse. However, for higher rank forms, we must choose random rank-1 elements to add on, which produces denser adjoint matrices and, hence, greater accumulation of size of intermediate expressions. Reducing modulo a large prime mitigates the coefficient explosion.
\end{remark}

\begin{example}[Possible semi-simple-like elements for $4\times 4\times 4$ tensors]\label{ex:quasi-semisimple}
There is a special set of elements in $\CC^4 \otimes \CC^4 \otimes \CC^4$ that mimick part of Nurmiev's \cite{nurmiev} classification of elements in $\CC^3 \otimes \CC^3 \otimes \CC^3$ from the Vinberg-Elashvili classification \cite{Vinberg-Elashvili} of trivectors in $\bw{3}\CC^9$. A basis $e_1,\ldots e_{12}$ of $\CC^{12}$, a partition 
\[P = \{\{1,2,3,4\},\{5,6,7,8\}, \{9,10,11,12\}\},\]
and corresponding splitting $\CC^{12} = \CC^4 \oplus \CC^4 \oplus \CC^4 $ induces an inclusion $\CC^4 \otimes \CC^4 \otimes \CC^4 \subset \bw{3} \CC^{12}$. We then ask for the combinatorial design that consists of quadruples of non-intersecting lines in $[12]$, each line with 3 points, each line must contain precisely 1 element from each of the parts of the partition $P$, and no two lines intersecting in more than 1 point. By exhaustion, we have 4 quadruples of lines, which lead to the following basic elements which we call \emph{quasi-semisimple}:
\[
\begin{matrix}
p_1 = e_{1}\otimes e_{5}\otimes e_{9} + e_{2}\otimes e_{6}\otimes e_{10} + e_{3}\otimes e_{7}\otimes e_{11} + e_{4}\otimes e_{8}\otimes e_{12},\\
p_2 = e_{1}\otimes e_{6}\otimes e_{11} + e_{2}\otimes e_{5}\otimes e_{12} + e_{3}\otimes e_{8}\otimes e_{9} + e_{4}\otimes e_{7}\otimes e_{10}, \\
p_3 = e_{1}\otimes e_{7}\otimes e_{12} + e_{2}\otimes e_{8}\otimes e_{11} + e_{3}\otimes e_{5}\otimes e_{10} + e_{4}\otimes e_{6}\otimes e_{9},\\
p_4 = e_{1}\otimes e_{8}\otimes e_{10} + e_{2}\otimes e_{7}\otimes e_{9} + e_{3}\otimes e_{6}\otimes e_{12} + e_{4}\otimes e_{5}\otimes e_{11}.
\end{matrix}
\]
These tensors have the same adjoint rank profiles as those of rank 4 listed in Example~\ref{ex:444}. The elements that come from $\fa_1$ appear to all be nilpotent, and hence, there would be no ad-semisimple elements that live entirely in $\fa_1$. Moreover, these elements do not commute. However, the combinatorial structure seems interesting and could be interesting for further study. 

We also note that the matrix multiplication tensor has an expression that looks like the sum of 2 of these semi-simple-like elements, and moreover, we checked that the adjoint rank profile of $p_i + p_j$ is identical to that of mmult. 

Moreover, like in example~\ref{ex:g36}, we are curious to understand the adjoint elements that are not purely in a single graded piece of $\fa$ and which could be semisimple. A potential candidate could be $h_0 + e_0 e_4 e_8$, with $h_0 $ the first basis vector in the Cartan in $\sl V$, which appears to have constant ranks for powers of the adjoint operator. 
  So it seems that there are ad-semisimple elements in this algebra, just not concentrated in a single grade.
\end{example}

\begin{example}[Trivectors on a 10-dimensional space]\label{ex:wedge3C10}
It is reasonable that this method could handle the $\ZZ_{10}$-graded algebra for $\bw{3}\CC^{10}$.

Here we show that our tools can be constructed for $\bw{3}\CC^{10}$, where $\fa = \g_0 \oplus \cdots \oplus \g_9$, in the order 0, 3, 6, 9, 12= 2, 5, 8, 11= 1, 4, 7, 
with $\g_i = \bw{3*i}\CC^{10}$ ,
 $\g_{3+i} = \bw{3*i-1}\CC^{10}$,
 and 
 $\g_{6+i} = \bw{3*i -2}\CC^{10}$, for $i=1,2,3$. 
 
 We were able to compute adjoint operators in this algebra. There are many blocks, so we do not report the entire block rank adjoint rank profile. Here are the results for the first 5 ranks, the results for rank 5 are the same as for sums of more than 5 rank 1 elements.

 Rank 1:
 $\left\{176,\:1,\:0\right\}$
 
Rank 2:
$\left\{322,\:94,\:14,\:1,\:0\right\}$ 

Rank 3:
$\left\{444,\:223,\:130,\:56,\:0\right\}$

Rank 4:
$\left\{536,\:294,\:178,\:79,\:0\right\}$

Rank 5:
$\left\{566,\:337,\:218,\:99,\:0\right\}$

The point of this example is to say that computations in this algebra are in the computable regime, and we encourage future research here.
\end{example}

\subsection{Connection to geometric constructions}
The preceding two examples seem relevant for many other geometric constructions, such as K3 surfaces, K\"ahler, and hyper-K\"ahler manifolds, and more, which we explain briefly now.

Hitchen \cite{hitchin2001stable} was interested in the existence of special $G$-structures, metrics with special holonomy.
A $p$-form is to be considered \emph{stable} if it lies in an open orbit. The existence of open orbits is the classification of prehomogeneous vector spaces \cite{sato1977classification}, which relies on the so-called castling transforms (see \cite{venturelli2019prehomogeneous} for a modern treatment in the case of tensors). A strong necessary condition is that the group $G$ must have dimension at least that of the vector space $W$ (in our case $W = \bw p V$ with $\dim V = n$) for there to be an open orbit. Hitchen considered the prehomogeneous cases $(n,p)$ (with  $n=\dim V$), $(2m, 2); (6,3), (7,3), (8,3),$ and their duals. Stable $p$-forms determine interesting $G$-structures in these cases.

 After this we can consider \emph{semi-stable} $p$-forms, which are those which do not contain $0$ in their orbit. Semi-stable $p$-forms govern many interesting geometric objects (K3 surfaces, K\"ahler, and hyper-K\"ahler manifolds, genus 2 curves, the Coble cubic), see for instance \cite{rains2018invariant, bernardara2021nested, benedetti2023hecke, bernardara2024even, swann1990hyperkahler} which explore connections to tensors in $\bw 3 \CC^9$.
  
This notion of semi-stability is the same as the notion of a form being not \emph{nilpotent}. In the case where Jordan decomposition exists for $\bw p V$ we can study the nicest semi-stable forms to be those that are semi-simple. A slight relaxation of this concept which we advocate for is when we may not have a Jordan decomposition for elements of $\bw p V$, but we still have GJD, and we can ask for forms that are Ad-semi-simple, i.e. elements of $T \in \bw p V$  such that $\ad(T)$ is semi-simple. We find it interesting to note that, as seen in Example \ref{ex:wedge3C10} it seems that every element of $\bw{3}\CC^{10}$ is ad-nilpotent, and our experiments suggest likewise for $\bw{3}\CC^{11}$ and $\bw 3 \CC^{12}$.

For those cases, in Example~\ref{ex:quasi-semisimple} we suggest a possible replacement for the notion of ad-semisimple in this last case. It would be interesting to understand the implications for geometry for these cases where it seems that the traditional notions of semi-stable and semi-simple $3$-forms or $4$-forms may not be possible, and we ask specifically if the notion of quasi-semisimple could be a meaningful replacement.

Finally, we remark that we found that elements of $\bw 4 \CC^{10}$, that are not ad-nilpotent, and we find that 4-vectors like these could be an interesting place to start a new investigation of semi-stable $4$-forms on a 10-dimensional space. In that case on a 2018 mac laptop it takes approximately one minute to compute adjoint operators when working over a field like $\ZZ_{1009}$.
For example, the form
\[e_{0}e_{2}e_{3}e_{7}+e_{1}e_{3}e_{6}e_{8}+e_{0}e_{4}e_{6}e_{8}+e_{2}e_{4}e_{5}e_{9}+e_{1}e_{5}e_{7}e_{9}\]
is not ad-nilpotent; it has the following rank profile:
\[\left|\begin{smallmatrix}
B_{00}&B_{10}&B_{20}&B_{30}&B_{40}&B_{01}&B_{11}&B_{21}&B_{31}&B_{41}&B_{02}&B_{12}&B_{22}&B_{32}&B_{42}&B_{03}&B_{13}&B_{23}&B_{33}&B_{43}&B_{04}&B_{14}&B_{24}&B_{34}&B_{44}& B
\\[.5ex]
\hline\\[.5ex]
       0&88&0&0&0&0&0&45&0&0&0&0&0&26&0&0&0&0&0&45&88&0&0&0&0&292\\
       0&0&30&0&0&0&0&0&26&0&0&0&0&0&26&30&0&0&0&0&0&88&0&0&0&200\\
       0&0&0&16&0&0&0&0&0&26&16&0&0&0&0&0&30&0&0&0&0&0&30&0&0&118\\
       0&0&0&0&16&16&0&0&0&0&0&16&0&0&0&0&0&6&0&0&0&0&0&16&0&70\\
       6&0&0&0&0&0&16&0&0&0&0&0&6&0&0&0&0&0&6&0&0&0&0&0&16&50\\
       0&6&0&0&0&0&0&6&0&0&0&0&0&6&0&0&0&0&0&6&6&0&0&0&0&30\\
       0&0&6&0&0&0&0&0&6&0&0&0&0&0&6&6&0&0&0&0&0&6&0&0&0&30\\
       0&0&0&6&0&0&0&0&0&6&6&0&0&0&0&0&6&0&0&0&0&0&6&0&0&30
\end{smallmatrix}\right|\]
As in \cite{oeding2022} we can make a Carter diagram recording the numbers of coincidences (recorded by the labels on the edges) in the indices on the basic 4-forms, which is the following for the candidate above:
\[
\xymatrix{
\ccircle{1368} \ar@{=}_{68}[dd] \ar@{-}^1[rr]&&\ccircle{1579}\ar@{=}^{59}[dd]\\
& \ccircle{0237}\ar@{-}_3[ul] \ar@{-}^7[ur] \ar@{-}_2[dr] \ar@{-}^0[dl]
\\
\ccircle{0468}\ar@{-}_4[rr] && \ccircle{2459}
}
\]
The Carter diagram may suggest ways of generalizing and constructing new interesting examples of forms that are not ad-nilpotent.

\section{Applications to quantum information}
The classical problem of classifying and distinguishing $G$-orbits on a module $M$ has regained popularity in the context of quantum information theory in the past twenty years. In this section, we explain the connection and how our Jordan decomposition can be used as an efficient tool to study entanglement. While the examples we consider have been addressed before, our point of departure is that we handle all these by a single construction (an adjoint operator) and straightforward computations (matrix ranks and eigenvalues).
A standard reference for quantum information is \cite{nielsen2002quantum}. For an introduction for mathematicians, see \cite{Landsberg2019}, and for an algebraic geometry perspective on entanglement classification, see \cite{holweck_entanglement}.
\subsection{General theory}
In quantum computation, information is encoded in quantum states. More precisely, the quantum analog of the classical bit $\{0,1\}$ in information theory is a quantum-bit or qubit. In Dirac's notation, a qubit reads
\begin{equation}
 \ket{\psi}=\alpha\ket{0}+\beta\ket{1}=\alpha\begin{pmatrix} 
 1 \\ 0
 \end{pmatrix}+\beta\begin{pmatrix}
 0\\1
 \end{pmatrix}=\begin{pmatrix}
 \alpha\\
 \beta
 \end{pmatrix}\in \CC^2, \text{with }|\alpha|^2+|\beta|^2=1.
\end{equation}
Mathematically, a qubit is just a normalized vector in $\CC^2$. The idea is that with the resources of quantum physics, the information could be given by a classical state of the system, state $\ket{0}$ or $\ket{1}$, but also in a superposition of the basis states. A system made of $n$ qubits will be a normalized tensor $\ket{\psi}\in (\CC^2)^{\otimes n}$. Normalized tensors of rank $1$ are said to be {\em separable} while all other states are said to be \emph{entangled}. From a physical point of view, a separable state, $\psi=\psi_1\otimes\psi_2\otimes \dots \otimes \psi_n$, can be fully described from the knowledge of one of its components, while entangled states share non-classical correlations among their different components. This property of entanglement is a non-classical resource in quantum information that can be exploited to perform non-classical tasks (quantum teleportation, super-dense coding). It is also considered a resource that could explain the speed-up of some quantum algorithms over their classical counterparts. 

There are different frameworks in which to study the entanglement classification problem. One such is to consider equivalence of quantum states up to Stochastic Local Operations with Classical Communication (SLOCC). The quantum physics operations corresponding to the SLOCC group include local unitary transformations, measurements and classical coordinations. This set of operations comprising the SLOOC group is the group of local invertible transformations \cite{dur2000}. The following cases are of interest to the quantum community. We denote by $M$ the module (also called a Hilbert space in quantum physics) where the quantum states live, and $G$ the corresponding SLOCC group:
\begin{enumerate}
 \item $M=(\CC^2)^{\otimes n}$, the Hilbert space of $n$-qubit quantum states, with $G=\SL_2(\CC)^{\times n}$. In the projective space $\PP M$, the variety of separable states is the Segre variety $X=\PP^1\times\dots\times \PP^1\subset \PP M$.
 \item $M=S^k \CC^n$, the Hilbert space of $k$-symmetric (bosonic) quantum states, with $G=\SL_n(\CC)$. Here, a single boson in a normalized vector of $\CC^n$. In the projective space $\PP M$, the variety of separable states is the Veronese variety $X=v_k(\PP^{n-1})\subset \PP(S^k \CC^n)$.
 \item $M=\bw{k} \CC^n$, the Hilbert space of $k$ fermions where a single fermion is a $n$-state particle, i.e., a normalized vector of $\CC^n$, with $G=\SL_n(\CC)$. In the projective space $\PP M$, the variety of separable states is the Grassmann variety $X=\Gr(k,n)\subset \PP(\bigwedge^k \CC^n)$. \end{enumerate}

\subsection{Application to QI for small numbers of qubits}

\subsubsection{3-qubits} 
A very famous case in quantum information is the three-qubit SLOCC classification, i.e., the $G=\SL_2\CC\times \SL_2\CC\times \SL_2\CC$ orbits of normalized tensors in $\CC^2\otimes\CC^2\otimes \CC^2$. Even though this classification was mathematically known for a long time \cite{lepaige81, GKZ}, the result got a lot of attention in the context of quantum physics as it proved for the first time that quantum states could be genuinely entangled in two non-equivalent ways \cite{dur2000}. The orbit structure of the three-qubit case is similar to other tripartite systems like the three-bosonic qubit ($v_3 \PP^1\subset \PP^4$) or three-fermion with $6$-single particle states ($\Gr(3,6)\subset \PP \bigwedge^3\CC^6$). The rank profile allows us to easily distinguish those orbits because those orbits (up to qubit permutation) have different dimensions; see Example~\ref{ex:g36}.

\subsubsection{4-qubits} Another important classification is the four-qubit classification, i.e., $M=(\CC^2)^{\otimes 4}$ and $G=(\SL_2\CC)^{\times 4}$. Here, the number of orbits is infinite, and a description depends on parameters. A classification was first established by Verstraete et al. \cite{verstraete02} and later corrected by \cite{ChtDjo:NormalFormsTensRanksPureStatesPureQubits}, see also \cite{HolweckLuquePlanat}. A complete and irredundant classification was given by Dietrich et al. \cite{ dietrich2022classification}.

Note that $\sl_8 \oplus \bw{4} \CC^8$ contains $\sl_2^{\oplus 4} \oplus (\CC^2)^{\otimes 4}$ as a subalgebra, and we can utilize the rank profiles in $\sl_8 \oplus \bw{4} \CC^8$ to distinugish orbits.
Though the complete irredundant classification from \cite{dietrich2022classification} is more complicated to demonstrate proof of concept we only handle random elements from the 9-family classification \cite{ChtDjo:NormalFormsTensRanksPureStatesPureQubits}.
We constructed the adjoint operators for random tensors (generic choice of parameters) of each format from the Verstraete classification and made random choices for those parameters for elements in the $9$ families (which are named 1, 2, 3, 6, 9, 10, 12, 14, 16 in \cites{ChtDjo:NormalFormsTensRanksPureStatesPureQubits}),
 and computed their rank profiles, which allows us to distinguish the different families (Table~\ref{tab:ChtDjo}). We also note that while the drops in ranks are related to the Jordan form associated with the 0-eigenspace, the final ranks in the block rank profiles distinguish the semi-simple parts. See \cite{dietrich2022classification} for further discussion on the semi-simple and nilpotent parts of these orbits.
 \begin{table}[htbp]
\begin{tabular}[t]{lll}
Family & Form & Rank Profile 
\\
\hline
\\[-2ex]
1. & \begin{tabular}{l}$\frac{a+d}{2}(\ket{0000} + \ket{1111}) + \frac{a-d}{2}(\ket{0011} + 
\ket{1100})$ \\
 $+\frac{b+c}{2}(\ket{0101}+\ket{1010})
+\frac{b-c}{2}(\ket{0110} + \ket{1001})$ \end{tabular}
& 
$ \left|\begin{smallmatrix}
 B_{00} & B_{01} & B_{10}& B_{11}& B \\[.5ex] 
\hline\\[.5ex]
0&60&60&0&120\\
60&0&0&60&120\\
0&60&60&0&120\\
60&0&0&60&120
\end{smallmatrix}\!\right|$
\\\\
2. & \begin{tabular}{l}  
$\frac{a+c-i}{2}(\ket{0000}+\ket{1111})+\frac{a-c+i}{2}(\ket{0011}+\ket{1100})$ \\ 
$+\frac{b+c+i}{2}(\ket{0101}+\ket{1010})+\frac{b-c-i}{2}(\ket{0110}+\ket{1001})$ \\  $
+\frac{i}{2}(\ket{0001}+\ket{0111}+\ket{1000}+\ket{1110} $ \\  
$-\ket{0010}-\ket{0100}-\ket{1011}-\ket{1101})$ 
\end{tabular}
&
$ \left|\begin{smallmatrix}
 B_{00} & B_{01} & B_{10}& B_{11}& B \\[.5ex] 
\hline\\[.5ex]
0&60&60&0&120\\
59&0&0&60&119\\
0&59&59&0&118\\
59&0&0&59&118\\
0&59&59&0&118
\end{smallmatrix}\!\right|$
\\\\
3. &
\begin{tabular}{l}  $ \frac{a}{2}(\ket{0000}+\ket{1111} +\ket{0011}+\ket{1100}) + \frac{b+1}{2}(\ket{0101}+\ket{1010})$ \\
$ +\frac{b-1}{2}(\ket{0110}+\ket{1001})+\frac{1}{2}(\ket{1101}+\ket{0010}-\ket{0001}-\ket{1110})$ 
\end{tabular}
&
$ \left|\begin{smallmatrix}
 B_{00} & B_{01} & B_{10}& B_{11}& B \\[.5ex] 
\hline\\[.5ex]
0&56&56&0&112\\
52&0&0&54&106\\
0&50&50&0&100\\
50&0&0&50&100\\
0&50&50&0&100
\end{smallmatrix}\!\right|$
\\\\
6. 
&
\begin{tabular}{l}  $\frac{a+b}{2}(\ket{0000}+\ket{1111})+b(\ket{0101}+\ket{1010})+i(\ket{1001}-\ket{0110})$ \\ 
 $+\frac{a-b}{2}(\ket{0011}+\ket{1100})+\frac{1}{2}(\ket{0010}+\ket{0100}+\ket{1011}+\ket{1101}$ \\  
$-\ket{0001}-\ket{0111}-\ket{1000}-\ket{1110})$ 
\end{tabular}
&
$ \left|\begin{smallmatrix}
 B_{00} & B_{01} & B_{10}& B_{11}& B \\[.5ex] 
\hline\\[.5ex]
0&60&60&0&120\\
58&0&0&60&118\\
0&58&58&0&116\\
57&0&0&58&115\\
0&57&57&0&114\\
57&0&0&57&114\\
0&57&57&0&114
\end{smallmatrix}\!\right|$
\\\\
9. 
&\begin{tabular}{l}  
$a(\ket{0000}+\ket{0101}+\ket{1010}+\ket{1111})$ \\ 
 $-2i(\ket{0100}-\ket{1001}-\ket{1110})$
 \end{tabular}
&
$ \left|\begin{smallmatrix}
 B_{00} & B_{01} & B_{10}& B_{11}& B \\[.5ex] 
\hline\\[.5ex]
0&56&56&0&112\\
51&0&0&54&105\\
0&49&49&0&98\\
45&0&0&47&92\\
0&43&43&0&86\\
42&0&0&43&85\\
0&42&42&0&84\\
42&0&0&42&84\\
0&42&42&0&84
\end{smallmatrix}\!\right|$
\\\\
10.&
\begin{tabular}{l}   $ \frac{a+i}{2}(\ket{0000}+\ket{1111} 
+\ket{0011}+\ket{1100})+\frac{a-i+1}{2}(\ket{0101}+\ket{1010})$ \\ 
 $ +\frac{a-i-1}{2}(\ket{0110}+\ket{1001})+\frac{i+1}{2}(\ket{1101}+\ket{0010})$ \\ 
  $+\frac{i-1}{2}(\ket{0001}+\ket{1110}) -\frac{i}{2}(\ket{0100}+\ket{0111}+\ket{1000}+\ket{1011})$ 
  \end{tabular}
&
$ \left|\begin{smallmatrix}
 B_{00} & B_{01} & B_{10}& B_{11}& B \\[.5ex] 
\hline\\[.5ex]
0&48&48&0&96\\
39&0&0&42&81\\
0&33&33&0&66\\
33&0&0&33&66\\
0&33&33&0&66
\end{smallmatrix}\!\right|$
\\\\
12. 
&
\begin{tabular}{l}  $(\ket{0101}-\ket{0110}+\ket{1100}+\ket{1111})+(i+1)(\ket{1001}+\ket{1010})$ \\
 $-i(\ket{0100}+\ket{0111}+\ket{1101}-\ket{1110})$ 
 \end{tabular}
&
$ \left|\begin{smallmatrix}
 B_{00} & B_{01} & B_{10}& B_{11}& B \\[.5ex] 
\hline\\[.5ex]
0&48&48&0&96\\
38&0&0&42&80\\
0&32&32&0&64\\
23&0&0&26&49\\
0&17&17&0&34\\
8&0&0&11&19\\
0&2&2&0&4\\
1&0&0&2&3\\
0&1&1&0&2\\
0&0&0&1&1\\
0&0&0&0&0\\
\end{smallmatrix}\!\right|$
\\\\
14.& 
\begin{tabular}{l}  
$ \frac{i+1}{2}(\ket{0000}+\ket{1111}-\ket{0010}-\ket{1101})$ \\ 
$ +\frac{i-1}{2}(\ket{0001}+\ket{1110}-\ket{0011}-\ket{1100})$ \\ 
$ +\frac{1}{2}(\ket{0100}+\ket{1001}+\ket{1010}+\ket{0111})$ \\ 
$+\frac{1-2i}{2}(\ket{1000}+\ket{0101}+\ket{0110}+\ket{1011})$ 
\end{tabular}
&
$ \left|\begin{smallmatrix}
 B_{00} & B_{01} & B_{10}& B_{11}& B \\[.5ex] 
\hline\\[.5ex]
0&47&47&0&94\\
30&0&0&34&64\\
0&17&17&0&34\\
9&0&0&10&19\\
0&2&2&0&4\\
0&0&0&2&2\\
0&0&0&0&0\\
\end{smallmatrix}\!\right|$
\\\\
16. 
&
\begin{tabular}{l}  $\frac{1}{2}(\ket{0}+\ket{1})\otimes(\ket{000}+\ket{011}+\ket{100}+\ket{111}$ \\
 $ +i(\ket{001}+\ket{010}-\ket{101}-\ket{110}))$ 
 \end{tabular}
&
$ \left|\begin{smallmatrix}
 B_{00} & B_{01} & B_{10}& B_{11}& B \\[.5ex] 
\hline\\[.5ex]
0&33&33&0&66\\
14&0&0&20&34\\
0&1&1&0&2\\
1&0&0&0&1\\
0&0&0&0&0\\
\end{smallmatrix}\!\right|$
\end{tabular}
\caption{ 
$\text{SLOCC}^*$--orbits,  \cite{ChtDjo:NormalFormsTensRanksPureStatesPureQubits}
with (random) $a= 1,b = 2,c = \frac{8}{5},d =\frac{1}{2}$.
}\label{tab:ChtDjo}
\end{table}

Setting all parameters ($a,b,c,d$) in the generic expressions to $0$, one obtains representatives of the $9$ orbits of the null-cone (up to qubit permutation). Here again, those orbits can also be distinguished by the rank profile. It furnishes an efficient and easy way to identify the class of a nilpotent element. An algorithm based on invariants and covariants was published in the quantum information literature to classify four-qubit nilpotent element \cite{HLT14_atlas} and later implemented in the study of Grover's quantum search algorithm and Shor's quantum algorithm \cite{JH19}. Here, because the nilpotent orbits of the four-qubit case have different dimensions, the rank profile can be used as an alternative to distinguish the nilpotent orbits (up to qubit permutation), see Table~\ref{tab:ChtDjo}. 

We computed that the root profile of the adjoint operators in $\sl_2^{\times 4} \oplus (\CC^2)^{\otimes 4}$ (which are $28\times 28$ matrices)
 for the states that are generic semi-simple in \cite{dietrich2022classification}, denoted $G_{abcd}$ (in \cite{verstraete02}) or family 1 \cite{ChtDjo:NormalFormsTensRanksPureStatesPureQubits}, is made of $24$ single roots and $0$ with order $4$. This is a direct consequence of the fact that generic elements $G_{abcd}$ are the elements of $(\CC^2)^{\otimes 4}$ that do not annihilate the $2\times 2\times 2\times 2$ Cayley hyperdeterminant, a degree $24$ invariant polynomial. Recall that Cayley hyperdeterminant is the defining equation of the dual of $X=(\PP^1)^4\subset \PP((\CC^2)^{\otimes 4})$. As shown in \cite{holweck_4qubit2} it is the restriction to $(\CC^2)^{\otimes 4}$ of the adjoint discriminant $\Delta_{\mathfrak{g}}$ for $\mathfrak{g}=\mathfrak{s}\mathfrak{o}(8)$. The adjoint discriminant of a Lie algebra $\mathfrak{g}$ \cite[p.~$4605$]{TevelevJMS} is
\begin{equation}
\Delta_{\mathfrak{g}}(x)=\Delta\left(\frac{1}{t^n}\chi_{\ad_{x}}\right),
\end{equation} 
where $n$ is the rank of the Lie algebra and $\Delta$ is the discriminant with respect to $t$. Here we consider again the algebra $\sl_8 \oplus \bw{4} \CC^8$ which contains $\sl_2^{\oplus 4} \oplus (\CC^2)^{\otimes 4}$. In particular, if one restricts to $x=T\in M=(\CC^2)^{\otimes 4}$, the adjoint root profile should be made of $24$ single roots (and $0$ with order $4$) for tensors $T$ outside the zero locus of the dual variety. In other words, when $\mathfrak{g}\oplus M$ is a Lie algebra, the adjoint root profile allows us to decide if a tensor $T$ belongs to the dual variety of the highest weight vector $G$-orbit of $M$, $X_{G, M}=G.[v]\subset \PP(M)$.
Note that when $\mathfrak{g}\oplus M$ is not a Lie algebra, like in Example~\ref{ex:g36} for $\bw{3}\CC^6$, we still have a root profile, but the connection with the dual of the adjoint variety is not clear at present.

\subsubsection{5 qubits}
For larger quantum systems, few results are known. The $5$-qubit classification is considered intractable because the number of orbits is infinite, as is the number of nilpotent orbits. However, our method already provides a criterion to decide if two SLOCC orbits are distinct. Indeed, for a Hilbert space $M$ with SLOCC group $G$ and corresponding Lie algebra $\mathfrak{g}$, then quantum states $\psi_1,\psi_2\in M$ are not SLOCC equivalent if the rank profiles of $\ad_{\psi_1}$ and $\ad_{\psi_2}$ are distinct.
In \cite{Osterloh06}, the notion of entanglement filter is used to distinguish some five qubit states $M=(\CC^2)^{\otimes 5}$. We list these states and their adjoint rank profiles in Table~\ref{tab:Osterloh}.
The fact that these $5$-qubit quantum states are not SLOCC equivalent was also obtained by \cite{LT05} using covariants. The rank profiles in Table~\ref{tab:Osterloh} lead to the same conclusion.

As already noted, $\sl_{10} \oplus \bw{ 5} \CC^{10}$ does not have a Lie algebra structure, but it does have GJD. 
Notice that $\CC^{10} = (\CC^2) ^{\oplus 5}$, and correspondingly $ (\CC^2)^{\otimes 5}\subset \bw5 \CC^{10}$. 
This set of tensors is realized by grouping basis vectors $\{e_i, e_{i+5}\} = \CC^2_i$. Therefore, we can consider the subalgebra and the Jordan decomposition for the $\ZZ_2$-graded algebra we have constructed without any more effort. 

As seen in Table~\ref{tab:C2o5}, the rank profiles appear to detect border ranks up to 4 but not higher. In \cite{OedingSam}, we proved that the set of border rank 5 tensors is defined by invariants of degrees 6 and 16. The trace power invariants of this adjoint form cannot produce these invariants for border rank 5 since the block structure of the adjoint operator implies that the only powers that can have a non-zero trace are multiples of 4, and the degree 6 invariant cannot be produced this way. This feature is also apparent in the characteristic polynomials in Table~\ref{tab:C2o5}. 

Here is an analogy to the Vinberg cases, which we label as a proposition for later reference even if the idea is still in its initial stages.
\begin{prop}
The algebra $\fa = \sl_{10} \oplus \bw{ 5} \CC^{10}$ has a 10-dimensional Cartan-like subalgebra that lives entirely in $(\CC^2)^{\otimes 5}\subset \bw 5\CC^{10}$. A set of basic almost ad-semisimple elements is the following:
\[
\begin{matrix}
 p_{0,\pm} = e_{1}e_{2}e_{4}e_{6}e_{8} \pm e_{0}e_{3}e_{5}e_{7}e_{9},&
 p_{1,\pm} = e_{0}e_{3}e_{4}e_{6}e_{8} \pm e_{1}e_{2}e_{5}e_{7}e_{9},\\
 p_{2,\pm} = e_{0}e_{2}e_{5}e_{6}e_{8} \pm e_{1}e_{3}e_{4}e_{7}e_{9},&
 p_{3,\pm} = e_{0}e_{2}e_{4}e_{7}e_{8} \pm e_{1}e_{3}e_{5}e_{6}e_{9},\\
 p_{4,\pm} = e_{1}e_{3}e_{5}e_{7}e_{8} \pm e_{0}e_{2}e_{4}e_{6}e_{9}.
\end{matrix}
\]
\end{prop}
\begin{proof}
First, a word of warning on notation. By \defi{Cartan-like} subalgebra, we mean an abelian subalgebra $\mathfrak{h} \subset \fa$ of almost diagonalizable elements.
Since the bracket is commuting on this grade 1 piece of $\fa$, it is not a restriction to ask that this set be abelian. So we asked, instead, if $[a,b]=0$ for any $a$ and $b$ independent elements of our basis of match vectors, but we don't have, for instance, $[a, a] = 0$ necessarily, so the algebra $\mathfrak{h}$ is not nilpotent. We noticed that for basis elements $x$ we have $[x,x] = \sum_{i=1}^5 \pm h_i \neq 0$. 

The almost diagonalizability is that the dimension of the $0$-eigenspace is 249 versus the algebraic multiplicity of 251, but the other 4 eigenspaces have the correct dimension of 25 each. The extra kernel seems to be due to the following facts:
\begin{itemize}
\item $[p_{i,+}, p_{i,-}] = 0$ and $p_{i,-}$ is in the kernel of $(\ad_{p_{i,+}})^2$ but not the kernel of $(\ad_{p_{i,+}})$,
\item $[p_{i,+}, p_{i,+}] = k:= \sum_{i=1}^5 \pm h_i $ is in the kernel of $(\ad_{p_{i,+}})^3$ but not the kernel of $(\ad_{p_{i,+}})^2$.
\end{itemize}

Also, we did not prove that this subalgebra is maximal in $\fa$. We will explain in what sense the algebra we construct is maximal.
Let us consider some elements that have a chance to be basic semisimple elements, by analogy to the $\bw{4}\CC^8$ and the $\bw3\CC^9$ cases, whose basic semisimple elements are governed by certain combinatorial designs (respectively Steiner quadruple and triple systems). The bipartition $\{\{0,1,2,3,4\}, \{5,6,7,8,9\}\}$ indicates a copy of $(\CC^2)^{\otimes 5}$ in $\bw{5}\CC^{10}$. 
The 16 matchings on the bipartite graph associated to that bipartition are the following.
\[\begin{matrix}
\left\{\left\{0,\:2,\:4,\:6,\:8\right\},\:\left\{1,\:3,\:5,\:7,\:9\right\}\right\}, &
 \left\{\left\{0,\:2,\:4,\:6,\:9\right\},\:\left\{1,\:3,\:5,\:7,\:8\right\}\right\},\\
 \left\{\left\{0,\:2,\:4,\:7,\:8\right\},\:\left\{1,\:3,\:5,\:6,\:9\right\}\right\}, &
 \left\{\left\{0,\:2,\:4,\:7,\:9\right\},\:\left\{1,\:3,\:5,\:6,\:8\right\}\right\},\\
 \left\{\left\{0,\:2,\:5,\:6,\:8\right\},\:\left\{1,\:3,\:4,\:7,\:9\right\}\right\}, &
 \left\{\left\{0,\:2,\:5,\:6,\:9\right\},\:\left\{1,\:3,\:4,\:7,\:8\right\}\right\},\\
 \left\{\left\{0,\:2,\:5,\:7,\:8\right\},\:\left\{1,\:3,\:4,\:6,\:9\right\}\right\}, &
 \left\{\left\{0,\:3,\:4,\:6,\:8\right\},\:\left\{1,\:2,\:5,\:7,\:9\right\}\right\},\\
 \left\{\left\{0,\:3,\:4,\:6,\:9\right\},\:\left\{1,\:2,\:5,\:7,\:8\right\}\right\}, &
 \left\{\left\{0,\:3,\:4,\:7,\:8\right\},\:\left\{1,\:2,\:5,\:6,\:9\right\}\right\},\\
 \left\{\left\{0,\:3,\:5,\:6,\:8\right\},\:\left\{1,\:2,\:4,\:7,\:9\right\}\right\}, &
 \left\{\left\{1,\:2,\:4,\:6,\:8\right\},\:\left\{0,\:3,\:5,\:7,\:9\right\}\right\},\\
 \left\{\left\{1,\:2,\:4,\:6,\:9\right\},\:\left\{0,\:3,\:5,\:7,\:8\right\}\right\}, &
 \left\{\left\{1,\:2,\:4,\:7,\:8\right\},\:\left\{0,\:3,\:5,\:6,\:9\right\}\right\},\\
 \left\{\left\{1,\:2,\:5,\:6,\:8\right\},\:\left\{0,\:3,\:4,\:7,\:9\right\}\right\}, &
 \left\{\left\{1,\:3,\:4,\:6,\:8\right\},\:\left\{0,\:2,\:5,\:7,\:9\right\}\right\}.
\end{matrix}
\]
Each of these matchings corresponds to two elements, call them \defi{match vectors}, of our distinguished copy of $(\CC^2)^{\otimes 5} \subset \bw{5}\CC^{10}$, for instance in the first case we have the vectors $e_{0}e_{2}e_{4}e_{6}e_{8}\pm e_{1}e_{3}e_{5}e_{7}e_{9}$. Individually, the rank profiles and characteristic polynomials of match vectors are all identical; see Table~\ref{tab:C2o5ss}. The fact that the adjoint rank profile stabilizes led us to believe that match vectors are not ad-nilpotent and might be ad-semisimple. We checked the geometric multiplicity of the eigenvalues ($4^\text{th}$-roots of unity) found by factoring the characteristic polynomial and found that all 4 non-zero eigenspaces have geometric multiplicity 25, which agrees with their algebraic multiplicity, so the adjoint forms of match vectors are all almost diagonalizable.
\begin{remark}When working in M2, we noted that computations of dimensions of eigenspaces via matrix kernels were problematic working over $\CC$, but when we made the symbolic field extension by hand and worked in exact arithmetic, we could verify that the geometric and algebraic multiplicities of the eigenvalues of our adjoint forms were indeed the same. 
\end{remark}

The fact that the match vectors all have the same adjoint characteristic polynomial means they have the same adjoint eigenvalues counting multiplicities. However, they're not necessarily simultaneously diagonalizable unless they also commute, which seems not to be the case. 

The next step is to attempt to find a maximal subset of these. We checked that a subset of 5 of these forms ($p_{0,+},\ldots,p_{4,+}$, for instance) satisfies $[a,b]=0$ for $a\neq b$, then we tried to add in $p_{i,-}$ for each $p_{i,+}$ in our set. Adding any other basic match vector to the set failed $[a,b]=0$ for $a\neq b$, so in this sense, our set is maximal.
\end{proof}

\begin{table} 
\begin{tabular}{CCCC}
\scalebox{.8}{normal form}& \scalebox{.9}{Rank Profile in $\bw{5}\CC^{10}$ }&\scalebox{.9}{Rank Profile in $(\CC^{2})^{\otimes 5}$} & \scalebox{.9}{$\chi_{T}(t)$, $\chi_{T}(t)_{\mid(\CC^{2})^{\otimes 5}}$}\\ \hline\\
\begin{matrix}
e_{1}e_{2}e_{4}e_{6}e_{8}+e_{0}e_{2}e_{5}e_{6}e_{8}\\
+e_{1}e_{3}e_{4}e_{7}e_{9}+e_{0}e_{3}e_{5}e_{7}e_{9}
\end{matrix}
&
 \left|\begin{smallmatrix}
 B_{00} & B_{01} & B_{10}& B_{11}& B \\[.5ex]
\hline\\[.5ex]
0&72&72&0&144\\
66&0&0&72&138\\
0&66&66&0&132\\
66&0&0&66&132
\end{smallmatrix}\right| 
&
 \left|\begin{smallmatrix}
 B_{00} & B_{01} & B_{10}& B_{11}& B \\[.5ex]
\hline\\[.5ex]
0&10&10&0&20\\
6&0&0&10&16\\
0&6&6&0&12\\
6&0&0&6&12
\end{smallmatrix}\right| 
&
\begin{smallmatrix}
\left(t\right)^{219}\left(t^4-1\right)^{24}\left(t^{4}-4\right)^{9},\\
\left(t\right)^{35}\left(t^{4}-4\right)^{3}
\end{smallmatrix}
\\\\
\begin{matrix}e_{1}e_{2}e_{4}e_{6}e_{8}+e_{0}e_{3}e_{5}e_{7}e_{8} \\
+e_{1}e_{2}e_{4}e_{6}e_{9}+e_{0}e_{3}e_{5}e_{7}e_{9}
\end{matrix}
&
 \left|\begin{smallmatrix}
 B_{00} & B_{01} & B_{10}& B_{11}& B \\[.5ex]
\hline\\[.5ex]
0&51&51&0&102\\
18&0&0&34&52\\
0&1&1&0&2\\
1&0&0&0&1\\
0&0&0&0&0
\end{smallmatrix}\right| 
&
 \left|\begin{smallmatrix}
 B_{00} & B_{01} & B_{10}& B_{11}& B \\[.5ex]
\hline\\[.5ex]
0&11&11&0&22\\
2&0&0&10&12\\
0&1&1&0&2\\
1&0&0&0&1\\
0&0&0&0&0
\end{smallmatrix}\right| 
&
\begin{matrix}
\left(t\right)^{351},&
\left(t\right)^{47}
\end{matrix}
\\\\
2\,e_{1}e_{2}e_{4}e_{6}e_{8}
&
 \left|\begin{smallmatrix}
 B_{00} & B_{01} & B_{10}& B_{11}& B \\[.5ex]
\hline\\[.5ex]
0&26&26&0&52\\
0&0&0&1&1\\
0&0&0&0&0
\end{smallmatrix}\right| 
&
 \left|\begin{smallmatrix}
 B_{00} & B_{01} & B_{10}& B_{11}& B \\[.5ex]
\hline\\[.5ex]
0&6&6&0&12\\
0&0&0&1&1\\
0&0&0&0&0
\end{smallmatrix}\right| 
&
\begin{matrix}
\left(t\right)^{351},&
\left(t\right)^{47}
\end{matrix}
\\\\
\begin{matrix}e_{1}e_{2}e_{4}e_{6}e_{8}+e_{0}e_{3}e_{4}e_{6}e_{8}\\
-e_{1}e_{2}e_{5}e_{7}e_{9}+e_{0}e_{3}e_{5}e_{7}e_{9}
\end{matrix}
&
 \left|\begin{smallmatrix}
 B_{00} & B_{01} & B_{10}& B_{11}& B \\[.5ex]
\hline\\[.5ex]
0&76&76&0&152\\
56&0&0&76&132\\
0&56&56&0&112\\
56&0&0&56&112
\end{smallmatrix}\right| 
&
 \left|\begin{smallmatrix}
 B_{00} & B_{01} & B_{10}& B_{11}& B \\[.5ex]
\hline\\[.5ex]
0&12&12&0&24\\
4&0&0&12&16\\
0&4&4&0&8\\
4&0&0&4&8
\end{smallmatrix}\right| 
&
\begin{smallmatrix}
\left(t\right)^{239}\left(t^4-1\right)^{24}\left(t^{4}-4\right)^{4},\\
\left(t\right)^{39}\left(t^{4}-4\right)^{2}
\end{smallmatrix}
\\\\
\begin{matrix}
e_{1}e_{2}e_{4}e_{6}e_{8}+e_{0}e_{2}e_{4}e_{7}e_{8}\\
-e_{1}e_{3}e_{5}e_{6}e_{9}+e_{0}e_{3}e_{5}e_{7}e_{9}
\end{matrix}
&
 \left|\begin{smallmatrix}
 B_{00} & B_{01} & B_{10}& B_{11}& B \\[.5ex]
\hline\\[.5ex]
0&76&76&0&152\\
56&0&0&76&132\\
0&56&56&0&112\\
56&0&0&56&112
\end{smallmatrix}\right| 
&
 \left|\begin{smallmatrix}
 B_{00} & B_{01} & B_{10}& B_{11}& B \\[.5ex]
\hline\\[.5ex]
0&12&12&0&24\\
4&0&0&12&16\\
0&4&4&0&8\\
4&0&0&4&8
\end{smallmatrix}\right| 
&
\begin{smallmatrix}
\left(t\right)^{239}\left(t^4-1\right)^{24}\left(t^{4}-4\right)^{4},\\
\left(t\right)^{39}\left(t^{4}-4\right)^{2}
\end{smallmatrix}
\\\\
\begin{matrix}
e_{1}e_{2}e_{4}e_{6}e_{8}+e_{1}e_{2}e_{4}e_{7}e_{8}\\
-e_{0}e_{3}e_{5}e_{6}e_{9}+e_{0}e_{3}e_{5}e_{7}e_{9}
\end{matrix}
&
 \left|\begin{smallmatrix}
 B_{00} & B_{01} & B_{10}& B_{11}& B \\[.5ex]
\hline\\[.5ex]
0&51&51&0&102\\
50&0&0&51&101\\
0&50&50&0&100\\
50&0&0&50&100
\end{smallmatrix}\right| 
&
 \left|\begin{smallmatrix}
 B_{00} & B_{01} & B_{10}& B_{11}& B \\[.5ex]
\hline\\[.5ex]
0&11&11&0&22\\
10&0&0&11&21\\
0&10&10&0&20\\
10&0&0&10&20
\end{smallmatrix}\right| 
&
\begin{matrix}
\left(t\right)^{251}\left(t^{4}-4\right)^{25},\\
\left(t\right)^{27}\left(t^{4}-4\right)^{5}
\end{matrix}
\\\\
\end{tabular}
\caption{Adjoint rank profiles and characteristic polynomials for some sums of semisimple-like tensors in $(\CC^2)^{\otimes 5}$.}\label{tab:C2o5ss} 
\end{table}

We are curious about the behavior of the ad-semisimple elements coming from the match vectors and linear combinations of such.
When we add specific pairs, we obtain other forms that seem like semisimple forms and some that are nilpotent. We list these types as rows 2 and 3 in Table \ref{tab:C2o5ss}.
Like in the standard Vinberg theory, the ranks and characteristic polynomials depend on the coefficients we apply when combining semisimple elements. 
For example, the form $a (e_{1}e_{2}e_{4}e_{6}e_{8}+e_{0}e_{3}e_{5}e_{7}e_{8})
+b(e_{1}e_{2}e_{4}e_{6}e_{9}+e_{0}e_{3}e_{5}e_{7}e_{9})$ corresponds to the row of Table~\ref{tab:C2o5ss} the adjoint rank 144 when $a^2-b^2 \neq 0$ and 102 when $a^2-b^2 =0$. This collapse comes from a factorization:
\[
a (e_{1}e_{2}e_{4}e_{6}e_{8}+e_{0}e_{3}e_{5}e_{7}e_{8})
+b(e_{1}e_{2}e_{4}e_{6}e_{9}+e_{0}e_{3}e_{5}e_{7}e_{9})
= e_{1}e_{2}e_{4}e_{6}(a e_{8} + b e_9) +e_{0}e_{3}e_{5}e_{7}(a e_{8} + b e_9)
.\]
This form is equivalent to a semisimple-like form
\[ e_{1}e_{2}e_{4}e_{6} \tilde e_8+e_{0}e_{3}e_{5}e_{7}\tilde e_9,\]
 when $(\tilde e_8) = (a e_{8} + b e_9)$ and $(\tilde e_9) = (a e_{8} + b e_9)$ are linearly independent, i.e. when $\det \left( \begin{smallmatrix}
a & b \\ b & a\end{smallmatrix}\right) \neq 0$. Otherwise, the form is not concise and becomes a point on a restricted chordal variety \cite{BidlemanOeding}:
\[ (e_{1}e_{2}e_{4}e_{6} +e_{0}e_{3}e_{5}e_{7})(\tilde e_8).\]
Like in the Vinberg situations, the non-concise tensors are nilpotent.

We can make linear combinations of semisimple forms with higher rank. For instance, when adding pairs $0,3,4,5$ we obtain the form in row 4 of Table~\ref{tab:C2o5ss}. Indeed, this appears to be a rich theory and indicates a new way to construct tensors of high rank.
\begin{table}
\begin{tabular}{CCCC}
\scalebox{.8}{Tensor Rank }& \scalebox{.9}{Rank Profile in $\bw{5}\CC^{10}$ }&\scalebox{.9}{Rank Profile in $(\CC^{2})^{\otimes 5}$} & \scalebox{.9}{$\chi_{T}(t)_{\mid(\CC^{2})^{\otimes 5}}$}\\ \hline\\
1 &
 \left|\begin{smallmatrix}
 B_{00} & B_{01} & B_{10}& B_{11}& B \\[.5ex]
\hline\\[.5ex]
0&26&26&0&52\\
0&0&0&1&1\\
0&0&0&0&0
\end{smallmatrix}\right| &
 \left|\begin{smallmatrix}
 B_{00} & B_{01} & B_{10}& B_{11}& B \\[.5ex]
\hline\\[.5ex]
0&6&6&0&12\\
0&0&0&1&1\\
0&0&0&0&0
\end{smallmatrix}\right| 
& t^{47} 
\\[5ex]
2& 
 \left|\begin{smallmatrix}
 B_{00} & B_{01} & B_{10}& B_{11}& B \\[.5ex]
\hline\\[.5ex]
0&51&51&0&102\\
50&0&0&51&101\\
0&50&50&0&100
\end{smallmatrix}\right| &
 \left|\begin{smallmatrix}
 B_{00} & B_{01} & B_{10}& B_{11}& B \\[.5ex]
\hline\\[.5ex]
0&11&11&0&22\\
10&0&0&11&21\\
0&10&10&0&20
\end{smallmatrix}\right|
&\begin{smallmatrix}
\left(t\right)^{27}\left(t^{4}-a^2\right)^{5}
\end{smallmatrix}
\\[5ex]
3 & 
 \left|\begin{smallmatrix}
 B_{00} & B_{01} & B_{10}& B_{11}& B \\[.5ex]
\hline\\[.5ex]
0&75&75&0&150\\
50&0&0&75&125\\
0&50&50&0&100\\
\end{smallmatrix}\right| &
 \left|\begin{smallmatrix}
 B_{00} & B_{01} & B_{10}& B_{11}& B \\[.5ex]
\hline\\[.5ex]
0&15&15&0&30\\
10&0&0&15&25\\
0&10&10&0&20
\end{smallmatrix}\right|
&\begin{smallmatrix}
 \left(t\right)^{27}\left(t^{4}-a\right)^{5}
 \end{smallmatrix}
\\[5ex]
\geq 4 & 
 \left|\begin{smallmatrix}
 B_{00} & B_{01} & B_{10}& B_{11}& B \\[.5ex]
\hline\\[.5ex]
0&95&95&0&190\\
90&0&0&95&185\\
0&90&90&0&180
\end{smallmatrix}\right|
&
 \left|\begin{smallmatrix}
 B_{00} & B_{01} & B_{10}& B_{11}& B \\[.5ex]
\hline\\[.5ex]
0&15&15&0&30\\
10&0&0&15&25\\
0&10&10&0&20
\end{smallmatrix}\right|
&
\begin{smallmatrix}\left(t\right)^{27}
\left(t^{4}-a\right)\left(t^{4}-b\right)\left(t^{4}-c\right)\left(t^{4}-d\right)\left(t^{4}-e\right)
\end{smallmatrix}\\\\
\end{tabular}
\caption{Adjoint rank profiles and characteristic polynomials for random tensors of each tensor rank in $(\CC^2)^{\otimes 5}$ ($a,b,\ldots, e$ denote random rational numbers).}\label{tab:C2o5}
\end{table}

In Table~\ref{tab:Osterloh} we report our computations for the examples in Osterloh \cite{Osterloh06}, which are the following in the physics notation:
\begin{eqnarray*}
\ket{\Psi_2}&=&\frac{1}{\sqrt{2}}(\ket{00000}+\ket{11111}), \\ 
 \ket{\Psi_4}&=&\frac{1}{2}(\ket{11111}+\ket{11100}+\ket{00010}+\ket{00001}),
\\
\ket{\Psi_5}&=&\frac{1}{\sqrt{6}}(\sqrt{2}\ket{11111}+\ket{11000}+\ket{00100}+\ket{00010}+\ket{00001}), \\
\ket{\Psi_6}&=&\frac{1}{2\sqrt{2}}(\sqrt{3}\ket{11111}+\ket{10000}+\ket{01000}+\ket{00100}+\ket{00010}+\ket{00001}).
\end{eqnarray*}
Our computations show that the single construction of the adjoint operators and their ranks and characteristic polynomials are enough to distinguish these orbits.
\begin{table}
\begin{tabular}{CCCC}
\scalebox{.9}{Normal Form in $\bw{5}\CC^{10}$ }& \scalebox{.9}{Rank Profile in $\bw{5}\CC^{10}$ }&\scalebox{.9}{Rank Profile in $(\CC^{2})^{\otimes 5}$} & \scalebox{.9}{$\chi_{T}(t)_{\mid(\CC^{2})^{\otimes 5}}$}\\ \hline\\
\begin{smallmatrix}
\Psi_2 &=& {e}_{0}{e}_{2}{e}_{4}{e}_{6}{e}_{8} \\
&&+{e}_{1}{e}_{3}{e}_{5}{e}_{7}{e}_{9} 
\end{smallmatrix}
&
\left|\begin{smallmatrix}
 B_{00} & B_{01} & B_{10}& B_{11}& B \\[.5ex]
\hline\\[.5ex]
0&51&51&0&102\\
50&0&0&51&101\\
0&50&50&0&100
\end{smallmatrix}\right|
&
\left|\begin{smallmatrix}
 B_{00} & B_{01} & B_{10}& B_{11}& B \\[.5ex]
\hline\\[.5ex]
0&11&11&0&22\\
10&0&0&11&21\\
0&10&10&0&20\end{smallmatrix}\right|
&t^{27}\left({t^{4}-1}\right)^{2}
\\[4ex]
\begin{smallmatrix}
\Psi_4 &=& 
e_{1}e_{3}e_{5}e_{6}e_{8}\\
&&+e_{0}e_{2}e_{4}e_{7}e_{8}\\
&&+e_{0}e_{2}e_{4}e_{6}e_{9}\\
&&+e_{1}e_{3}e_{5}e_{7}e_{9}
\end{smallmatrix} &
\left|\begin{smallmatrix}
 B_{00} & B_{01} & B_{10}& B_{11}& B \\[.5ex]
\hline\\[.5ex]
0&76&76&0&152\\
56&0&0&76&132\\
0&56&56&0&112\end{smallmatrix}\right|
&
\left|\begin{smallmatrix}
 B_{00} & B_{01} & B_{10}& B_{11}& B \\[.5ex]
\hline\\[.5ex]
0&12&12&0&24\\
4&0&0&12&16\\
0&4&4&0&8\end{smallmatrix}\right|
& \begin{smallmatrix}t^{39}\left({t^{4}-2}\right)^{2}\end{smallmatrix}
\\[4ex] 
\begin{smallmatrix}
\Psi_5 &=& \sqrt2 e_{1}e_{3}e_{5}e_{7}e_{9}\\
&&+e_{1}e_{3}e_{4}e_{6}e_{8}\\
&&+e_{0}e_{2}e_{5}e_{6}e_{8}\\
&&+e_{0}e_{2}e_{4}e_{7}e_{8}\\
&&+e_{0}e_{2}e_{4}e_{6}e_{9}
\end{smallmatrix}
&
\left|\begin{smallmatrix}
 B_{00} & B_{01} & B_{10}& B_{11}& B \\[.5ex]
\hline\\[.5ex]
 0&84&84&0&168\\
 42&0&0&84&126\\
 0&42&42&0&84\\
 9&0&0&42&51\\
 0&9&9&0&18\\
 0&0&0&9&9\\
 0&0&0&0&0\end{smallmatrix} \right|
&
\left|\begin{smallmatrix}
 B_{00} & B_{01} & B_{10}& B_{11}& B \\[.5ex]
\hline\\[.5ex]
 0&14&14&0&28\\
 6&0&0&14&20\\
 0&6&6&0&12\\
 3&0&0&6&9\\
 0&3&3&0&6\\
 0&0&0&3&3\\
 0&0&0&0&0
 \end{smallmatrix}\right|
&t^{47}
\\[6ex]
\begin{smallmatrix}
\Psi_6 &=& \sqrt3 e_{1}e_{3}e_{5}e_{7}e_{9}\\
&& +e_{1}e_{2}e_{4}e_{6}e_{8}\\
&&+e_{0}e_{3}e_{4}e_{6}e_{8}\\
&&+e_{0}e_{2}e_{5}e_{6}e_{8}\\
&&+e_{0}e_{2}e_{4}e_{7}e_{8}\\
&&+e_{0}e_{2}e_{4}e_{6}e_{9}
\end{smallmatrix}
&
\left|\begin{smallmatrix}
 B_{00} & B_{01} & B_{10}& B_{11}& B \\[.5ex]
\hline\\[.5ex]
 0&75&75&0&150\\
 50&0&0&75&125\\
 0&50&50&0&100\\
 25&0&0&50&75\\
 0&25&25&0&50\\
 0&0&0&25&25\\
 0&0&0&0&0\end{smallmatrix}\right|
&
\left|\begin{smallmatrix}
 B_{00} & B_{01} & B_{10}& B_{11}& B \\[.5ex]
\hline\\[.5ex]
 0&15&15&0&30\\
 10&0&0&15&25\\
 0&10&10&0&20\\
 5&0&0&10&15\\
 0&5&5&0&10\\
 0&0&0&5&5\\
 0&0&0&0&0
 \end{smallmatrix}\right|
 &t^{47} \\\\
 \end{tabular}
 \caption{Adjoint rank profiles and characteristic polynomials for the examples for 5 qubits from \cite{Osterloh06}.}\label{tab:Osterloh}
\end{table}

\begin{example}[$5\times 5\times 5$ tensors]
Now we consider tensors in $\bw3 W$, with $\dim W = 15$, and we do computations over $\ZZ/1000000007$ to avoid coefficient explosion for computing ranks of matrices. In particular, we're interested in tensors in $V_1\otimes V_2 \otimes V_3 \subset \bw3 W$, with $\dim V_j= 5$. Respectively we take the bases $\{e_{0+5*(j-1)},\ldots,e_{4 + 5*(j-1)}\}$ of $V_j$ for $j=1,2,3$.

The rank profiles of the adjoint operators for forms of each rank are recorded in Tables~\ref{tab:W3C15} (for normal forms for low rank) and \ref{tab:W3C15b} (for randomized forms of a given rank).

\begin{table}
\begin{tabular}{l}
$e_{0}e_{5}e_{10}$ 
\\
$\left|\begin{smallmatrix}
B_{00}&B_{10}&B_{20}&B_{30} & B_{40}&B_{01}&B_{11}&B_{21}&B_{31}&B_{41}&B_{02}&B_{12}&B_{22}&B_{32}&B_{42}&B_{03}&B_{13}&B_{23}&B_{33}&B_{43}&B_{04}&B_{14}&B_{24}&B_{34}&B_{44}&B \\
\hline\\[.5ex]
0&37&0&0&0&0&0&220&0&0&0&0&0&924&0&0&0&0&0&220&37&0&0&0&0&1\,438\\
0&0&0&0&0&0&0&0&0&0&0&0&0&0&0&0&0&0&0&0&0&1&0&0&0&1\\
0&0&0&0&0&0&0&0&0&0&0&0&0&0&0&0&0&0&0&0&0&0&0&0&0&0\\
\end{smallmatrix}\right|$
\\
 -- used 229.807 seconds to compute the block ranks
\\[1ex]

\hline
$e_{0}e_{5}e_{10}+e_{1}e_{6}e_{11}$
\\
$\left|\begin{smallmatrix}
B_{00}&B_{10}&B_{20}&B_{30} & B_{40}&B_{01}&B_{11}&B_{21}&B_{31}&B_{41}&B_{02}&B_{12}&B_{22}&B_{32}&B_{42}&B_{03}&B_{13}&B_{23}&B_{33}&B_{43}&B_{04}&B_{14}&B_{24}&B_{34}&B_{44}&B \\
\hline\\[.5ex]
0&74&0&0&0&0&0&355&0&0&0&0&0&1\,680&0&0&0&0&0&355&74&0&0&0&0&2\,538\\
0&0&55&0&0&0&0&0&0&0&0&0&0&0&0&55&0&0&0&0&0&20&0&0&0&130\\
0&0&0&0&0&0&0&0&0&0&0&0&0&0&0&0&1&0&0&0&0&0&1&0&0&2\\
0&0&0&0&0&0&0&0&0&0&0&0&0&0&0&0&0&1&0&0&0&0&0&0&0&1\\
0&0&0&0&0&0&0&0&0&0&0&0&0&0&0&0&0&0&0&0&0&0&0&0&0&0
\end{smallmatrix}\right|$
\\
 -- used 265.469 seconds to compute the block ranks
\\[1ex]
\hline
$e_{0}e_{5}e_{10}+e_{1}e_{6}e_{11}+e_{2}e_{7}e_{12}$
\\
$\left|\begin{smallmatrix}
B_{00}&B_{10}&B_{20}&B_{30} & B_{40}&B_{01}&B_{11}&B_{21}&B_{31}&B_{41}&B_{02}&B_{12}&B_{22}&B_{32}&B_{42}&B_{03}&B_{13}&B_{23}&B_{33}&B_{43}&B_{04}&B_{14}&B_{24}&B_{34}&B_{44}&B \\
\hline\\[.5ex]
0&111&0&0&0&0&0&427&0&0&0&0&0&2\,310&0&0&0&0&0&427&111&0&0&0&0&3\,386\\
0&0&110&0&0&0&0&0&0&0&0&0&0&0&0&110&0&0&0&0&0&57&0&0&0&277\\
0&0&0&0&0&0&0&0&0&0&0&0&0&0&0&0&56&0&0&0&0&0&56&0&0&112\\
0&0&0&0&0&0&0&0&0&0&0&0&0&0&0&0&0&56&0&0&0&0&0&0&0&56\\
0&0&0&0&0&0&0&0&0&0&0&0&0&0&0&0&0&0&0&0&0&0&0&0&0&0
\end{smallmatrix}\right|$
\\
 -- used 310.79 seconds to compute the block ranks
\\[1ex]
\hline

$e_{0}e_{5}e_{10}+e_{1}e_{6}e_{11}+e_{2}e_{7}e_{12}+e_{3}e_{8}e_{13}$
\\
$\left|\begin{smallmatrix}
B_{00}&B_{10}&B_{20}&B_{30} & B_{40}&B_{01}&B_{11}&B_{21}&B_{31}&B_{41}&B_{02}&B_{12}&B_{22}&B_{32}&B_{42}&B_{03}&B_{13}&B_{23}&B_{33}&B_{43}&B_{04}&B_{14}&B_{24}&B_{34}&B_{44}&B \\
\hline\\[.5ex]
0&148&0&0&0&0&0&454&0&0&0&0&0&2\,850&0&0&0&0&0&454&148&0&0&0&0&4\,054\\
0&0&147&0&0&0&0&0&0&0&0&0&0&0&0&147&0&0&0&0&0&112&0&0&0&406\\
0&0&0&0&0&0&0&0&0&0&0&0&0&0&0&0&111&0&0&0&0&0&111&0&0&222\\
0&0&0&0&0&0&0&0&0&0&0&0&0&0&0&0&0&111&0&0&0&0&0&0&0&111\\
0&0&0&0&0&0&0&0&0&0&0&0&0&0&0&0&0&0&0&0&0&0&0&0&0&0
\end{smallmatrix}\right|$
\\
 -- used 338.107 seconds to compute the block ranks
\\[1ex]
\hline

$e_{0}e_{5}e_{10}+e_{1}e_{6}e_{11}+e_{2}e_{7}e_{12}+e_{3}e_{8}e_{13}+e_{4}e_{9}e_{14}$
\\
$\left|\begin{smallmatrix}
B_{00}&B_{10}&B_{20}&B_{30} & B_{40}&B_{01}&B_{11}&B_{21}&B_{31}&B_{41}&B_{02}&B_{12}&B_{22}&B_{32}&B_{42}&B_{03}&B_{13}&B_{23}&B_{33}&B_{43}&B_{04}&B_{14}&B_{24}&B_{34}&B_{44}&B \\
\hline\\[.5ex]
0&184&0&0&0&0&0&454&0&0&0&0&0&3\,336&0&0&0&0&0&454&184&0&0&0&0&4\,612\\
0&0&184&0&0&0&0&0&0&0&0&0&0&0&0&184&0&0&0&0&0&184&0&0&0&552\\
0&0&0&0&0&0&0&0&0&0&0&0&0&0&0&0&184&0&0&0&0&0&184&0&0&368\\
0&0&0&0&0&0&0&0&0&0&0&0&0&0&0&0&0&184&0&0&0&0&0&0&0&184\\
0&0&0&0&0&0&0&0&0&0&0&0&0&0&0&0&0&0&0&0&0&0&0&0&0&0
\end{smallmatrix}\right|$
\\
 -- used 359.908 seconds to compute the block ranks
\\[1ex]
\hline
\end{tabular}
 \caption{Forms of each rank in $\bw{3}\CC^{15}$ and their adjoint rank profiles in $\fa = \sl_{15} \oplus \bigoplus_k \bw{3k}\CC^{15}$} \label{tab:W3C15}
\end{table}
\begin{table}
\begin{tabular}{l}
\text{random rank 6}
 -- used 142.025 seconds to construct the adjoint form
\\
$\left|\begin{smallmatrix}
B_{00}&B_{10}&B_{20}&B_{30} & B_{40}&B_{01}&B_{11}&B_{21}&B_{31}&B_{41}&B_{02}&B_{12}&B_{22}&B_{32}&B_{42}&B_{03}&B_{13}&B_{23}&B_{33}&B_{43}&B_{04}&B_{14}&B_{24}&B_{34}&B_{44}&B \\
\hline\\[.5ex]
0&216&0&0&0&0&0&454&0&0&0&0&0&3\,796&0&0&0&0&0&454&216&0&0&0&0&5\,136\\
0&0&216&0&0&0&0&0&0&0&0&0&0&0&0&216&0&0&0&0&0&216&0&0&0&648\\
0&0&0&0&0&0&0&0&0&0&0&0&0&0&0&0&216&0&0&0&0&0&216&0&0&432\\
0&0&0&0&0&0&0&0&0&0&0&0&0&0&0&0&0&216&0&0&0&0&0&0&0&216\\
0&0&0&0&0&0&0&0&0&0&0&0&0&0&0&0&0&0&0&0&0&0&0&0&0&0
\end{smallmatrix}\right|$
\\
 -- used 277.98 seconds to compute the block ranks
\\[1ex]
\hline

\text{random rank 7}
 -- used 139.479 seconds to construct the adjoint form
 \\
$\left|\begin{smallmatrix}
B_{00}&B_{10}&B_{20}&B_{30} & B_{40}&B_{01}&B_{11}&B_{21}&B_{31}&B_{41}&B_{02}&B_{12}&B_{22}&B_{32}&B_{42}&B_{03}&B_{13}&B_{23}&B_{33}&B_{43}&B_{04}&B_{14}&B_{24}&B_{34}&B_{44}&B \\
\hline\\[.5ex]
0&222&0&0&0&0&0&454&0&0&0&0&0&4\,252&0&0&0&0&0&454&222&0&0&0&0&5\,604\\
0&0&222&0&0&0&0&0&0&0&0&0&0&0&0&222&0&0&0&0&0&222&0&0&0&666\\
0&0&0&0&0&0&0&0&0&0&0&0&0&0&0&0&222&0&0&0&0&0&222&0&0&444\\
0&0&0&0&0&0&0&0&0&0&0&0&0&0&0&0&0&222&0&0&0&0&0&0&0&222\\
0&0&0&0&0&0&0&0&0&0&0&0&0&0&0&0&0&0&0&0&0&0&0&0&0&0
\end{smallmatrix}\right|$
\\
 -- used 298.66 seconds to compute the block ranks
 \\[1ex]
\hline

\text{random rank 8}

 -- used 135.61 seconds to construct the adjoint form
\\
$\left|\begin{smallmatrix}
B_{00}&B_{10}&B_{20}&B_{30} & B_{40}&B_{01}&B_{11}&B_{21}&B_{31}&B_{41}&B_{02}&B_{12}&B_{22}&B_{32}&B_{42}&B_{03}&B_{13}&B_{23}&B_{33}&B_{43}&B_{04}&B_{14}&B_{24}&B_{34}&B_{44}&B \\
\hline\\[.5ex]
0&222&0&0&0&0&0&454&0&0&0&0&0&4\,464&0&0&0&0&0&454&222&0&0&0&0&5\,816\\
0&0&222&0&0&0&0&0&0&0&0&0&0&0&0&222&0&0&0&0&0&222&0&0&0&666\\
0&0&0&0&0&0&0&0&0&0&0&0&0&0&0&0&222&0&0&0&0&0&222&0&0&444\\
0&0&0&0&0&0&0&0&0&0&0&0&0&0&0&0&0&222&0&0&0&0&0&0&0&222\\
0&0&0&0&0&0&0&0&0&0&0&0&0&0&0&0&0&0&0&0&0&0&0&0&0&0
\end{smallmatrix}\right|$
\\
 -- used 291.446 seconds to compute the block ranks
\\[1ex]
\hline

\text{random rank 9}

 -- used 133.797 seconds to construct the adjoint form
\\
$\left|\begin{smallmatrix}
B_{00}&B_{10}&B_{20}&B_{30} & B_{40}&B_{01}&B_{11}&B_{21}&B_{31}&B_{41}&B_{02}&B_{12}&B_{22}&B_{32}&B_{42}&B_{03}&B_{13}&B_{23}&B_{33}&B_{43}&B_{04}&B_{14}&B_{24}&B_{34}&B_{44}&B \\
\hline\\[.5ex]
0&222&0&0&0&0&0&454&0&0&0&0&0&4\,476&0&0&0&0&0&454&222&0&0&0&0&5\,828\\
0&0&222&0&0&0&0&0&0&0&0&0&0&0&0&222&0&0&0&0&0&222&0&0&0&666\\
0&0&0&0&0&0&0&0&0&0&0&0&0&0&0&0&222&0&0&0&0&0&222&0&0&444\\
0&0&0&0&0&0&0&0&0&0&0&0&0&0&0&0&0&222&0&0&0&0&0&0&0&222\\
0&0&0&0&0&0&0&0&0&0&0&0&0&0&0&0&0&0&0&0&0&0&0&0&0&0
\end{smallmatrix}\right|$
\\
 -- used 293.198 seconds to compute the block ranks
\end{tabular}
 \caption{Randomized forms of each rank in $\bw{3}\CC^{15}$ and their adjoint rank profiles in $\fa = \sl_{15} \oplus \bigoplus_k \bw{3k}\CC^{15}$} \label{tab:W3C15b}
\end{table}

For rank 10 and up, we obtain the same adjoint rank profile as for rank 9, making rank 9 the largest rank that these rank profiles can detect. We also note that for rank 7 and larger, the only block whose rank changes is $B_{32}$. 

Landsberg points out \cite{landsberg2015nontriviality} that the determinant of a particular exterior flattening gives non-trivial equations for border rank 8, so ours is not the first tool that we use can distinguish border rank 8 from 9 (the generic border rank). However, we remark that in the case of exterior flattenings, one has to construct each flattening separately and check for non-trivial equations, the adjoint operator in the extension of the exterior algebra contains all of these exterior flattenings in a single construction. 
The fact that this case is still within our computational range tells us we can use the adjoint operators to study this type of tensor further.
\end{example}

\section*{Acknowledgements}
Holweck thanks Pr. Ash Abebe and his staff of the AU Mathematics and Statistics department for making his visiting year in Auburn possible as well as Pr. Ghislain Montavon from UTBM for his support in this project. Holweck was partially supported by the Conseil R\'egional de Bourgogne Franche-Comt\'e (GeoQuant project). Oeding and Holweck acknowledge partial support from the Thomas Jefferson Foundation. Oeding acknowledges partial support from the University of Bourgogne for summer salary that supported part of this work.
We also thank Laurent Manivel, Sasha Elashvili, Mamuka Jibladze, Skip Garibaldi, and Ian Tan for helpful discussions. An anonymous referee's remarks provided valuable insight for a revision of this work.

Statement: During the preparation of this work the authors used Grammarly in order to identify and correct grammatical issues. After using this tool/service, the authors reviewed and edited the content as needed and take full responsibility for the content of the publication.
\FloatBarrier
\newcommand{\arxiv}[1]{\href{http://arxiv.org/abs/#1}{{\tt arXiv:#1}}}
\bibliography{/Users/lao0004/Library/CloudStorage/Box-Box/texFiles/main_bibfile.bib}

\def\Dbar{\leavevmode\lower.6ex\hbox to 0pt{\hskip-.23ex \accent"16\hss}D}
  \def\cprime{$'$}
\begin{bibdiv}
\begin{biblist}

\bib{Antonyan}{article}{
      author={Antonyan, L.~V.},
       title={Classification of four-vectors of an eight-dimensional space},
        date={1981},
        ISSN={0204-3165},
     journal={Trudy Sem. Vektor. Tenzor. Anal.},
      number={20},
       pages={144\ndash 161},
      review={\MR{622013}},
}

\bib{bari2022structure}{article}{
      author={Bari, Kashif~K},
       title={On the structure tensor of sln},
        date={2022},
     journal={Linear Algebra and its Applications},
      volume={653},
       pages={266\ndash 286},
        note={\arxiv{2105.08171}},
}

\bib{benedetti2023hecke}{article}{
      author={Benedetti, Vladimiro},
      author={Bolognesi, Michele},
      author={Faenzi, Daniele},
      author={Manivel, Laurent},
       title={Hecke cycles on moduli of vector bundles and orbital degeneracy
  loci},
        date={2023},
     journal={arXiv preprint arXiv:2310.06482},
}

\bib{bernardara2024even}{article}{
      author={Bernardara, Marcello},
      author={Fatighenti, Enrico},
      author={Kapustka, Grzegorz},
      author={Kapustka, Micha{\l}},
      author={Manivel, Laurent},
      author={Mongardi, Giovanni},
      author={Tanturri, Fabio},
       title={Even nodal surfaces of k3 type},
        date={2024},
     journal={arXiv preprint arXiv:2402.08528},
}

\bib{bernardara2021nested}{article}{
      author={Bernardara, Marcello},
      author={Fatighenti, Enrico},
      author={Manivel, Laurent},
       title={Nested varieties of k3 type},
        date={2021},
     journal={Journal de l’{\'E}cole polytechnique—Math{\'e}matiques},
      volume={8},
       pages={733\ndash 778},
}

\bib{BidlemanOeding}{article}{
      author={Bidleman, Dalton},
      author={Oeding, Luke},
       title={Restricted secant varieties of {G}rassmannians},
        date={accepted 2023},
     journal={Collectanea Mathematica},
        note={\arxiv{2211.01469}},
}

\bib{BremnerHuOeding}{article}{
      author={{Bremner}, M.},
      author={{Hu}, J.},
      author={{Oeding}, L.},
       title={{The $3 \times 3 \times 3$ hyperdeterminant as a polynomial in
  the fundamental invariants for $SL_3(\mathbb{C}) \times SL_3(\mathbb{C})
  \times SL_3(\mathbb{C})$}},
    language={English},
        date={2014},
        ISSN={1661-8270},
     journal={Mathematics in Computer Science},
       pages={1\ndash 10},
         url={http://dx.doi.org/10.1007/s11786-014-0186-9},
        note={\arxiv{1310.3257}},
}

\bib{CartwrightSturmfels2011}{article}{
      author={Cartwright, D.},
      author={Sturmfels, B.},
       title={The number of eigenvalues of a tensor},
        date={2013},
     journal={Linear Algebra Appl.},
      volume={438},
      number={2},
       pages={942\ndash 952},
         url={http://dx.doi.org/10.1016/j.laa.2011.05.040},
}

\bib{ChtDjo:NormalFormsTensRanksPureStatesPureQubits}{incollection}{
      author={Chterental, O.},
      author={Djokovic, D.},
       title={{Normal Forms and Tensor Ranks of Pure States of Four Qubits}},
        date={2007},
   booktitle={{Linear Algebra Research Advances}},
      editor={Ling, G.~D.},
   publisher={Nova Science Publishers Inc.},
       pages={133–167},
        note={arXiv:quant-ph/0612184},
}

\bib{MR3841899}{article}{
      author={Cui, Lu-Bin},
      author={Li, Ming-Hui},
       title={A note on the three-way generalization of the {J}ordan canonical
  form},
        date={2018},
     journal={Open Math.},
      volume={16},
      number={1},
       pages={897\ndash 912},
         url={https://doi.org/10.1515/math-2018-0078},
      review={\MR{3841899}},
}

\bib{Kac85}{article}{
      author={Dadok, J.},
      author={Kac, V.},
       title={Polar representations},
        date={1985},
     journal={J. Algebra},
      volume={92},
      number={2},
       pages={504\ndash 524},
}

\bib{dietrich2022classification}{article}{
      author={Dietrich, Heiko},
      author={De~Graaf, Willem~A},
      author={Marrani, Alessio},
      author={Origlia, Marcos},
       title={Classification of four qubit states and their stabilisers under
  slocc operations},
        date={2022},
     journal={Journal of Physics A: Mathematical and Theoretical},
        note={\arxiv{2111.05488}},
}

\bib{dur2000}{article}{
      author={D{\"u}r, Wolfgang},
      author={Vidal, Guifre},
      author={Cirac, J~Ignacio},
       title={Three qubits can be entangled in two inequivalent ways},
        date={2000},
     journal={Physical Review A},
      volume={62},
      number={6},
       pages={062314},
}

\bib{dynkin2000selected}{book}{
      author={Dynkin, E.B.},
      author={Yushkevich, A.A.},
      author={Yushkevich, A.A.},
      author={Seitz, G.M.},
      author={Onishchik, A.L.},
       title={{Selected Papers of E. B. Dynkin with Commentary}},
      series={CWorks / American Mathematical Society},
   publisher={American Mathematical Society},
        date={2000},
        ISBN={9780821810651},
         url={https://books.google.com/books?id=D9ZF5O\_JH2gC},
}

\bib{dynkin1960semisimple}{article}{
      author={Dynkin, Evgeni{\u\i}~Borisovich},
       title={{Semisimple subalgebras of semisimple Lie algebras}},
        date={1960},
     journal={Amer. Math. Soc. Transl. Ser. 2},
      volume={6},
       pages={111\ndash 244},
}

\bib{FultonHarris}{book}{
      author={Fulton, W.},
      author={Harris, J.},
       title={Representation theory, a first course},
      series={Graduate Texts in Mathematics},
   publisher={New York: Springer-Verlag},
        date={1991},
      volume={129},
        ISBN={0-387-97527-6; 0-387-97495-4},
}

\bib{galkazka2017vector}{article}{
      author={Ga{\l}{\k{a}}zka, Maciej},
       title={Vector bundles give equations of cactus varieties},
        date={2017},
     journal={Linear Algebra and its Applications},
      volume={521},
       pages={254\ndash 262},
}

\bib{GattiViniberghi}{article}{
      author={Gatti, V.},
      author={Viniberghi, E.},
       title={Spinors of {$13$}-dimensional space},
        date={1978},
        ISSN={0001-8708},
     journal={Adv. in Math.},
      volume={30},
      number={2},
       pages={137\ndash 155},
         url={https://doi.org/10.1016/0001-8708(78)90034-8},
      review={\MR{513846}},
}

\bib{GKZ}{book}{
      author={Gelfand, I.~M.},
      author={Kapranov, M.~M.},
      author={Zelevinsky, A.~V.},
       title={Discriminants, resultants, and multidimensional determinants},
      series={Mathematics: Theory \& Applications},
   publisher={Boston: Birkh\"auser},
     address={Boston, MA},
        date={1994},
        ISBN={0-8176-3660-9},
}

\bib{Gnang2011}{article}{
      author={Gnang, Edinah~K.},
      author={Elgammal, Ahmed},
      author={Retakh, Vladimir},
       title={A spectral theory for tensors},
    language={eng},
        date={(2011)7},
     journal={Annales de la faculté des sciences de Toulouse Mathématiques},
      volume={20},
      number={4},
       pages={801\ndash 841},
         url={http://eudml.org/doc/219706},
}

\bib{M2}{misc}{
      author={Grayson, Daniel~R.},
      author={Stillman, Michael~E.},
       title={Macaulay2, a software system for research in algebraic geometry},
         how={Available at \url{http://www.math.uiuc.edu/Macaulay2/}},
}

\bib{GurevichBook}{book}{
      author={Gurevich, G.~B.},
       title={Foundations of the theory of algebraic invariants},
      series={Translated by J. R. M. Radok and A. J.M. Spencer},
   publisher={P. Noordhoff Ltd.},
     address={Groningen},
        date={1964},
}

\bib{hitchin2001stable}{article}{
      author={Hitchin, Nigel},
       title={Stable forms and special metrics},
        date={2001},
     journal={Contemporary mathematics},
      volume={288},
       pages={70\ndash 89},
}

\bib{HolweckOedingE8}{article}{
      author={{Holweck}, F.},
      author={{Oeding}, L.},
       title={{Hyperdeterminants from the E8 Discriminant}},
        date={2022},
        ISSN={0021-8693},
     journal={J. Algebra},
  url={https://www.sciencedirect.com/science/article/pii/S0021869321004919},
        note={\arxiv{1810.05857}},
}

\bib{HolweckLuquePlanat}{article}{
      author={Holweck, Fr{\'e}d{\'e}ric},
      author={Luque, Jean-Gabriel},
      author={Planat, Michel},
       title={{Singularity of type D4 arising from four-qubit systems}},
        date={2014},
     journal={Journal of Physics A: Mathematical and Theoretical},
      volume={47},
      number={13},
       pages={135301},
}

\bib{holweck_entanglement}{article}{
      author={Holweck, Fr{\'e}d{\'e}ric},
      author={Luque, Jean-Gabriel},
      author={Thibon, Jean-Yves},
       title={Geometric descriptions of entangled states by auxiliary
  varieties},
        date={2012},
     journal={J. Math. Phys.},
      volume={53},
      number={10},
       pages={102203},
}

\bib{HLT14_atlas}{article}{
      author={Holweck, Fr{\'e}d{\'e}ric},
      author={Luque, Jean-Gabriel},
      author={Thibon, Jean-Yves},
       title={Entanglement of four qubit systems: A geometric atlas with
  polynomial compass i (the finite world)},
        date={2014},
     journal={Journal of Mathematical Physics},
      volume={55},
      number={1},
       pages={012202},
}

\bib{holweck_4qubit2}{article}{
      author={Holweck, Fr{\'e}d{\'e}ric},
      author={Luque, Jean-Gabriel},
      author={Thibon, Jean-Yves},
       title={Entanglement of four-qubit systems: a geometric atlas with
  polynomial compass {II} (the tame world)},
        date={2017},
     journal={J. Math. Phys.},
      volume={58},
      number={2},
       pages={022201},
}

\bib{HuangKim}{article}{
      author={Huang, Huajun},
      author={Kim, Sangjib},
       title={{On Kostant's partial order on hyperbolic elements}},
        date={2010},
     journal={Linear and Multilinear Algebra},
      volume={58},
      number={6},
       pages={783\ndash 788},
      eprint={https://doi.org/10.1080/03081080903062113},
         url={https://doi.org/10.1080/03081080903062113},
}

\bib{JH19}{article}{
      author={Jaffali, Hamza},
      author={Holweck, Fr{\'e}d{\'e}ric},
       title={Quantum entanglement involved in grover’s and shor’s
  algorithms: the four-qubit case},
        date={2019},
     journal={Quantum Information Processing},
      volume={18},
      number={5},
       pages={1\ndash 41},
}

\bib{Kac80}{article}{
      author={Kac, V.},
       title={Some remarks on nilpotent orbits},
        date={1980},
     journal={J. Algebra},
      volume={64},
      number={1},
       pages={190\ndash 213},
}

\bib{kostant1973convexity}{inproceedings}{
      author={Kostant, Bertram},
       title={{On convexity, the Weyl group and the Iwasawa decomposition}},
        date={1973},
   booktitle={Annales scientifiques de l'{\'e}cole normale sup{\'e}rieure},
      volume={6},
       pages={413\ndash 455},
}

\bib{landsberg2015nontriviality}{article}{
      author={Landsberg, JM},
       title={Nontriviality of equations and explicit tensors in $\mathbb{C}^m
  \otimes \mathbb{C}^m\otimes \mathbb{C}^m$ of border rank at least $2m-2$},
        date={2015},
     journal={Journal of Pure and Applied Algebra},
      volume={219},
      number={8},
       pages={3677\ndash 3684},
}

\bib{LM02}{article}{
      author={Landsberg, J.M.},
      author={Manivel, L.},
       title={Construction and classification of complex simple {L}ie algebras
  via projective geometry},
        date={2002},
     journal={Selecta Math. (N.S.)},
      volume={8},
      number={1},
       pages={137\ndash 159},
}

\bib{LanOtt11_Equations}{article}{
      author={Landsberg, J.M.},
      author={Ottaviani, G.},
       title={Equations for secant varieties of {V}eronese and other
  varieties},
    language={English},
        date={2011},
     journal={Ann. Mat. Pura Appl. (4)},
       pages={1\ndash 38},
         url={http://dx.doi.org/10.1007/s10231-011-0238-6},
        note={\arxiv{1111.4567}},
}

\bib{Landsberg2019}{incollection}{
      author={Landsberg, Joseph~M},
       title={A very brief introduction to quantum computing and quantum
  information theory for mathematicians},
        date={2019},
   booktitle={Quantum physics and geometry},
   publisher={Springer},
       pages={5\ndash 41},
}

\bib{lepaige81}{article}{
      author={Le~Paige, Constantin},
       title={Sur la th{\'e}orie des formes binaires {\`a} plusieurs s{\'e}ries
  de variables},
        date={1881},
     journal={Bull. Acad. Roy. Sci. Belgique (3), bf},
      volume={2},
      number={40-53},
       pages={45},
}

\bib{levy2014rationality}{article}{
      author={Levy, Jason},
       title={Rationality and the jordan--gatti--viniberghi decomposition},
        date={2014},
     journal={Canadian Mathematical Bulletin},
      volume={57},
      number={1},
       pages={97\ndash 104},
}

\bib{Lim05_evectors}{inproceedings}{
      author={Lim, L.-H.},
       title={Singular values and eigenvalues of tensors: a variational
  approach},
        date={2005},
   booktitle={Computational advances in multi-sensor adaptive processing, 2005
  1st {IEEE} international workshop on},
       pages={129 \ndash 132},
}

\bib{LT05}{article}{
      author={Luque, Jean-Gabriel},
      author={Thibon, Jean-Yves},
       title={Algebraic invariants of five qubits},
        date={2005},
     journal={Journal of physics A: mathematical and general},
      volume={39},
      number={2},
       pages={371},
}

\bib{MR0007750}{article}{
      author={Morozov, V.~V.},
       title={On a nilpotent element in a semi-simple {L}ie algebra},
        date={1942},
     journal={C. R. (Doklady) Acad. Sci. URSS (N.S.)},
      volume={36},
       pages={83\ndash 86},
      review={\MR{0007750}},
}

\bib{nielsen2002quantum}{misc}{
      author={Nielsen, Michael~A},
      author={Chuang, Isaac},
       title={Quantum computation and quantum information},
   publisher={American Association of Physics Teachers},
        date={2002},
}

\bib{nurmiev}{article}{
      author={Nurmiev, A.~G.},
       title={Orbits and invariants of third-order matrices},
        date={2000},
     journal={Mat. Sb.},
      volume={191},
      number={5},
       pages={101\ndash 108},
         url={http://dx.doi.org/10.1070/SM2000v191n05ABEH000478},
}

\bib{OedOtt13_Waring}{article}{
      author={Oeding, L.},
      author={Ottaviani, G.},
       title={Eigenvectors of tensors and algorithms for {W}aring
  decomposition},
        date={2013},
     journal={J. Symbolic Comput.},
      volume={54},
       pages={9\ndash 35},
         url={http://dx.doi.org/10.1016/j.jsc.2012.11.005},
        note={\arxiv{1103.0203}},
}

\bib{OedingSam}{article}{
      author={Oeding, L.},
      author={Sam, S.~V},
       title={{Equations for the fifth secant variety of Segre products of
  projective spaces}},
        date={2016},
     journal={Experimental Mathematics},
      volume={25},
       pages={94\ndash 99},
        note={\arxiv{1502.00203}},
}

\bib{oeding2022}{article}{
      author={Oeding, Luke},
       title={{A Translation of ``Classification of four-vectors of an
  8-dimensional space,'' by Antonyan, L. V., with an appendix by the
  translator}},
        date={2022},
     journal={Tr. Mosk. Mat. Obs.},
      volume={83},
       pages={269\ndash 296},
         url={\url{http://mi.mathnet.ru/mmo674}},
        note={\arxiv{2205.09741}},
}

\bib{oeding2023exteriorextensions}{article}{
      author={Oeding, Luke},
       title={Exteriorextensions: A package for macaulay2},
        date={2023},
        note={\arxiv{2312.11368}},
}

\bib{Osterloh06}{article}{
      author={Osterloh, Andreas},
      author={Siewert, Jens},
       title={Entanglement monotones and maximally entangled states in
  multipartite qubit systems},
        date={2006},
     journal={International journal of quantum information},
      volume={4},
      number={03},
       pages={531\ndash 540},
}

\bib{Qi05_eigen}{article}{
      author={Qi, L.},
       title={Eigenvalues of a real supersymmetric tensor},
        date={2005},
     journal={J. Symbolic Comput.},
      volume={40},
      number={6},
       pages={1302\ndash 1324},
         url={http://dx.doi.org/10.1016/j.jsc.2005.05.007},
}

\bib{qi2018tensor}{book}{
      author={Qi, Liqun},
      author={Chen, Haibin},
      author={Chen, Yannan},
       title={Tensor eigenvalues and their applications},
      series={Advances in Mechanics and Mathematics},
   publisher={Singapore: Springer},
        date={2018},
      volume={39},
}

\bib{rains2018invariant}{article}{
      author={Rains, Eric},
      author={Sam, Steven},
       title={Invariant theory of $\bigwedge^{3}\mathbb{C}^9$ and genus-2
  curves},
        date={2018},
     journal={Algebra \& Number Theory},
      volume={12},
      number={4},
       pages={935\ndash 957},
}

\bib{sato1977classification}{article}{
      author={Sato, Mikio},
      author={Kimura, Tatsuo},
       title={A classification of irreducible prehomogeneous vector spaces and
  their relative invariants},
        date={1977},
     journal={Nagoya Mathematical Journal},
      volume={65},
       pages={1\ndash 155},
}

\bib{Schouten31}{article}{
      author={Schouten, J.},
       title={Klassifizierung der alternierenden gr\"ozen dritten grades in 7
  dimensionen},
        date={1931},
     journal={Rendiconti del Circolo Matematico di Palermo (1884 - 1940)},
      volume={55},
       pages={137\ndash 156},
         url={http://dx.doi.org/10.1007/BF03016791},
        note={10.1007/BF03016791},
}

\bib{swann1990hyperkahler}{thesis}{
      author={Swann, Andrew},
       title={Hyperk{\"a}hler and quaternionic k{\"a}hler geometry},
        type={Ph.D. Thesis},
        date={1990},
}

\bib{TevelevJMS}{article}{
      author={Tevelev, E.~A.},
       title={Projectively dual varieties},
        date={2003},
        ISSN={1072-3374},
     journal={J. Math. Sci. (N. Y.)},
      volume={117},
      number={6},
       pages={4585\ndash 4732},
         url={https://doi.org/10.1023/A:1025366207448},
        note={Algebraic geometry},
      review={\MR{2027446}},
}

\bib{venturelli2019prehomogeneous}{article}{
      author={Venturelli, Federico},
       title={Prehomogeneous tensor spaces},
        date={2019},
     journal={Linear and Multilinear Algebra},
      volume={67},
      number={3},
       pages={510\ndash 526},
}

\bib{verstraete02}{article}{
      author={Verstraete, Frank},
      author={Dehaene, Jeroen},
      author={De~Moor, Bart},
      author={Verschelde, Henri},
       title={Four qubits can be entangled in nine different ways},
        date={2002},
     journal={Physical Review A},
      volume={65},
      number={5},
       pages={052112},
}

\bib{Vinberg75}{article}{
      author={Vinberg, \`E.~B.},
       title={The classification of nilpotent elements of graded {L}ie
  algebras},
        date={1975},
        ISSN={0002-3264},
     journal={Dokl. Akad. Nauk SSSR},
      volume={225},
      number={4},
       pages={745\ndash 748},
      review={\MR{0506488}},
}

\bib{Vinberg-Elashvili}{article}{
      author={Vinberg, {\`E}.~B.},
      author={{\`E}la{\v{s}}vili, A.~G.},
       title={A classification of the three-vectors of nine-dimensional space},
        date={1978},
     journal={Trudy Sem. Vektor. Tenzor. Anal.},
      volume={18},
       pages={197\ndash 233},
}

\bib{wallach2005hilbert}{article}{
      author={Wallach, Nolan~R},
       title={The {H}ilbert series of measures of entanglement for 4 qubits},
        date={2005},
     journal={Acta Applicandae Mathematicae},
      volume={86},
      number={1-2},
       pages={203\ndash 220},
}

\bib{ye2018fast}{article}{
      author={Ye, Ke},
      author={Lim, Lek-Heng},
       title={{Fast structured matrix computations: tensor rank and Cohn--Umans
  method}},
        date={2018},
     journal={Foundations of Computational Mathematics},
      volume={18},
      number={1},
       pages={45\ndash 95},
}

\end{biblist}
\end{bibdiv}
\newpage
\section*{Appendix}\label{sec:appendix}
In this appendix we provide the trace power invariants for general semisimple elements and the adjoint rank profiles of all orbits in Vinberg and {\`E}la{\v{s}}vili's classification for $\bw3 \CC^9$, \cite{Vinberg-Elashvili} as an ancillary file to our article ``Jordan Decompositions of Tensors.''

The orbits of $\SL_9$ in $\bw{3}\CC^9$ occur in 7 families depending on the form of their semi-simple parts.
Recall the basic semi-simple elements:
\[p_1  = e_{0}e_{1}e_{2}+e_{3}e_{4}e_{5}+e_{6}e_{7}e_{8},  \quad p_2 = e_{0}e_{3}e_{6}+e_{1}e_{4}e_{7}+e_{2}e_{5}e_{8}\]
\[p_3  = e_{1}e_{5}e_{6}+e_{2}e_{3}e_{7}+e_{0}e_{4}e_{8}, \quad p_4 = e_{2}e_{4}e_{6}+e_{0}e_{5}e_{7}+e_{1}e_{3}e_{8}\]

Here is the adjoint rank profile for random elements of each of the different families of semi-simple elements: We note that each of these families have non-zero values for their trace powers in degrees 12, 18, 24, 30, etc.

\subsection*{Family one} Basic semisimple form $\sum_{i=1}^4 \lambda_i p_i$. Block Ranks:
$\left(\begin{smallmatrix}
       0&0&80&80&0&0&0&80&0&240\\
       0&80&0&0&0&80&80&0&0&240\\
       \end{smallmatrix}\right)$

The trace powers on generic semi-simple elements are:

$f_{12} = ( \lambda_{1}^{12}+22\, \lambda_{1}^{6} \lambda_{2}^{6}+ \lambda_{2}^{12}-220\, \lambda_{1}^{6} \lambda_{2}^{3} \lambda_{3}^{3}-220\, \lambda_{1}^{3} \lambda_{2}^{6} \lambda_{3}^{3}+22\, \lambda_{1}^{6} \lambda_{3}^{6}+220\, \lambda_{1}^{3} \lambda_{2}^{3} \lambda_{3}^{6}+22\, \lambda_{2}^{6} \lambda_{3}^{6}+ \lambda_{3}^{12}-220\, \lambda_{1}^{6} \lambda_{2}^{3} \lambda_{4}^{3}+220\, \lambda_{1}^{3} \lambda_{2}^{6} \lambda_{4}^{3}-220\, \lambda_{1}^{6} \lambda_{3}^{3} \lambda_{4}^{3}+220\, \lambda_{2}^{6} \lambda_{3}^{3} \lambda_{4}^{3}-220\, \lambda_{1}^{3} \lambda_{3}^{6} \lambda_{4}^{3}+220\, \lambda_{2}^{3} \lambda_{3}^{6} \lambda_{4}^{3}+22\, \lambda_{1}^{6} \lambda_{4}^{6}-220\, \lambda_{1}^{3} \lambda_{2}^{3} \lambda_{4}^{6}+22\, \lambda_{2}^{6} \lambda_{4}^{6}+220\, \lambda_{1}^{3} \lambda_{3}^{3} \lambda_{4}^{6}+220\, \lambda_{2}^{3} \lambda_{3}^{3} \lambda_{4}^{6}+22\, \lambda_{3}^{6} \lambda_{4}^{6}+ \lambda_{4}^{12})\left(4536\right)$

$f_{18} = ( \lambda_{1}^{18}-17\, \lambda_{1}^{12} \lambda_{2}^{6}-17\, \lambda_{1}^{6} \lambda_{2}^{12}+ \lambda_{2}^{18}+170\, \lambda_{1}^{12} \lambda_{2}^{3} \lambda_{3}^{3}+1870\, \lambda_{1}^{9} \lambda_{2}^{6} \lambda_{3}^{3}+1870\, \lambda_{1}^{6} \lambda_{2}^{9} \lambda_{3}^{3}+170\, \lambda_{1}^{3} \lambda_{2}^{12} \lambda_{3}^{3}-17\, \lambda_{1}^{12} \lambda_{3}^{6}-1870\, \lambda_{1}^{9} \lambda_{2}^{3} \lambda_{3}^{6}-7854\, \lambda_{1}^{6} \lambda_{2}^{6} \lambda_{3}^{6}-1870\, \lambda_{1}^{3} \lambda_{2}^{9} \lambda_{3}^{6}-17\, \lambda_{2}^{12} \lambda_{3}^{6}+1870\, \lambda_{1}^{6} \lambda_{2}^{3} \lambda_{3}^{9}+1870\, \lambda_{1}^{3} \lambda_{2}^{6} \lambda_{3}^{9}-17\, \lambda_{1}^{6} \lambda_{3}^{12}-170\, \lambda_{1}^{3} \lambda_{2}^{3} \lambda_{3}^{12}-17\, \lambda_{2}^{6} \lambda_{3}^{12}+ \lambda_{3}^{18}+170\, \lambda_{1}^{12} \lambda_{2}^{3} \lambda_{4}^{3}-1870\, \lambda_{1}^{9} \lambda_{2}^{6} \lambda_{4}^{3}+1870\, \lambda_{1}^{6} \lambda_{2}^{9} \lambda_{4}^{3}-170\, \lambda_{1}^{3} \lambda_{2}^{12} \lambda_{4}^{3}+170\, \lambda_{1}^{12} \lambda_{3}^{3} \lambda_{4}^{3}-170\, \lambda_{2}^{12} \lambda_{3}^{3} \lambda_{4}^{3}+1870\, \lambda_{1}^{9} \lambda_{3}^{6} \lambda_{4}^{3}-1870\, \lambda_{2}^{9} \lambda_{3}^{6} \lambda_{4}^{3}+1870\, \lambda_{1}^{6} \lambda_{3}^{9} \lambda_{4}^{3}-1870\, \lambda_{2}^{6} \lambda_{3}^{9} \lambda_{4}^{3}+170\, \lambda_{1}^{3} \lambda_{3}^{12} \lambda_{4}^{3}-170\, \lambda_{2}^{3} \lambda_{3}^{12} \lambda_{4}^{3}-17\, \lambda_{1}^{12} \lambda_{4}^{6}+1870\, \lambda_{1}^{9} \lambda_{2}^{3} \lambda_{4}^{6}-7854\, \lambda_{1}^{6} \lambda_{2}^{6} \lambda_{4}^{6}+1870\, \lambda_{1}^{3} \lambda_{2}^{9} \lambda_{4}^{6}-17\, \lambda_{2}^{12} \lambda_{4}^{6}-1870\, \lambda_{1}^{9} \lambda_{3}^{3} \lambda_{4}^{6}-1870\, \lambda_{2}^{9} \lambda_{3}^{3} \lambda_{4}^{6}-7854\, \lambda_{1}^{6} \lambda_{3}^{6} \lambda_{4}^{6}-7854\, \lambda_{2}^{6} \lambda_{3}^{6} \lambda_{4}^{6}-1870\, \lambda_{1}^{3} \lambda_{3}^{9} \lambda_{4}^{6}-1870\, \lambda_{2}^{3} \lambda_{3}^{9} \lambda_{4}^{6}-17\, \lambda_{3}^{12} \lambda_{4}^{6}+1870\, \lambda_{1}^{6} \lambda_{2}^{3} \lambda_{4}^{9}-1870\, \lambda_{1}^{3} \lambda_{2}^{6} \lambda_{4}^{9}+1870\, \lambda_{1}^{6} \lambda_{3}^{3} \lambda_{4}^{9}-1870\, \lambda_{2}^{6} \lambda_{3}^{3} \lambda_{4}^{9}+1870\, \lambda_{1}^{3} \lambda_{3}^{6} \lambda_{4}^{9}-1870\, \lambda_{2}^{3} \lambda_{3}^{6} \lambda_{4}^{9}-17\, \lambda_{1}^{6} \lambda_{4}^{12}+170\, \lambda_{1}^{3} \lambda_{2}^{3} \lambda_{4}^{12}-17\, \lambda_{2}^{6} \lambda_{4}^{12}-170\, \lambda_{1}^{3} \lambda_{3}^{3} \lambda_{4}^{12}-170\, \lambda_{2}^{3} \lambda_{3}^{3} \lambda_{4}^{12}-17\, \lambda_{3}^{6} \lambda_{4}^{12}+ \lambda_{4}^{18})\left(-117936\right)$

$f_{24}  =(111\, \lambda_{1}^{24}+506\, \lambda_{1}^{18} \lambda_{2}^{6}+10166\, \lambda_{1}^{12} \lambda_{2}^{12}+506\, \lambda_{1}^{6} \lambda_{2}^{18}+111\, \lambda_{2}^{24}-5060\, \lambda_{1}^{18} \lambda_{2}^{3} \lambda_{3}^{3}-206448\, \lambda_{1}^{15} \lambda_{2}^{6} \lambda_{3}^{3}-1118260\, \lambda_{1}^{12} \lambda_{2}^{9} \lambda_{3}^{3}-1118260\, \lambda_{1}^{9} \lambda_{2}^{12} \lambda_{3}^{3}-206448\, \lambda_{1}^{6} \lambda_{2}^{15} \lambda_{3}^{3}-5060\, \lambda_{1}^{3} \lambda_{2}^{18} \lambda_{3}^{3}+506\, \lambda_{1}^{18} \lambda_{3}^{6}+206448\, \lambda_{1}^{15} \lambda_{2}^{3} \lambda_{3}^{6}+4696692\, \lambda_{1}^{12} \lambda_{2}^{6} \lambda_{3}^{6}+12300860\, \lambda_{1}^{9} \lambda_{2}^{9} \lambda_{3}^{6}+4696692\, \lambda_{1}^{6} \lambda_{2}^{12} \lambda_{3}^{6}+206448\, \lambda_{1}^{3} \lambda_{2}^{15} \lambda_{3}^{6}+506\, \lambda_{2}^{18} \lambda_{3}^{6}-1118260\, \lambda_{1}^{12} \lambda_{2}^{3} \lambda_{3}^{9}-12300860\, \lambda_{1}^{9} \lambda_{2}^{6} \lambda_{3}^{9}-12300860\, \lambda_{1}^{6} \lambda_{2}^{9} \lambda_{3}^{9}-1118260\, \lambda_{1}^{3} \lambda_{2}^{12} \lambda_{3}^{9}+10166\, \lambda_{1}^{12} \lambda_{3}^{12}+1118260\, \lambda_{1}^{9} \lambda_{2}^{3} \lambda_{3}^{12}+4696692\, \lambda_{1}^{6} \lambda_{2}^{6} \lambda_{3}^{12}+1118260\, \lambda_{1}^{3} \lambda_{2}^{9} \lambda_{3}^{12}+10166\, \lambda_{2}^{12} \lambda_{3}^{12}-206448\, \lambda_{1}^{6} \lambda_{2}^{3} \lambda_{3}^{15}-206448\, \lambda_{1}^{3} \lambda_{2}^{6} \lambda_{3}^{15}+506\, \lambda_{1}^{6} \lambda_{3}^{18}+5060\, \lambda_{1}^{3} \lambda_{2}^{3} \lambda_{3}^{18}+506\, \lambda_{2}^{6} \lambda_{3}^{18}+111\, \lambda_{3}^{24}-5060\, \lambda_{1}^{18} \lambda_{2}^{3} \lambda_{4}^{3}+206448\, \lambda_{1}^{15} \lambda_{2}^{6} \lambda_{4}^{3}-1118260\, \lambda_{1}^{12} \lambda_{2}^{9} \lambda_{4}^{3}+1118260\, \lambda_{1}^{9} \lambda_{2}^{12} \lambda_{4}^{3}-206448\, \lambda_{1}^{6} \lambda_{2}^{15} \lambda_{4}^{3}+5060\, \lambda_{1}^{3} \lambda_{2}^{18} \lambda_{4}^{3}-5060\, \lambda_{1}^{18} \lambda_{3}^{3} \lambda_{4}^{3}+5060\, \lambda_{2}^{18} \lambda_{3}^{3} \lambda_{4}^{3}-206448\, \lambda_{1}^{15} \lambda_{3}^{6} \lambda_{4}^{3}+206448\, \lambda_{2}^{15} \lambda_{3}^{6} \lambda_{4}^{3}-1118260\, \lambda_{1}^{12} \lambda_{3}^{9} \lambda_{4}^{3}+1118260\, \lambda_{2}^{12} \lambda_{3}^{9} \lambda_{4}^{3}-1118260\, \lambda_{1}^{9} \lambda_{3}^{12} \lambda_{4}^{3}+1118260\, \lambda_{2}^{9} \lambda_{3}^{12} \lambda_{4}^{3}-206448\, \lambda_{1}^{6} \lambda_{3}^{15} \lambda_{4}^{3}+206448\, \lambda_{2}^{6} \lambda_{3}^{15} \lambda_{4}^{3}-5060\, \lambda_{1}^{3} \lambda_{3}^{18} \lambda_{4}^{3}+5060\, \lambda_{2}^{3} \lambda_{3}^{18} \lambda_{4}^{3}+506\, \lambda_{1}^{18} \lambda_{4}^{6}-206448\, \lambda_{1}^{15} \lambda_{2}^{3} \lambda_{4}^{6}+4696692\, \lambda_{1}^{12} \lambda_{2}^{6} \lambda_{4}^{6}-12300860\, \lambda_{1}^{9} \lambda_{2}^{9} \lambda_{4}^{6}+4696692\, \lambda_{1}^{6} \lambda_{2}^{12} \lambda_{4}^{6}-206448\, \lambda_{1}^{3} \lambda_{2}^{15} \lambda_{4}^{6}+506\, \lambda_{2}^{18} \lambda_{4}^{6}+206448\, \lambda_{1}^{15} \lambda_{3}^{3} \lambda_{4}^{6}+206448\, \lambda_{2}^{15} \lambda_{3}^{3} \lambda_{4}^{6}+4696692\, \lambda_{1}^{12} \lambda_{3}^{6} \lambda_{4}^{6}+4696692\, \lambda_{2}^{12} \lambda_{3}^{6} \lambda_{4}^{6}+12300860\, \lambda_{1}^{9} \lambda_{3}^{9} \lambda_{4}^{6}+12300860\, \lambda_{2}^{9} \lambda_{3}^{9} \lambda_{4}^{6}+4696692\, \lambda_{1}^{6} \lambda_{3}^{12} \lambda_{4}^{6}+4696692\, \lambda_{2}^{6} \lambda_{3}^{12} \lambda_{4}^{6}+206448\, \lambda_{1}^{3} \lambda_{3}^{15} \lambda_{4}^{6}+206448\, \lambda_{2}^{3} \lambda_{3}^{15} \lambda_{4}^{6}+506\, \lambda_{3}^{18} \lambda_{4}^{6}-1118260\, \lambda_{1}^{12} \lambda_{2}^{3} \lambda_{4}^{9}+12300860\, \lambda_{1}^{9} \lambda_{2}^{6} \lambda_{4}^{9}-12300860\, \lambda_{1}^{6} \lambda_{2}^{9} \lambda_{4}^{9}+1118260\, \lambda_{1}^{3} \lambda_{2}^{12} \lambda_{4}^{9}-1118260\, \lambda_{1}^{12} \lambda_{3}^{3} \lambda_{4}^{9}+1118260\, \lambda_{2}^{12} \lambda_{3}^{3} \lambda_{4}^{9}-12300860\, \lambda_{1}^{9} \lambda_{3}^{6} \lambda_{4}^{9}+12300860\, \lambda_{2}^{9} \lambda_{3}^{6} \lambda_{4}^{9}-12300860\, \lambda_{1}^{6} \lambda_{3}^{9} \lambda_{4}^{9}+12300860\, \lambda_{2}^{6} \lambda_{3}^{9} \lambda_{4}^{9}-1118260\, \lambda_{1}^{3} \lambda_{3}^{12} \lambda_{4}^{9}+1118260\, \lambda_{2}^{3} \lambda_{3}^{12} \lambda_{4}^{9}+10166\, \lambda_{1}^{12} \lambda_{4}^{12}-1118260\, \lambda_{1}^{9} \lambda_{2}^{3} \lambda_{4}^{12}+4696692\, \lambda_{1}^{6} \lambda_{2}^{6} \lambda_{4}^{12}-1118260\, \lambda_{1}^{3} \lambda_{2}^{9} \lambda_{4}^{12}+10166\, \lambda_{2}^{12} \lambda_{4}^{12}+1118260\, \lambda_{1}^{9} \lambda_{3}^{3} \lambda_{4}^{12}+1118260\, \lambda_{2}^{9} \lambda_{3}^{3} \lambda_{4}^{12}+4696692\, \lambda_{1}^{6} \lambda_{3}^{6} \lambda_{4}^{12}+4696692\, \lambda_{2}^{6} \lambda_{3}^{6} \lambda_{4}^{12}+1118260\, \lambda_{1}^{3} \lambda_{3}^{9} \lambda_{4}^{12}+1118260\, \lambda_{2}^{3} \lambda_{3}^{9} \lambda_{4}^{12}+10166\, \lambda_{3}^{12} \lambda_{4}^{12}-206448\, \lambda_{1}^{6} \lambda_{2}^{3} \lambda_{4}^{15}+206448\, \lambda_{1}^{3} \lambda_{2}^{6} \lambda_{4}^{15}-206448\, \lambda_{1}^{6} \lambda_{3}^{3} \lambda_{4}^{15}+206448\, \lambda_{2}^{6} \lambda_{3}^{3} \lambda_{4}^{15}-206448\, \lambda_{1}^{3} \lambda_{3}^{6} \lambda_{4}^{15}+206448\, \lambda_{2}^{3} \lambda_{3}^{6} \lambda_{4}^{15}+506\, \lambda_{1}^{6} \lambda_{4}^{18}-5060\, \lambda_{1}^{3} \lambda_{2}^{3} \lambda_{4}^{18}+506\, \lambda_{2}^{6} \lambda_{4}^{18}+5060\, \lambda_{1}^{3} \lambda_{3}^{3} \lambda_{4}^{18}+5060\, \lambda_{2}^{3} \lambda_{3}^{3} \lambda_{4}^{18}+506\, \lambda_{3}^{6} \lambda_{4}^{18}+111\, \lambda_{4}^{24})\left(28728\right)$

$f_{30} = (584\, \lambda_{1}^{30}-435\, \lambda_{1}^{24} \lambda_{2}^{6}-63365\, \lambda_{1}^{18} \lambda_{2}^{12}-63365\, \lambda_{1}^{12} \lambda_{2}^{18}-435\, \lambda_{1}^{6} \lambda_{2}^{24}+584\, \lambda_{2}^{30}+4350\, \lambda_{1}^{24} \lambda_{2}^{3} \lambda_{3}^{3}+440220\, \lambda_{1}^{21} \lambda_{2}^{6} \lambda_{3}^{3}+6970150\, \lambda_{1}^{18} \lambda_{2}^{9} \lambda_{3}^{3}+25852920\, \lambda_{1}^{15} \lambda_{2}^{12} \lambda_{3}^{3}+25852920\, \lambda_{1}^{12} \lambda_{2}^{15} \lambda_{3}^{3}+6970150\, \lambda_{1}^{9} \lambda_{2}^{18} \lambda_{3}^{3}+440220\, \lambda_{1}^{6} \lambda_{2}^{21} \lambda_{3}^{3}+4350\, \lambda_{1}^{3} \lambda_{2}^{24} \lambda_{3}^{3}-435\, \lambda_{1}^{24} \lambda_{3}^{6}-440220\, \lambda_{1}^{21} \lambda_{2}^{3} \lambda_{3}^{6}-29274630\, \lambda_{1}^{18} \lambda_{2}^{6} \lambda_{3}^{6}-284382120\, \lambda_{1}^{15} \lambda_{2}^{9} \lambda_{3}^{6}-588153930\, \lambda_{1}^{12} \lambda_{2}^{12} \lambda_{3}^{6}-284382120\, \lambda_{1}^{9} \lambda_{2}^{15} \lambda_{3}^{6}-29274630\, \lambda_{1}^{6} \lambda_{2}^{18} \lambda_{3}^{6}-440220\, \lambda_{1}^{3} \lambda_{2}^{21} \lambda_{3}^{6}-435\, \lambda_{2}^{24} \lambda_{3}^{6}+6970150\, \lambda_{1}^{18} \lambda_{2}^{3} \lambda_{3}^{9}+284382120\, \lambda_{1}^{15} \lambda_{2}^{6} \lambda_{3}^{9}+1540403150\, \lambda_{1}^{12} \lambda_{2}^{9} \lambda_{3}^{9}+1540403150\, \lambda_{1}^{9} \lambda_{2}^{12} \lambda_{3}^{9}+284382120\, \lambda_{1}^{6} \lambda_{2}^{15} \lambda_{3}^{9}+6970150\, \lambda_{1}^{3} \lambda_{2}^{18} \lambda_{3}^{9}-63365\, \lambda_{1}^{18} \lambda_{3}^{12}-25852920\, \lambda_{1}^{15} \lambda_{2}^{3} \lambda_{3}^{12}-588153930\, \lambda_{1}^{12} \lambda_{2}^{6} \lambda_{3}^{12}-1540403150\, \lambda_{1}^{9} \lambda_{2}^{9} \lambda_{3}^{12}-588153930\, \lambda_{1}^{6} \lambda_{2}^{12} \lambda_{3}^{12}-25852920\, \lambda_{1}^{3} \lambda_{2}^{15} \lambda_{3}^{12}-63365\, \lambda_{2}^{18} \lambda_{3}^{12}+25852920\, \lambda_{1}^{12} \lambda_{2}^{3} \lambda_{3}^{15}+284382120\, \lambda_{1}^{9} \lambda_{2}^{6} \lambda_{3}^{15}+284382120\, \lambda_{1}^{6} \lambda_{2}^{9} \lambda_{3}^{15}+25852920\, \lambda_{1}^{3} \lambda_{2}^{12} \lambda_{3}^{15}-63365\, \lambda_{1}^{12} \lambda_{3}^{18}-6970150\, \lambda_{1}^{9} \lambda_{2}^{3} \lambda_{3}^{18}-29274630\, \lambda_{1}^{6} \lambda_{2}^{6} \lambda_{3}^{18}-6970150\, \lambda_{1}^{3} \lambda_{2}^{9} \lambda_{3}^{18}-63365\, \lambda_{2}^{12} \lambda_{3}^{18}+440220\, \lambda_{1}^{6} \lambda_{2}^{3} \lambda_{3}^{21}+440220\, \lambda_{1}^{3} \lambda_{2}^{6} \lambda_{3}^{21}-435\, \lambda_{1}^{6} \lambda_{3}^{24}-4350\, \lambda_{1}^{3} \lambda_{2}^{3} \lambda_{3}^{24}-435\, \lambda_{2}^{6} \lambda_{3}^{24}+584\, \lambda_{3}^{30}+4350\, \lambda_{1}^{24} \lambda_{2}^{3} \lambda_{4}^{3}-440220\, \lambda_{1}^{21} \lambda_{2}^{6} \lambda_{4}^{3}+6970150\, \lambda_{1}^{18} \lambda_{2}^{9} \lambda_{4}^{3}-25852920\, \lambda_{1}^{15} \lambda_{2}^{12} \lambda_{4}^{3}+25852920\, \lambda_{1}^{12} \lambda_{2}^{15} \lambda_{4}^{3}-6970150\, \lambda_{1}^{9} \lambda_{2}^{18} \lambda_{4}^{3}+440220\, \lambda_{1}^{6} \lambda_{2}^{21} \lambda_{4}^{3}-4350\, \lambda_{1}^{3} \lambda_{2}^{24} \lambda_{4}^{3}+4350\, \lambda_{1}^{24} \lambda_{3}^{3} \lambda_{4}^{3}-4350\, \lambda_{2}^{24} \lambda_{3}^{3} \lambda_{4}^{3}+440220\, \lambda_{1}^{21} \lambda_{3}^{6} \lambda_{4}^{3}-440220\, \lambda_{2}^{21} \lambda_{3}^{6} \lambda_{4}^{3}+6970150\, \lambda_{1}^{18} \lambda_{3}^{9} \lambda_{4}^{3}-6970150\, \lambda_{2}^{18} \lambda_{3}^{9} \lambda_{4}^{3}+25852920\, \lambda_{1}^{15} \lambda_{3}^{12} \lambda_{4}^{3}-25852920\, \lambda_{2}^{15} \lambda_{3}^{12} \lambda_{4}^{3}+25852920\, \lambda_{1}^{12} \lambda_{3}^{15} \lambda_{4}^{3}-25852920\, \lambda_{2}^{12} \lambda_{3}^{15} \lambda_{4}^{3}+6970150\, \lambda_{1}^{9} \lambda_{3}^{18} \lambda_{4}^{3}-6970150\, \lambda_{2}^{9} \lambda_{3}^{18} \lambda_{4}^{3}+440220\, \lambda_{1}^{6} \lambda_{3}^{21} \lambda_{4}^{3}-440220\, \lambda_{2}^{6} \lambda_{3}^{21} \lambda_{4}^{3}+4350\, \lambda_{1}^{3} \lambda_{3}^{24} \lambda_{4}^{3}-4350\, \lambda_{2}^{3} \lambda_{3}^{24} \lambda_{4}^{3}-435\, \lambda_{1}^{24} \lambda_{4}^{6}+440220\, \lambda_{1}^{21} \lambda_{2}^{3} \lambda_{4}^{6}-29274630\, \lambda_{1}^{18} \lambda_{2}^{6} \lambda_{4}^{6}+284382120\, \lambda_{1}^{15} \lambda_{2}^{9} \lambda_{4}^{6}-588153930\, \lambda_{1}^{12} \lambda_{2}^{12} \lambda_{4}^{6}+284382120\, \lambda_{1}^{9} \lambda_{2}^{15} \lambda_{4}^{6}-29274630\, \lambda_{1}^{6} \lambda_{2}^{18} \lambda_{4}^{6}+440220\, \lambda_{1}^{3} \lambda_{2}^{21} \lambda_{4}^{6}-435\, \lambda_{2}^{24} \lambda_{4}^{6}-440220\, \lambda_{1}^{21} \lambda_{3}^{3} \lambda_{4}^{6}-440220\, \lambda_{2}^{21} \lambda_{3}^{3} \lambda_{4}^{6}-29274630\, \lambda_{1}^{18} \lambda_{3}^{6} \lambda_{4}^{6}-29274630\, \lambda_{2}^{18} \lambda_{3}^{6} \lambda_{4}^{6}-284382120\, \lambda_{1}^{15} \lambda_{3}^{9} \lambda_{4}^{6}-284382120\, \lambda_{2}^{15} \lambda_{3}^{9} \lambda_{4}^{6}-588153930\, \lambda_{1}^{12} \lambda_{3}^{12} \lambda_{4}^{6}-588153930\, \lambda_{2}^{12} \lambda_{3}^{12} \lambda_{4}^{6}-284382120\, \lambda_{1}^{9} \lambda_{3}^{15} \lambda_{4}^{6}-284382120\, \lambda_{2}^{9} \lambda_{3}^{15} \lambda_{4}^{6}-29274630\, \lambda_{1}^{6} \lambda_{3}^{18} \lambda_{4}^{6}-29274630\, \lambda_{2}^{6} \lambda_{3}^{18} \lambda_{4}^{6}-440220\, \lambda_{1}^{3} \lambda_{3}^{21} \lambda_{4}^{6}-440220\, \lambda_{2}^{3} \lambda_{3}^{21} \lambda_{4}^{6}-435\, \lambda_{3}^{24} \lambda_{4}^{6}+6970150\, \lambda_{1}^{18} \lambda_{2}^{3} \lambda_{4}^{9}-284382120\, \lambda_{1}^{15} \lambda_{2}^{6} \lambda_{4}^{9}+1540403150\, \lambda_{1}^{12} \lambda_{2}^{9} \lambda_{4}^{9}-1540403150\, \lambda_{1}^{9} \lambda_{2}^{12} \lambda_{4}^{9}+284382120\, \lambda_{1}^{6} \lambda_{2}^{15} \lambda_{4}^{9}-6970150\, \lambda_{1}^{3} \lambda_{2}^{18} \lambda_{4}^{9}+6970150\, \lambda_{1}^{18} \lambda_{3}^{3} \lambda_{4}^{9}-6970150\, \lambda_{2}^{18} \lambda_{3}^{3} \lambda_{4}^{9}+284382120\, \lambda_{1}^{15} \lambda_{3}^{6} \lambda_{4}^{9}-284382120\, \lambda_{2}^{15} \lambda_{3}^{6} \lambda_{4}^{9}+1540403150\, \lambda_{1}^{12} \lambda_{3}^{9} \lambda_{4}^{9}-1540403150\, \lambda_{2}^{12} \lambda_{3}^{9} \lambda_{4}^{9}+1540403150\, \lambda_{1}^{9} \lambda_{3}^{12} \lambda_{4}^{9}-1540403150\, \lambda_{2}^{9} \lambda_{3}^{12} \lambda_{4}^{9}+284382120\, \lambda_{1}^{6} \lambda_{3}^{15} \lambda_{4}^{9}-284382120\, \lambda_{2}^{6} \lambda_{3}^{15} \lambda_{4}^{9}+6970150\, \lambda_{1}^{3} \lambda_{3}^{18} \lambda_{4}^{9}-6970150\, \lambda_{2}^{3} \lambda_{3}^{18} \lambda_{4}^{9}-63365\, \lambda_{1}^{18} \lambda_{4}^{12}+25852920\, \lambda_{1}^{15} \lambda_{2}^{3} \lambda_{4}^{12}-588153930\, \lambda_{1}^{12} \lambda_{2}^{6} \lambda_{4}^{12}+1540403150\, \lambda_{1}^{9} \lambda_{2}^{9} \lambda_{4}^{12}-588153930\, \lambda_{1}^{6} \lambda_{2}^{12} \lambda_{4}^{12}+25852920\, \lambda_{1}^{3} \lambda_{2}^{15} \lambda_{4}^{12}-63365\, \lambda_{2}^{18} \lambda_{4}^{12}-25852920\, \lambda_{1}^{15} \lambda_{3}^{3} \lambda_{4}^{12}-25852920\, \lambda_{2}^{15} \lambda_{3}^{3} \lambda_{4}^{12}-588153930\, \lambda_{1}^{12} \lambda_{3}^{6} \lambda_{4}^{12}-588153930\, \lambda_{2}^{12} \lambda_{3}^{6} \lambda_{4}^{12}-1540403150\, \lambda_{1}^{9} \lambda_{3}^{9} \lambda_{4}^{12}-1540403150\, \lambda_{2}^{9} \lambda_{3}^{9} \lambda_{4}^{12}-588153930\, \lambda_{1}^{6} \lambda_{3}^{12} \lambda_{4}^{12}-588153930\, \lambda_{2}^{6} \lambda_{3}^{12} \lambda_{4}^{12}-25852920\, \lambda_{1}^{3} \lambda_{3}^{15} \lambda_{4}^{12}-25852920\, \lambda_{2}^{3} \lambda_{3}^{15} \lambda_{4}^{12}-63365\, \lambda_{3}^{18} \lambda_{4}^{12}+25852920\, \lambda_{1}^{12} \lambda_{2}^{3} \lambda_{4}^{15}-284382120\, \lambda_{1}^{9} \lambda_{2}^{6} \lambda_{4}^{15}+284382120\, \lambda_{1}^{6} \lambda_{2}^{9} \lambda_{4}^{15}-25852920\, \lambda_{1}^{3} \lambda_{2}^{12} \lambda_{4}^{15}+25852920\, \lambda_{1}^{12} \lambda_{3}^{3} \lambda_{4}^{15}-25852920\, \lambda_{2}^{12} \lambda_{3}^{3} \lambda_{4}^{15}+284382120\, \lambda_{1}^{9} \lambda_{3}^{6} \lambda_{4}^{15}-284382120\, \lambda_{2}^{9} \lambda_{3}^{6} \lambda_{4}^{15}+284382120\, \lambda_{1}^{6} \lambda_{3}^{9} \lambda_{4}^{15}-284382120\, \lambda_{2}^{6} \lambda_{3}^{9} \lambda_{4}^{15}+25852920\, \lambda_{1}^{3} \lambda_{3}^{12} \lambda_{4}^{15}-25852920\, \lambda_{2}^{3} \lambda_{3}^{12} \lambda_{4}^{15}-63365\, \lambda_{1}^{12} \lambda_{4}^{18}+6970150\, \lambda_{1}^{9} \lambda_{2}^{3} \lambda_{4}^{18}-29274630\, \lambda_{1}^{6} \lambda_{2}^{6} \lambda_{4}^{18}+6970150\, \lambda_{1}^{3} \lambda_{2}^{9} \lambda_{4}^{18}-63365\, \lambda_{2}^{12} \lambda_{4}^{18}-6970150\, \lambda_{1}^{9} \lambda_{3}^{3} \lambda_{4}^{18}-6970150\, \lambda_{2}^{9} \lambda_{3}^{3} \lambda_{4}^{18}-29274630\, \lambda_{1}^{6} \lambda_{3}^{6} \lambda_{4}^{18}-29274630\, \lambda_{2}^{6} \lambda_{3}^{6} \lambda_{4}^{18}-6970150\, \lambda_{1}^{3} \lambda_{3}^{9} \lambda_{4}^{18}-6970150\, \lambda_{2}^{3} \lambda_{3}^{9} \lambda_{4}^{18}-63365\, \lambda_{3}^{12} \lambda_{4}^{18}+440220\, \lambda_{1}^{6} \lambda_{2}^{3} \lambda_{4}^{21}-440220\, \lambda_{1}^{3} \lambda_{2}^{6} \lambda_{4}^{21}+440220\, \lambda_{1}^{6} \lambda_{3}^{3} \lambda_{4}^{21}-440220\, \lambda_{2}^{6} \lambda_{3}^{3} \lambda_{4}^{21}+440220\, \lambda_{1}^{3} \lambda_{3}^{6} \lambda_{4}^{21}-440220\, \lambda_{2}^{3} \lambda_{3}^{6} \lambda_{4}^{21}-435\, \lambda_{1}^{6} \lambda_{4}^{24}+4350\, \lambda_{1}^{3} \lambda_{2}^{3} \lambda_{4}^{24}-435\, \lambda_{2}^{6} \lambda_{4}^{24}-4350\, \lambda_{1}^{3} \lambda_{3}^{3} \lambda_{4}^{24}-4350\, \lambda_{2}^{3} \lambda_{3}^{3} \lambda_{4}^{24}-435\, \lambda_{3}^{6} \lambda_{4}^{24}+584\, \lambda_{4}^{30})\left(-147420\right)$

\subsection*{Family two} Basic semi-simple form $p = \lambda_1p_1 + \lambda_2p_2 - \lambda_3p_3$.  We randomize the $\lambda$'s and add the different nilpotent parts $e$: 

Here are the corresponding adjoint rank profiles:

\textnumero 3 $e=0$:
$\left(
\right)$                     
       
\subsection*{Family three} Basic semi-simple form $p = \lambda_1p_1 + \lambda_2p_2 $. 

\textnumero 1 $e = e_{2}e_{6}e_{7}+e_{1}e_{6}e_{8}+e_{2}e_{4}e_{9}+e_{1}e_{5}e_{9}$:
$\left(%
\right)$

\subsection*{Family four}
$e_{2}e_{3}e_{6}+e_{0}e_{5}e_{6}+e_{1}e_{4}e_{7}+e_{0}e_{3}e_{8}$
$\left(%
\right)$

\subsection*{Family five}

 $e_{0}e_{1}e_{2}+e_{3}e_{4}e_{5}+e_{2}e_{3}e_{6}+e_{0}e_{5}e_{6}+e_{1}e_{4}e_{7}+e_{0}e_{3}e_{8}$
$\left(%
\right)$

\subsection*{Family six}
                            
$e_{2}e_{3}e_{6}+e_{1}e_{5}e_{6}+e_{1}e_{4}e_{7}+e_{0}e_{5}e_{7}+e_{1}e_{3}e_{8}+e_{0}e_{4}e_{8}$
$\left(%
\right)$

Here are the rank profiles of the nilpotent elements:

1:
$e_{3}e_{4}e_{5}+e_{2}e_{4}e_{6}+e_{1}e_{5}e_{6}+e_{2}e_{3}e_{7}+e_{1}e_{4}e_{7}+e_{0}e_{5}e_{7}+e_{1}e_{2}e_{8}+e_{0}e_{3}e_{8}$
$\left(%
\right)$

\end{document}